\documentclass[author-year, review]{elsarticle}
\usepackage[margin=2.195cm]{geometry}

\usepackage{lineno,hyperref}
\modulolinenumbers[1]

\journal{European Journal of Operational Research}
 
\usepackage{setspace}
\usepackage{amssymb}
\usepackage{amsmath}
\usepackage{latexsym}
\usepackage{mathabx}
\usepackage{amsthm}
\usepackage{epigraph}
\usepackage{booktabs}
\usepackage{braket}
\usepackage{bbm}
\usepackage{appendix}
\usepackage{wrapfig}
\usepackage{graphicx}
\usepackage{array,multirow}
\usepackage{float}
\usepackage{pdflscape}
\usepackage{epstopdf}
\usepackage{pgfplots}
\usepackage{longtable}
\usepackage{soul}
\soulregister\cite7
\soulregister\ref7
\soulregister\eqref7
\soulregister\pageref7
\usepackage{cancel}
\usetikzlibrary{decorations.markings}
\DeclareMathOperator*{\argmax}{arg\,max}

\newcommand{\norm}[1]{\left\lVert#1\right\rVert}
\newcommand{\virg}[1]{\text{``#1''}}
\newcommand{\abs}[1]{\left\lvert#1\right\rvert}

\newcommand\RotText[1]{\fontsize{9}{9}\selectfont
  \rotatebox[origin=c]{90}{\parbox{3.5cm}{%
\centering#1}}}
\usepackage{pifont}
\newcommand{\cmark}{\ding{51}}%
\theoremstyle{definition}
\newtheorem{theorem}{Theorem}

\newtheorem{lemma}{Lemma}
\newtheorem{proposition}{Proposition}
\newtheorem{corollary}{Corollary}

\usepackage[font={footnotesize}]{caption}
\usepackage[font={scriptsize}]{subcaption}

\usepackage{algorithm,algorithmic}
\makeatletter
\renewcommand{\ALG@name}{Pseudocode}
\makeatother

\usepackage{fancyhdr}
\usepackage{numcompress}\biboptions{authoryear}

\makeatletter
\def\ps@pprintTitle{%
 \let\@oddhead\@empty
 \let\@evenhead\@empty
 \def\@oddfoot{}%
 \let\@evenfoot\@oddfoot}
\makeatother

\makeatletter
\newcommand{\mathleft}{\@fleqntrue\@mathmargin0pt}
\newcommand{\mathcenter}{\@fleqnfalse}
\makeatother

\begin{document}

\begin{frontmatter}

\title{A Novel Robust Optimization Model for \\Nonlinear Support Vector Machine}

\author[Maggioni_address]{Francesca Maggioni\corref{mycorrespondingauthor}}
\cortext[mycorrespondingauthor]{Corresponding author}
\ead{francesca.maggioni@unibg.it}

\author[Maggioni_address]{Andrea Spinelli}

\address[Maggioni_address]{Department of Management, Information and Production Engineering, University of Bergamo, Viale G. Marconi 5, Dalmine 24044, Italy}

\begin{abstract}
In this paper, we present new optimization models for Support Vector Machine (SVM), with the aim of separating data points in two or more classes. The classification task is handled by means of nonlinear classifiers induced by kernel functions and consists in two consecutive phases: first, a classical SVM model is solved, followed by a linear search procedure, aimed at minimizing the total number of misclassified data points. To address the problem of data perturbations and protect the model against uncertainty, we construct bounded-by-norm uncertainty sets around each training data and apply robust optimization techniques. We rigorously derive the robust counterpart extension of the deterministic SVM approach, providing computationally tractable reformulations. Closed-form expressions for the bounds of the uncertainty sets in the feature space have been formulated for typically used kernel functions. Finally, extensive numerical results on real-world datasets show the benefits of the proposed robust approach in comparison with various SVM alternatives in the machine learning literature.
\end{abstract}


\end{frontmatter}

\nolinenumbers
\thispagestyle{fancy}

\section{Introduction}
\emph{Support Vector Machine} (SVM) is one of the main supervised \emph{Machine Learning} (ML) techniques commonly deployed for classification and regression purposes. Within the \emph{Operational Research} (OR) domain, supervised ML methods are designed to support better decision-making and solve hard optimization problems (\cite{GamGhaSaw2021}). To this end, a plethora of methodologies have been devised and applied to various OR fields (\cite{DeBocketal2023}). In particular, combinatorial optimization (\cite{BenLodPro2021,WeiHaoRenGlov2023}), customer churn prediction (\cite{ChenFanSun2012,MalLopVai2020,Ben-PenBlanCarRam-Cob2024,SzeSlo2024}), banking (\cite{YaoCroAnd2017,DouZopGouPlaZha2023,KatLelPyrAndFer2024}) and maritime industry (\cite{MiWangetal2019,RaeSahMan2023}).

Currently, deep learning algorithms are adopted whenever classical ML methods fail to capture complex relationships between input data both for classification and regression tasks (\cite{GamGhaSaw2021}). Nevertheless, the advantage of mathematical programming approaches to model deep neural networks has been explored only for small-sized datasets, and without a guarantee on the effectiveness of the performance (\cite{GunSepBaeOskLem2021}). For this reason, the investigation of novel ML techniques is a relevant ongoing research issue (\cite{MalLopCar2022}).

Introduced in \cite{VapChe1974}, SVM has outperformed most other ML systems, due to its simplicity and better performance. Therefore, it has been applied in many practical research fields, such as finance (\cite{TayCao2001,LuoYanTian2020}), chemistry (\cite{LiLiXu2009,MarDeLNic2020}), medicine (\cite{WanZheWoKo2018,MagFacGheManBonORAHS2022}), and vehicles smog rating classification (\cite{DeLMagSpi2024,MagSpi2024}), to name a few.

\emph{Hard Margin}-SVM (HM-SVM) is the original approach formulated in \cite{VapChe1974}, consisting in finding a hyperplane classifying observations into two classes, such that the margin, i.e. the $\ell_2$-distance from the hyperplane to the nearest point of each class, is maximized. The underlying hypothesis of the HM-SVM is that training data can always be linearly separated, such that no observation is misclassified. To overcome the assumption of linear separability, in \cite{CorVap1995} the \emph{Soft Margin}-SVM (SM-SVM) is proposed. In this case, the optimal hyperplane seeks a trade-off between the maximization of the margin and the minimization of the training error of misclassification.

In order to improve the accuracy of the method, several SVM variants have been devised in the literature. Specifically, in this paper we focus our attention on the one presented in \cite{LiuPot2009}. The advantages of this technique over other SVM approaches are mainly due to a two-step procedure. Indeed, rather than considering a single hyperplane, training data are firstly separated by means of two parallel hyperplanes as solutions of a SM-SVM model. The final optimal hyperplane is then searched in the strip between them, such that the total number of misclassified points is minimized. Compared to classical SM-SVM, numerical experiments show that this formulation achieves higher levels of computational accuracy.

Nevertheless, training observations may not be always separable by means of hyperplanes and, even with \emph{ad hoc} variants of linear SVM, the misclassification error may be significant. In \cite{BosGuyVap1992}, the extension of the linear HM-SVM model is introduced, by considering nonlinear transformation of the data. According to this technique, kernel functions are used to embed data points onto a higher-dimensional space (the so-called \emph{feature space}), without increasing the computational complexity of the problem. Several variants of this methodology have been proposed in the ML literature (see for example \cite{BenMan1992,Man1998,SchSmoWilBar2000,Yaj2005,JayKheCha2007,Hao2010,Peng2011,DingHua2014,DingZhaZhaZhaXue2019,BlaPueRod-Chi2020,CerGarRodAsd2020,DuZhaCheSunCheShao2021,GaoFangLuoMed2021,Jim-CorMorPin2021}).

For the methods mentioned above, all data points are implicitly assumed to be known exactly. However, in real-world observations this condition may not be always true. Indeed, measurement errors during data collection, random perturbations, presence of noise and other forms of uncertainty may corrupt the quality of input values, resulting in worsening performance of the classification process. In recent years, different techniques have been investigated with the aim of facing uncertainty in ML methods. Among them, \emph{Robust Optimization} (RO) is recognized as one of the main paradigms to protect optimization models against uncertainty (see for example \cite{Ben-TalElGNem2009,XuCarMan2009,BerBroCar2010}). RO assumes that all possible realizations of the uncertain parameter belong to a prescribed uncertainty set. The corresponding robust model is then derived by optimizing against the worst-case realization of the parameter across the entire uncertainty set (\cite{BerDunPawZhu2019}). The application of RO strategies typically results in higher predictiveness (\cite{MalLopVai2020,FacMagPot2022}). For this reason, it is worth designing novel RO models with the aim of improving the accuracy of the classification process.

In this paper, we present novel SVM models aiming at separating classes of data points. The formulation extends the approach of \cite{LiuPot2009} to the context of multiclass and nonlinear classification. In order to protect the model against perturbations, we introduce bounded-by-norm uncertainty sets around each training observation and rigorously derive the robust counterpart of the deterministic approach, providing computationally tractable reformulations. In addition, our proposal represents a valid contribution to the state of the art on SVM thanks to the computation of the uncertainty set bounds in the feature space as function of the bounds in the input space. This is a novel development in the ML domain.

The main contributions of the paper are four-fold and can be summarized as follows:
\begin{itemize}
\item To extend the binary linear SVM approach of \cite{LiuPot2009} to the cases of multiclass and nonlinear classification;
\item To formulate the robust extension of the SVM model with nonlinear classifiers using bounded-by-$\ell_p$-norm uncertainty sets and provide computationally tractable reformulations;
\item To rigorously derive bounds on the radii of the uncertainty sets in the feature space for some of the most used kernel functions in the ML literature;
\item To provide extensive numerical experiments based on real-world datasets with the aim of evaluating the performance of the proposed models and comparing the results with extant SVM methods in the literature.
\end{itemize}

The remainder of the paper is organized as follows. Section \ref{sec_literature_review} reviews the existing literature on the problem. In Section \ref{sec_background_notation}, the notation is introduced, along with a brief discussion on related SVM-type problems. In Section \ref{sec_novel_deterministic_model}, the novel deterministic model with nonlinear classifier is introduced for both binary and multiclass classification. Section \ref{sec_robust_model} considers the robust extension together with the construction of the uncertainty sets. In Section \ref{sec_computational_results}, the computational results are shown. Finally, Section \ref{sec_conclusions} concludes the paper and discusses future works.

\section{Literature review} \label{sec_literature_review}
The nonlinear SVM approach presented in \cite{BosGuyVap1992} has been explored in several works, leading to alternative formulations. In \cite{Man1998,LeeManWol2000} a kernel-induced decision boundary is derived by considering quadratic and piecewise-linear objective function, resulting in a convex model. In \cite{SchSmoWilBar2000} the formulation of \emph{$\nu$-Support Vector Classification} ($\nu$-SVC) is proposed for both linear and nonlinear classifiers. This algorithm differs from the classical SVM paradigm of \cite{Vap1995} since it involves a new parameter $\nu$ in the objective function, controlling the number of support vectors. In \cite{JayKheCha2007} the \emph{TWin Support Vector Machine} (TWSVM) is designed. Contrary to standard SVM, TWSVM determines a pair of nonparallel hyperplanes by solving two small-sized SVM-type problems. TWSVM is combined in \cite{Peng2011} with a flexible parametric margin model (\cite{Hao2010}), deriving the \emph{Twin Parametric Margin Support Vector Machine} (TPMSVM). Recently, in \cite{BlaPueRod-Chi2020} the classical $\ell_2$-norm problem has been extended to the general case of $\ell_p$-norm with $p>1$, resulting in a \textit{Second-Order Cone Programming} formulation (SOCP, \cite{MagPotBerAll2009}). Within the field of \textit{Double Well Potential} functions (DWP), a kernel-free DWP model for SVM is derived in \cite{GaoFangLuoMed2021} for classifying nonlinearly separable data. The problem of feature selection in nonlinear SVM is explored in \cite{Jim-CorMorPin2021}, where a method based on a min-max optimization model is proposed. With respect to the extant literature on nonlinear SVM, the first contribution of this work is the extension of the linear SVM variant developed in \cite{LiuPot2009} to the case of nonlinear classifiers. The model benefits from such extension since it handles cases of nonlinearly separable data with a low misclassification error.

In order to prevent low accuracies in the classification process when training data are plagued by uncertainty, RO techniques are applied in the SVM context (\cite{WanPar2014}). In \cite{Bhat2004} hyperellipsoids around data points are considered, and the robust model results in a SOCP problem. A tractable robust counterpart of the linear SM-SVM approach is derived in \cite{BerDunPawZhu2019}. The authors robustify the model by considering additive and bounded-by-norm perturbations in the training data. In \cite{ElGLan2003} the binary classification problem under feature uncertainty is formulated with uncertainty sets in the form of hyperrectangles and hyperellipsoids around input data. The same choices of uncertainty sets is made in \cite{FacMagPot2022}, where the RO extension of the linear SVM variant presented in \cite{LiuPot2009} is proposed. In this work, we further extend such approach by formulating a robust SVM model tailored for a general class of bounded-by-$\ell_p$-norm uncertainty sets. This improves the generalization capability of the model as the choice of the $\ell_p$-norm can be made according to the information available on the training dataset and the desired degree of conservatism.

As far as it concerns RO techniques applied to nonlinear SVM, various approaches exist in the literature. In \cite{BhaBhaBen-Tal2010,Ben-TalBhaNem2012} the kernel matrix is assumed to be affected by uncertainty, due to feature perturbations in the input data. Such matrix is decomposed as a linear combination of positive semidefinite matrices with bounded-by-$\ell_p$-norm coefficients. The main limitation of this approach is that the functional form of the matrices in the combination is typically unknown. Thus, it is not obvious how to characterize the elements in the uncertainty set, unless by using a sampling procedure. In \cite{BiZha2005,TraGil2006} training data points are subject to uncertain but bounded-by-$\ell_p$-norm perturbations. Robustified models are derived for both linear and nonlinear classifiers. A related work on bounded-by-norm uncertainty sets is \cite{XuCarMan2009}, where a link between regularization and robustness is provided. In \cite{TraAlw2010} the stability of SVM models with bounded perturbations is investigated by using discriminant functions. Polyhedral uncertainty sets are considered in \cite{FunManSha2002,JuTian2012,FanSadPar2014}, based on the nonlinear classifier proposed in \cite{Man1998}. In all these works on robust SVM with nonlinear classifiers, only the case with Gaussian kernel has been investigated. Indeed, for such kernel there exists a closed-form expression for the radius of the uncertainty set in the feature space based on the corresponding one in the input space (see \cite{XuCarMan2009}). In this paper, we prove further theoretical results valid for other classes of kernels, i.e. homogeneous and inhomogeneous polynomial kernels. These findings are beneficial for all robust SVM models with bounded-by-$\ell_p$-norm uncertainty sets and kernel-induced classifiers.

RO techniques are also applied to variants of the classical SVM model. In \cite{PengXu2013} a robust TWSVM classifier is proposed, by including uncertainty in the variance matrices of the two classes. In \cite{QiTianShi2013} the robust extension of TWSVM is derived. For the nonlinear case, only Gaussian kernel and ellipsoidal uncertainty sets are considered, resulting in a SOCP formulation. In \cite{DelMagSpi2023TPMSVM} the robust and multiclass extension of the TPMSVM is provided. A complete survey on recent developments on TWSVM models can be found in \cite{TanRajRasShaoGan2022}.

When partial or complete information on the probability distribution of the training data are available, other solution techniques dealing with uncertainties such as \emph{Chance-Constrained Programming} (CCP) and \emph{Distributionally Robust Optimization} (DRO) have been considered in the SVM literature (\cite{Ket2024,JiaPen2024}). The \emph{Minimax Probability Machine} (MPM) is the first distributionally robust SVM approach that minimizes the worst-case probability of misclassification (\cite{LanGhaBhaJor2002}). In \cite{MalLopVai2020} the MPMs are extended and applied to the robust profit-driven churn prediction. Within the MPM framework, the use of Cobb-Douglas function for maximizing the expected class accuracies under a worst-case distribution setting is proposed in \cite{MalLopCar2022}. As far as it concerns DRO methods applied to SVM, we mention the recent work of \cite{FacMagPot2022} where a moment-based distributionally robust extension of the \cite{LiuPot2009} formulation is designed. The problem of robust feature selection with CCP is explored in \cite{LopMalCar2018} by using difference of convex functions. Within the multiclass context, in \cite{LopMalCar2017} a robust CCP formulation for multiclass classification via TWSVM is proposed. Finally, a combination of CCP and DRO techniques applied to linear and nonlinear SVM models with uncertain data is explored in \cite{WanFanPar2018,Kha-ShiBab-AzaHos-NodPar2023} and in \cite{LinFangFangGao2024,LinFangFangGaoLuo2024}, respectively.

All the approaches discussed so far are listed in Table \ref{tab_SVM_literature_review}. For a comprehensive review of RO techniques applied to SVM models the reader is referred to \cite{SinGhoShu2020}.

In summary, the contributions of this paper differ from the literature described above in several aspects. First of all, we present a novel optimization model with nonlinear classifiers, extending the approach of \cite{LiuPot2009}, both in the case of binary and multiclass classification. Secondly, we consider general bounded-by-$\ell_p$-norm uncertainty sets around training observations. This increases the flexibility of the model, adapting the formulation to more complex perturbations in input data. In addition, it results in a generalization of the robust approach of \cite{FacMagPot2022} where only box and ellipsoidal uncertainty sets have been considered. Thirdly, we derive closed-form expressions of the bounds in the feature space for some of typically used kernel functions in ML literature. Finally, we deduce the robust counterpart of the deterministic formulations, protecting the models against data uncertainty.

\begin{table}[h!]
\centering
\resizebox{\textwidth}{!}{
    \begin{tabular}{@{} c l *{37}c}
        & & \RotText{\small\cite{VapChe1974}} & \RotText{\small\cite{BosGuyVap1992}} & \RotText{\small\cite{Vap1995}} & \RotText{\small\cite{Man1998}} & \RotText{\small\cite{LeeManWol2000}} & \RotText{\small\cite{SchSmoWilBar2000}} & \RotText{\small\cite{FunManSha2002}} & \RotText{\small\cite{LanGhaBhaJor2002}} & \RotText{\small\cite{ElGLan2003}} & \RotText{\small\cite{Bhat2004}} & \RotText{\small\cite{BiZha2005}} & \RotText{\footnotesize\cite{TraGil2006}} & \RotText{\small\cite{JayKheCha2007}} & \RotText{\small\cite{LiuPot2009}} & \RotText{\small\cite{XuCarMan2009}} & \RotText{\small\cite{BhaBhaBen-Tal2010}} & \RotText{\footnotesize\cite{TraAlw2010}} & \RotText{\small\cite{Peng2011}} & \RotText{\small\cite{Ben-TalBhaNem2012}} & \RotText{\small\cite{JuTian2012}} & \RotText{\small\cite{PengXu2013}} & \RotText{\small\cite{QiTianShi2013}} & \RotText{\small\cite{FanSadPar2014}} & \RotText{\small\cite{LopMalCar2017}} & \RotText{\small\cite{LopMalCar2018}} & \RotText{\small\cite{WanFanPar2018}} & \RotText{\small\cite{BerDunPawZhu2019}} & \RotText{\small\cite{BlaPueRod-Chi2020}} & \RotText{\small\cite{MalLopVai2020}} & \RotText{\small{\cite{GaoFangLuoMed2021}}} & \RotText{\small\cite{Jim-CorMorPin2021}} & \RotText{\small\cite{FacMagPot2022}} & \RotText{\small\cite{MalLopCar2022}} & \RotText{\small\cite{DelMagSpi2023TPMSVM}} & \RotText{\small\cite{Kha-ShiBab-AzaHos-NodPar2023}} & \RotText{\small{\cite{LinFangFangGao2024}}} & \RotText{\small{\cite{LinFangFangGaoLuo2024}}}
        \\
       \toprule
\multirow{2}{*}{{\small{SVM}}} & {\small{Linear classifier}} & \cmark & & \cmark & & & \cmark & \cmark & \cmark & \cmark & \cmark & \cmark & \cmark &  \cmark & \cmark & \cmark & & \cmark & \cmark & & \cmark & \cmark & \cmark & \cmark & \cmark & \cmark & \cmark & \cmark & \cmark & \cmark & & & \cmark & \cmark & \cmark & \cmark
\\
\cmidrule{2-39}
& {\small{Nonlinear classifier}} & & \cmark & \cmark & \cmark & \cmark & \cmark & \cmark & \cmark & & & \cmark & \cmark & \cmark & & \cmark & \cmark & \cmark & \cmark & \cmark & \cmark & \cmark & \cmark & \cmark & \cmark & \cmark & & & \cmark & \cmark & {\cmark} & \cmark & & \cmark & \cmark & & {\cmark} & {\cmark}
\\
\cmidrule[0.55pt]{1-39}
 & {\small{Box RO}} &&&&&&&& & \cmark &&&&&&&&&&&&&&&&&&&&&&& \cmark
\\
\cmidrule{2-39}
{\small{Type of}} & {\small{Ellipsoidal RO}} &&&&&&&&\cmark & \cmark & \cmark &&&&&&&&&&&&& \cmark &&&&&&&&& \cmark
\\
\cmidrule{2-39}
{\small{Robust}} & {\small{Polyhedral RO}}  &&&&&&& \cmark &&&&&&&&&&&&& \cmark & & & \cmark
\\
\cmidrule{2-39}
{\small{Methodology}} & {\small{Bounded-by-norm RO}} &&&&&&&& &  & & \cmark & \cmark & & & \cmark & & \cmark & & \cmark &&&&&&&& \cmark &&&&&&& \cmark
\\
\cmidrule{2-39}
& {\small{Matrix RO}} &&&&&&&& &&&&&&&& \cmark & & & \cmark & & \cmark
\\
\cmidrule{2-39}
& {\small{Chance-Constrained}} &&&&&&&& \cmark &&&&&&&&\cmark &&&&&&&& \cmark & \cmark & \cmark & & & \cmark & & & \cmark & & & \cmark & {\cmark} & {\cmark}
\\
\cmidrule{2-39}
& {\small{Distributionally RO}} &&&&&&&& \cmark &&&&&&&&&&&&&&&&&&\cmark & &&&& \cmark & & & & \cmark & {\cmark} & {\cmark}
\\
        \bottomrule
    \end{tabular}}
    \caption{A selected SVM literature review. In the first row of the table the methodological contributions are listed in chronological order. Second and third rows specify the type of SVM classifier (linear or nonlinear). Finally, the optimization under uncertainty methodologies employed in the articles are explored in rows four to ten.} \label{tab_SVM_literature_review}
\end{table}

\normalsize

\section{Background and notation} \label{sec_background_notation}
In this section, we report the notation (Section \ref{sec_notation}) and briefly recall the methods that are relevant for our proposal (Section \ref{sec_review_methods}).

\subsection{Notation} \label{sec_notation}
In the following, the set of nonnegative real numbers will be denoted by $\mathbb{R}^{+}$, whereas if zero is excluded we write $\mathbb{R}^{+}_0$. Hereinafter, all vectors will be column vectors, unless transposition by the superscript ``$^\top$''. If $a$ is a vector in $\mathbb{R}^n$, then its $i$-th component will be denoted by $a_i$. The scalar product in a inner product space $\mathcal{H}$ will be denoted by $\langle \cdot,\cdot\rangle$. If $\mathcal{H}=\mathbb{R}^n$ and $a,b\in\mathbb{R}^n$, the dot product will be indifferently denoted as $a^\top b$ or $\langle a, b \rangle$. For $p\in[1,\infty]$, $\norm{a}_p$ is the $\ell_p$-norm of $a$. Finally, if $c\in \mathbb{R}$, the indicator function $\mathbbm{1}(c)$ has value 1 if $c$ is positive and 0 otherwise. By convention, we assume that $\frac{1}{\infty}:=0$ and $\norm{a}_{\infty}^{\infty}:=\norm{a}_{\infty}$.

\subsection{A selected review of SVM models} \label{sec_review_methods}
Let $\{(x^{(i)},y^{(i)})\}_{i=1}^m$ be the set of training data points, where $x^{(i)}\in \mathbb{R}^n$ is the vector of features, and $y^{(i)}\in\{-1,+1\}$ is the label representing the class to which the $i$-th data point belongs. In particular, we denote by $\mathcal{A}$ and $\mathcal{B}$ the \emph{positive} (label ``$+1$'') and \emph{negative} (label ``$-1$'') classes, respectively.

The \emph{Soft Margin}-SVM approach (SM-SVM, \cite{CorVap1995}) finds the best separating hyperplane $H:=(w,\gamma)$ defined by the equation $w^{\top}x=\gamma$, where $w\in \mathbb{R}^n$ and $\gamma\in\mathbb{R}$, as solution of the following $\ell_q$-model, $q\in [1,\infty]$:
\begin{linenomath}
\begin{equation} \label{SMSVM}
\begin{aligned}
\min_{w,\gamma,\xi}  \quad & \norm{w}_q^q+\nu \sum_{i=1}^m\xi_i\\
\text{s.t.} \quad & y^{(i)}(w^{\top}x^{(i)}-\gamma)\geq 1-\xi_i & \quad  i=1,\ldots,m\\
 & \xi_i\geq 0 & \quad  i=1,\ldots,m.
\end{aligned}
\end{equation}
\end{linenomath}

The vector $\xi\in\mathbb{R}^m$ is the soft margin error vector and $\nu \geq 0$ is a regularization parameter. Data point $x^{(i)}$ is correctly classified by the separating hyperplane $H$ if $0\leq \xi_i \leq 1$, otherwise is misclassified.

Whenever a new observation $x\in\mathbb{R}^n$ occurs, it is classified as \emph{positive} or \emph{negative} depending on the decision function $\mathbbm{1}\big(w^{\top}x-\gamma\big)$.

Instead of a single hyperplane, in \cite{LiuPot2009} a pair of parallel hyperplanes $H_{\mathcal{A}}$ and $H_{\mathcal{B}}$ is constructed, satisfying the following properties:
\begin{enumerate}
\item[\textbf{(P1)}] all points of class $\mathcal{A}$ lie on one halfspace of $H_{\mathcal{A}}$;
\item[\textbf{(P2)}] all points of class $\mathcal{B}$ lie on the opposite halfspace of $H_{\mathcal{B}}$;
\item[\textbf{(P3)}] the intersection of the convex hulls of $\mathcal{A}$ and $\mathcal{B}$ is contained in the region between $H_{\mathcal{A}}$ and $H_{\mathcal{B}}$.
\end{enumerate}
The starting point of the formulation consists in solving the SM-SVM model \eqref{SMSVM} with $q=1$, determining an initial separating hyperplane $H_0:=(w,\gamma)$ and the soft margin vector $\xi$. Then, $H_0$ is shifted in order to identify $H_{\mathcal{A}}:=(w,\gamma-1+\omega_{A})$ and $H_{\mathcal{B}}:=(w,\gamma+1-\omega_{B})$, where:
\begin{linenomath}
\begin{equation} \label{omega_linear}
\omega_{\mathcal{A}}:=\max_{i:x^{(i)}\in\mathcal{A}} \big\{\xi_i\big\}, \quad \omega_{\mathcal{B}}:=\max_{i:x^{(i)}\in\mathcal{B}} \big\{\xi_i\big\}.
\end{equation}
\end{linenomath}

The choice of $\omega_{\mathcal{A}}$ and $\omega_{\mathcal{B}}$ according to condition \eqref{omega_linear} guarantees that $H_{\mathcal{A}}$ and $H_{\mathcal{B}}$ satisfy properties \textbf{(P1)-(P3)}.

Finally, the optimal separating hyperplane $H:=(w,b)$ is such that is parallel to $H_{\mathcal{A}}$ and $H_{\mathcal{B}}$, lies in their strip, and the number of misclassified points is minimized. These conditions are satisfied finding the optimal parameter $b$ as solution of the following problem:
\begin{linenomath}
\begin{equation}\label{linear_search_SVM}
\begin{aligned}
\min_{b}  \quad & \sum_{i:x^{(i)}\in\mathcal{A}} \mathbbm{1}\big(w^\top x^{(i)}-b\big)+\sum_{i:x^{(i)}\in\mathcal{B}} \mathbbm{1}\big(b-w^\top x^{(i)}\big)\\
\text{s.t.} \quad & \gamma+1-\omega_{\mathcal{B}}\leq b \leq \gamma-1+\omega_{\mathcal{A}}.
\end{aligned}
\end{equation}
\end{linenomath}

From a computational perspective, model \eqref{linear_search_SVM} is solved through a linear search procedure. Specifically, the interval $[\gamma+1-\omega_{\mathcal{B}},\gamma-1+\omega_{\mathcal{A}}]$ is divided into $N_{\max}$ sub-intervals of equal length and the problem is solved on each of them. The optimal solution $b$ is the one providing the overall minimum value of the objective function.

Similarly to SM-SVM, a new data point $x\in\mathbb{R}^n$ is classified in class $\mathcal{A}$ or $\mathcal{B}$ depending on the decision rule $\mathbbm{1}\big(w^{\top}x-b\big)$.

Whenever training observations are not linearly separable, the so-called \emph{kernel trick} can be applied (\cite{CorVap1995}). The key idea is to introduce a function $\phi(\cdot)$, usually referred to as \emph{feature map}, to translate data from the \emph{input space} $\mathbb{R}^n$ to a higher-dimensional space $\mathcal{H}$, equipped with the dot product $\langle\cdot,\cdot\rangle$. In $\mathcal{H}$, the transformed data $\{\phi(x^{(i)})\}_{i=1}^m$ are assumed to be linearly separable. Thus, model \eqref{SMSVM} can be written in the feature space $\mathcal{H}$ as:
\begin{linenomath}
\begin{equation} \label{phiSVM}
\begin{aligned}
\min_{\overline{w},\gamma,\xi}  \quad & \norm{\overline{w}}_{\mathcal{H}}+\nu \sum_{i=1}^m\xi_i\\
\text{s.t.} \quad & y^{(i)}(\langle \overline{w},\phi(x^{(i)})\rangle-\gamma)\geq 1-\xi_i & \quad  i=1,\ldots,m\\
 & \xi_i\geq 0 & \quad  i=1,\ldots,m.
\end{aligned}
\end{equation}
\end{linenomath}

Vector $\overline{w}$ $\in\mathcal{H}$ defines the linear classifier in the feature space and the norm $\norm{\cdot}_{\mathcal{H}}$ is induced by the inner product $\langle\cdot,\cdot\rangle$.

Unfortunately, the expression of the mapping $\phi(\cdot)$ is usually unknown and, consequently, model \eqref{phiSVM} cannot be solved in practice. To overcome this limitation, a symmetric and positive semidefinite kernel $k:\mathbb{R}^n\times \mathbb{R}^n \to \mathbb{R}$ is introduced. Examples of kernel functions typically used in ML literature are reported in Table \ref{tab_kernels}. For a comprehensive overview, the reader is referred to \cite{SchSmo2001}.

\begin{table}[h!]
\centering
\resizebox{0.8\textwidth}{!}{
\begin{tabular}{lll}\toprule
Kernel function & $k(x,x')$ & Parameter \\ \hline
Homogeneous polynomial & $k(x,x')=\langle x,x'\rangle^d$ & $d\in \mathbb{N}$\\
Inhomogeneous polynomial & $k(x,x')=(c+\langle x,x'\rangle)^d$ & $c \in \mathbb{R}^{+}$, $d\in \mathbb{N}$\\
Gaussian Radial Basis Function (RBF) & $k(x,x')= \exp\bigg(\displaystyle -\frac{\norm{x-x'}_2^2}{2 \alpha^2}\bigg)$ & $\alpha \in \mathbb{R}^{+}_{0}$\\
\bottomrule
\end{tabular}}
\caption{Examples of kernel functions. The first column reports the name of the functions. The second column provides their mathematical expressions. Finally, the third column contains the related relevant parameters.} \label{tab_kernels}
\end{table}

As in \cite{CorVap1995}, $\overline{w}$ can be decomposed into a finite linear combination of $\{\phi(x^{(j)})\}_{j=1}^m$, as $\overline{w}=\sum_{j=1}^m y^{(j)}u_j\phi(x^{(j)})$, for some coefficients $u_j\in\mathbb{R}$. Consequently, for all $i=1,\ldots,m$ the dot product $\langle \overline{w},\phi(x^{(i)})\rangle$ in the first set of constraints of model \eqref{phiSVM} can be formulated as $\langle \overline{w},\phi(x^{(i)})\rangle =\sum_{j=1}^m K_{ij}y^{(j)}u_j$, where $K_{ij}:=k(x^{(i)},x^{(j)})=\langle\phi(x^{(i)}),\phi(x^{(j)})\rangle$. The properties of the kernel function imply that the Gram matrix $K=[K_{ij}]$ is a real, symmetric and positive semidefinite $m \times m$ matrix (\cite{PicSci2018}).

As in \cite{Man1998,LeeManWol2000}, in the objective function of model \eqref{phiSVM} the $\mathcal{H}$-norm $\norm{\overline{w}}_{\mathcal{H}}$ is replaced by $\norm{u}_{q}^q$, where $u:=[u_1,\ldots,u_m]^\top$. This choice guarantees the convexity of the optimization problem. Therefore, model \eqref{phiSVM} can be rewritten as:
\begin{linenomath}
\begin{equation} \label{GSVM}
\begin{aligned}
\min_{u,\gamma,\xi}  \quad & \norm{u}_{q}^q+\nu \sum_{i=1}^m\xi_i\\
\text{s.t.} \quad & y^{(i)}\bigg(\sum_{j=1}^m K_{ij}y^{(j)}u_j-\gamma\bigg)\geq 1-\xi_i & \quad  i=1,\ldots,m\\
 & \xi_i\geq 0 & \quad  i=1,\ldots,m.
\end{aligned}
\end{equation}
\end{linenomath}

Within this context, the separating hyperplane in the feature space translates into a nonlinear decision boundary $S:=(u,\gamma)$ in the input space, defined by the following equation:
\begin{linenomath}
\begin{equation}\label{nonlinear_hyperplane}
\sum_{i=1}^m k(x,x^{(i)})y^{(i)}u_i=\gamma.
\end{equation}
\end{linenomath}

Finally, each new observation $x\in \mathbb{R}^n$ is classified either in class $\mathcal{A}$ or $\mathcal{B}$ according to the decision function $\mathbbm{1}\big(\sum_{i=1}^m k(x,x^{(i)})y^{(i)}u_i-\gamma\big)$.

\section{A novel approach for deterministic nonlinear SVM} \label{sec_novel_deterministic_model}
In this section, we propose an extension of the SVM approach presented in \cite{LiuPot2009} to the nonlinear case. Specifically, we classify input observations by means of kernel-induced decision boundaries, such that the corresponding hyperplanes in the feature space satisfy properties \textbf{(P1)}-\textbf{(P3)}.

In Section \ref{sec_novel_binary_det} we tackle a binary classification task, whereas in Section \ref{sec_novel_multiclass_det} we extend the approach to the case of multiclass classification.

\subsection{Binary classification} \label{sec_novel_binary_det}

First of all, we start solving model \eqref{GSVM} and finding an initial decision boundary $S_0:=(u,\gamma)$. In the input space, hypersurface $S_0$ induces an initial nonlinear separation of training data points. Accordingly, in the feature space the corresponding hyperplane $H_0$ performs a linear classification of transformed observations.

Then, for each of the two classes, we compute the greatest misclassification error through the extended version of formulas \eqref{omega_linear}:
\begin{linenomath}
\begin{equation} \label{omega_nonlinear}
\omega_{\mathcal{A}}:= \max_{i=1,\ldots,m} {(D\xi)}_i \qquad \omega_{\mathcal{B}}:= \max_{i=1,\ldots,m} {(-D\xi)}_i,
\end{equation}
\end{linenomath}
where $D$ is a diagonal matrix with entries $D_{ii}:=y^{(i)}$, for all $i=1,\ldots,m$.

Due to the structure of problem \eqref{GSVM}, the modulus of $-1+\omega_{\mathcal{A}}$ represents the deviation of the farthest misclassified point of class $\mathcal{A}$ from $H_0$ and similarly for $1-\omega_{\mathcal{B}}$. Nevertheless, it may happen that $H_0$ already correctly classifies all the data points of one or both classes. In that case, the moduli are just the deviations of the closest data points from hyperplane $H_0$. According to the classic literature of SVM (see \cite{CorVap1995}), we call \emph{support vectors} of class $\mathcal{A}$ and $\mathcal{B}$ the transformed points that deviate $\abs{-1+\omega_{\mathcal{A}}}$ and $\abs{1-\omega_{\mathcal{B}}}$ from $H_0$, respectively.

At this stage, similarly to \cite{LiuPot2009}, we shift hyperplane $H_0$ by $-1+\omega_{\mathcal{A}}$ and $1-\omega_{\mathcal{B}}$, obtaining $H_{\mathcal{A}}$ and $H_{\mathcal{B}}$, respectively. Such a pair of parallel hyperplanes passes through the support vectors of the corresponding class and satisfies properties \textbf{(P1)}-\textbf{(P3)} in the feature space. According to equation \eqref{nonlinear_hyperplane}, the corresponding hypersurfaces $S_{\mathcal{A}}:=(u,\gamma-1+\omega_{\mathcal{A}})$ and $S_{\mathcal{B}}:=(u,\gamma+1-\omega_{\mathcal{B}})$ are then derived in the input space.

Finally, the optimal kernel-induced decision boundary $S:=(u,b)$ is deduced, where $b$ is the solution of the nonlinear version of model \eqref{linear_search_SVM}:
\begin{linenomath}
\begin{equation}\label{linear_search_GSVM}
\begin{aligned}
\min_{b}  \quad & \sum_{i=1}^m \mathbbm{1}\bigg(y^{(i)}b-y^{(i)}\sum_{j=1}^m K_{ij}y^{(j)}u_j\bigg)\\
\text{s.t.} \quad & \gamma+1-\omega_{\mathcal{B}}\leq b \leq \gamma-1+\omega_{\mathcal{A}}.
\end{aligned}
\end{equation}
\end{linenomath}

We observe that hypersurface $S$ in the input space is induced by hyperplane $H$ in the feature space, which is parallel to $H_{\mathcal{A}}$ and $H_{\mathcal{B}}$ and lies in the region between them. Therefore, a new observation $x\in\mathbb{R}^n$ is classified according to the decision function $\mathbbm{1}\big(\sum_{i=1}^m k(x,x^{(i)})y^{(i)}u_i-b\big)$.

For the sake of clarity, all the steps of the approach discussed so far are schematically reported in Pseudocode \ref{pseudocode_novel_approach}.

\begin{algorithm}[H]
\begin{algorithmic}[1]
\REQUIRE $\{(x^{(i)},y^{(i)})\}_{i=1}^m$, $q\in[1,\infty]$, $\nu\geq 0$, $k(\cdot,\cdot):\mathbb{R}^n\times \mathbb{R}^n \to \mathbb{R}$.
\STATE Calculate matrix $K$ as $K_{ij}=k\big(x^{(i)},x^{(j)}\big), i,j=1,\ldots,m$.
\STATE Solve model \eqref{GSVM}.
\STATE Find the initial separating hypersurface $S_0=(u,\gamma)$, defined by equation \eqref{nonlinear_hyperplane}.
\STATE Construct diagonal matrix $D$ as $D_{ii}=y^{(i)}, i=1,\ldots,m$, and compute $\omega_{\mathcal{A}}$ and $\omega_{\mathcal{B}}$ according to formulas \eqref{omega_nonlinear}.
\STATE Shift $S_0$ to get the separating hypersurface for each class, $S_{\mathcal{A}}=(u,\gamma-1+\omega_{\mathcal{A}})$ and $S_{\mathcal{B}}=(u,\gamma+1-\omega_{\mathcal{B}})$, defined by \eqref{nonlinear_hyperplane}.
\STATE Solve model \eqref{linear_search_GSVM}, obtaining parameter $b$.
\ENSURE The optimal decision boundary $S=(u,b)$, defined by \eqref{nonlinear_hyperplane}.
\end{algorithmic}
\caption{A novel approach for nonlinear SVM.} \label{pseudocode_novel_approach}
\end{algorithm}

The computational complexity of nonlinear SVM models is between $O(m^2)$ and $O(m^3)$ (\cite{Peng2011}). Since model \eqref{linear_search_GSVM} requires at most $N_{\max}$ iterations to be solved through a linear search procedure, the computational complexity of our approach is between $O(\max\{m^2,N_{\max}\})$ and $O(\max\{m^3,N_{\max}\})$.

By way of illustration, in Figure \ref{fig_toy_2d_det} we depict the separating surfaces obtained by applying the proposed SVM methodology to a bidimensional toy example. In model \eqref{GSVM} we set $q=1$, $\nu=1$ and consider linear and Gaussian RBF kernels. The graphical interpretation of the novel approach is illustrated in Figure \ref{fig_toy_3d_RBF}.

\begin{figure}[h!]
     \centering
  \begin{subfigure}[b]{0.49\textwidth}
         \centering
         \includegraphics[width=\textwidth]{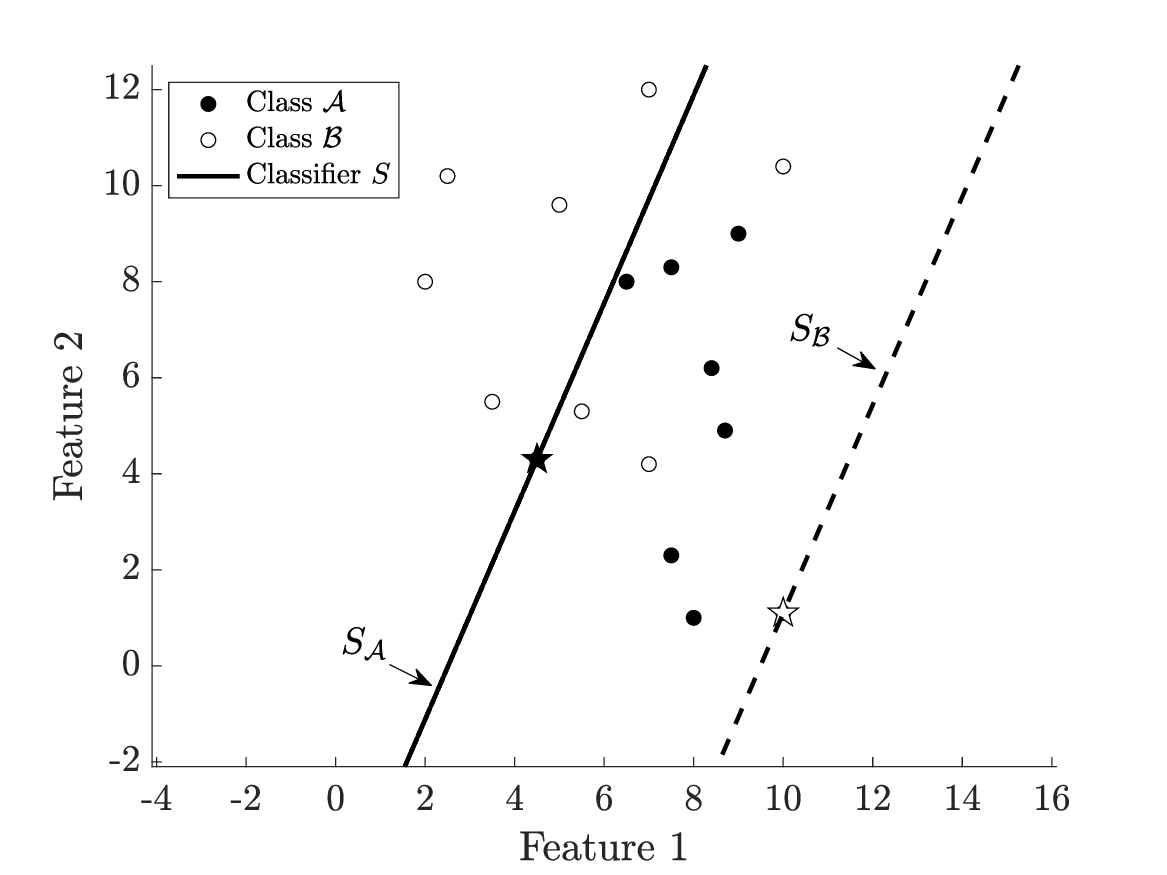}
         \caption{Linear kernel}
         \label{fig_toy_2d_det_lin}
     \end{subfigure}
     \begin{subfigure}[b]{0.49\textwidth}
         \centering
\includegraphics[width=\textwidth]{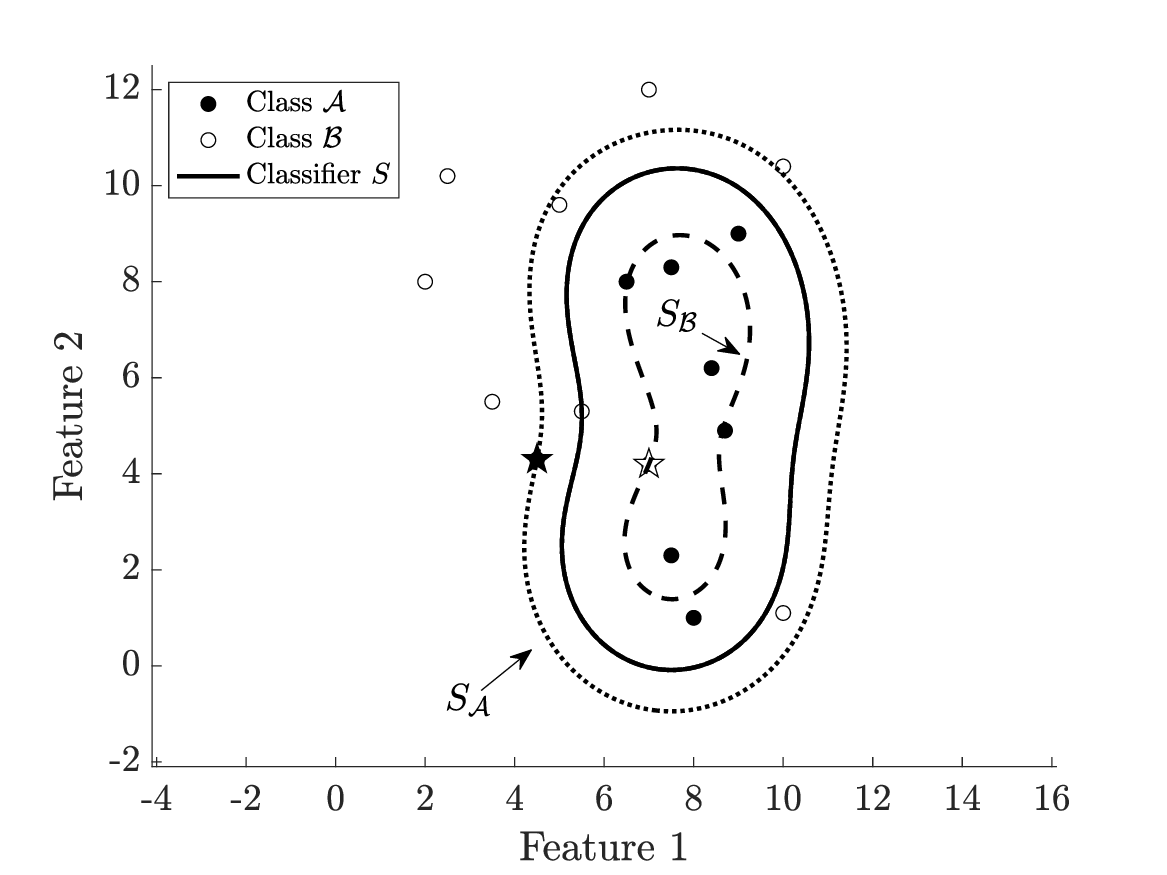}
         \caption{Gaussian RBF kernel ($\alpha=1.9$)}
         \label{fig_toy_2d_det_RBF}
     \end{subfigure}
          \caption{Separating surfaces obtained with linear and Gaussian RBF kernel functions. Support vectors are depicted as stars.}
        \label{fig_toy_2d_det}
\end{figure}

\begin{figure}[h!] 
\begin{center}
\includegraphics[width=5in]{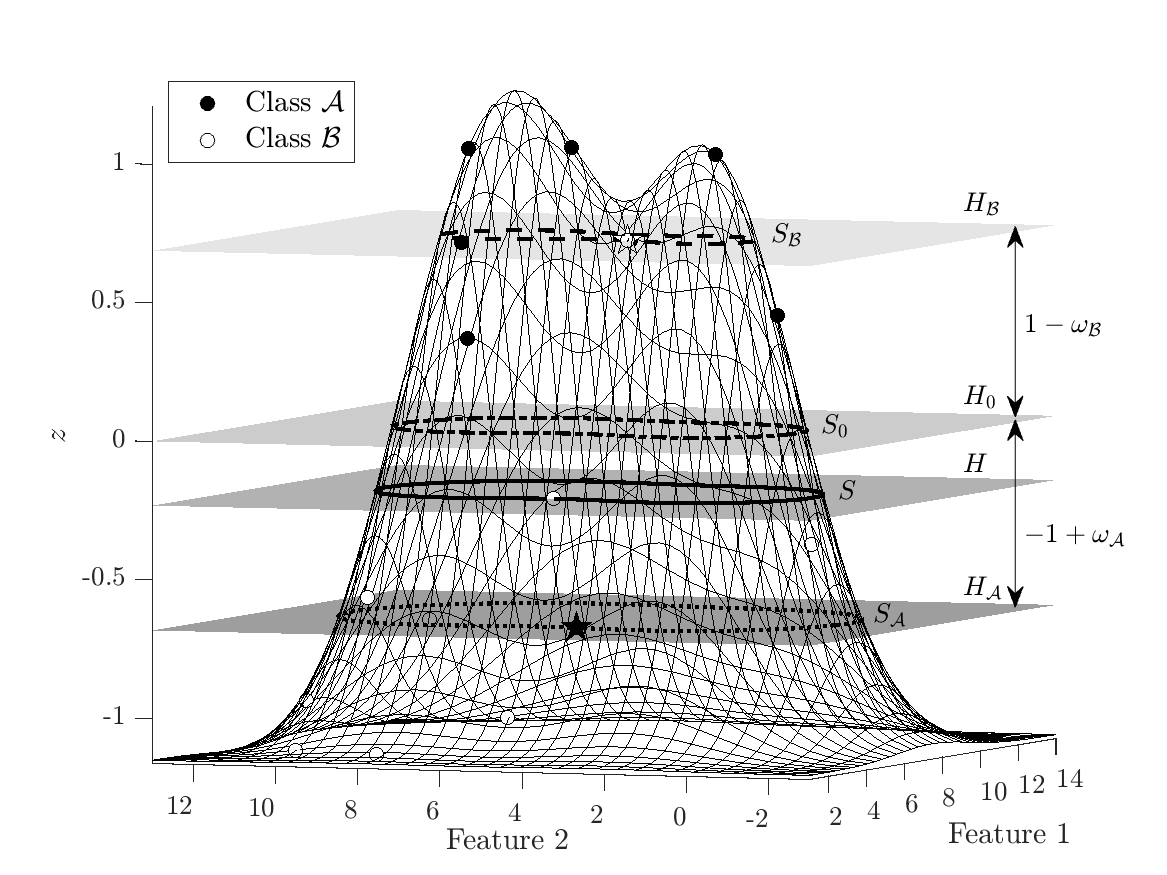}
\end{center}
\caption{Graphical representation of the implicit function \eqref{nonlinear_hyperplane}, in the case of Gaussian RBF kernel ($\alpha=1.9$), along with the separating hyperplanes and decision boundaries. Support vectors are drawn as stars.}
\label{fig_toy_3d_RBF}
\end{figure}

\subsection{Multiclass classification} \label{sec_novel_multiclass_det}
In this section, we derive the multiclass extension of the approach presented so far. We focus our attention on one of the most commonly used multiclass SVM framework, the \emph{one-versus-all} (\cite{Vap1995,WesWat1998}). According to this methodology, $L$ binary classifiers are constructed, where $L$ is the number of classifying categories, such that each class is independently separated by all the others grouped together. Formally, let $\{(x^{(i)},y^{(i)})\}_{i=1}^m$ be the set of training observations, with $x^{(i)}\in \mathbb{R}^n$ and $y^{(i)}\in\{1,\ldots,L\}$. For each class $l=1,\ldots,L$, we find an initial separating hypersurface $S_{l,0}:=(u_l,\gamma_l)$, where $u_l\in\mathbb{R}^n$ and $\gamma_l\in\mathbb{R}$ are the solutions of the following multiclass version of model \eqref{GSVM}:
\begin{linenomath}
\begin{equation} \label{GSVM_multiclass}
\begin{aligned}
\min_{u_l,\gamma_l,\xi_l}  \quad & \norm{u_l}_{q}^q+\nu \sum_{i=1}^m\xi_{l,i}\\
\text{s.t.} \quad & \widehat{y}^{(i)}_l\bigg(\sum_{j=1}^m K_{ij}\widehat{y}^{(j)}_lu_{l,j}-\gamma_l\bigg)\geq 1-\xi_{l,i} & \quad  i=1,\ldots,m\\
 & \xi_{l,i}\geq 0 & \quad  i=1,\ldots,m,
\end{aligned}
\end{equation}
\end{linenomath}
with $\widehat{y}^{(i)}_l=1$ if $y^{(i)}=l$, and $\widehat{y}^{(i)}_l=-1$ otherwise. Then, we construct the diagonal matrix $\widehat{D}_l$, with $\widehat{D}_{l,ii}:=\widehat{y}^{(i)}_l$, $i=1,\ldots,m$, and compute:
\begin{linenomath}
\begin{equation*}
\omega_{l}:= \max_{i=1,\ldots,m} {(\widehat{D}_l\xi_l)}_i \qquad \omega_{-l}:= \max_{i=1,\ldots,m} {(-\widehat{D}_l\xi_l)}_i.
\end{equation*}
\end{linenomath}

Hypersurface $S_{l,0}$ is then shifted to get $S_{l}:=(u_l,\gamma_l-1+\omega_l)$ and $S_{-l}:=(u_l,\gamma_l+1-\omega_{-l})$ in the input space. The corresponding hyperplanes in the feature space satisfy properties \textbf{(P1)}-\textbf{(P3)}. Finally, the optimal decision boundary for class $l$ versus all the others is $S_{l,-l}:=(u_l,b_l)$, with $b_l$ solution of the following model:
\begin{linenomath}
\begin{equation}\label{linear_search_GSVM_multiclass}
\begin{aligned}
\min_{b_l}  \quad & \sum_{i=1}^m \mathbbm{1}\bigg(\widehat{y}^{(i)}_lb_l-\widehat{y}^{(i)}_l\sum_{j=1}^m K_{ij}\widehat{y}^{(j)}_lu_{l,j}\bigg)\\
\text{s.t.} \quad & \gamma_l+1-\omega_{-l}\leq b_l \leq \gamma_l-1+\omega_{l}.
\end{aligned}
\end{equation}
\end{linenomath}

The decision function of the $l$-th class is given by $f_l(x):=\sum_{i=1}^m k(x,x^{(i)})\widehat{y}^{(i)}_lu_{l,i}-b_l$, and each new observation $x\in\mathbb{R}^n$ is assigned to the class $l^*:=\argmax_{l= 1,\ldots,L} f_l(x)$ (\cite{LopMalCar2017}).

Since the \emph{one-versus-all} strategy generates $L$ binary classifiers, one for each class, the computational complexity of our multiclass approach is between $O(L\cdot\max\{m^2,N_{\max}\})$ and $O(L\cdot\max\{m^3,N_{\max}\})$.

We represent in Figure \ref{fig_toy_2d_det_multiclass} the results of the proposed methodology in the case of a multiclass classification task. The parameters $q$ and $\nu$ are the same as in Figure \ref{fig_toy_2d_det}. Similarly to the binary case (see Figure \ref{fig_toy_2d_det_lin}), it may happen that either $S_{l}$ or $S_{-l}$ coincides with $S_{l,-l}$. This is due to the fact that in model \eqref{linear_search_GSVM_multiclass} the optimal parameter $b_l$ may be equal to $\gamma_l-1+\omega_l$ or $\gamma_l+1-\omega_{-l}$, respectively.

\begin{figure}[h!]
     \centering
  \begin{subfigure}[b]{0.49\textwidth}
         \centering
         \includegraphics[width=\textwidth]{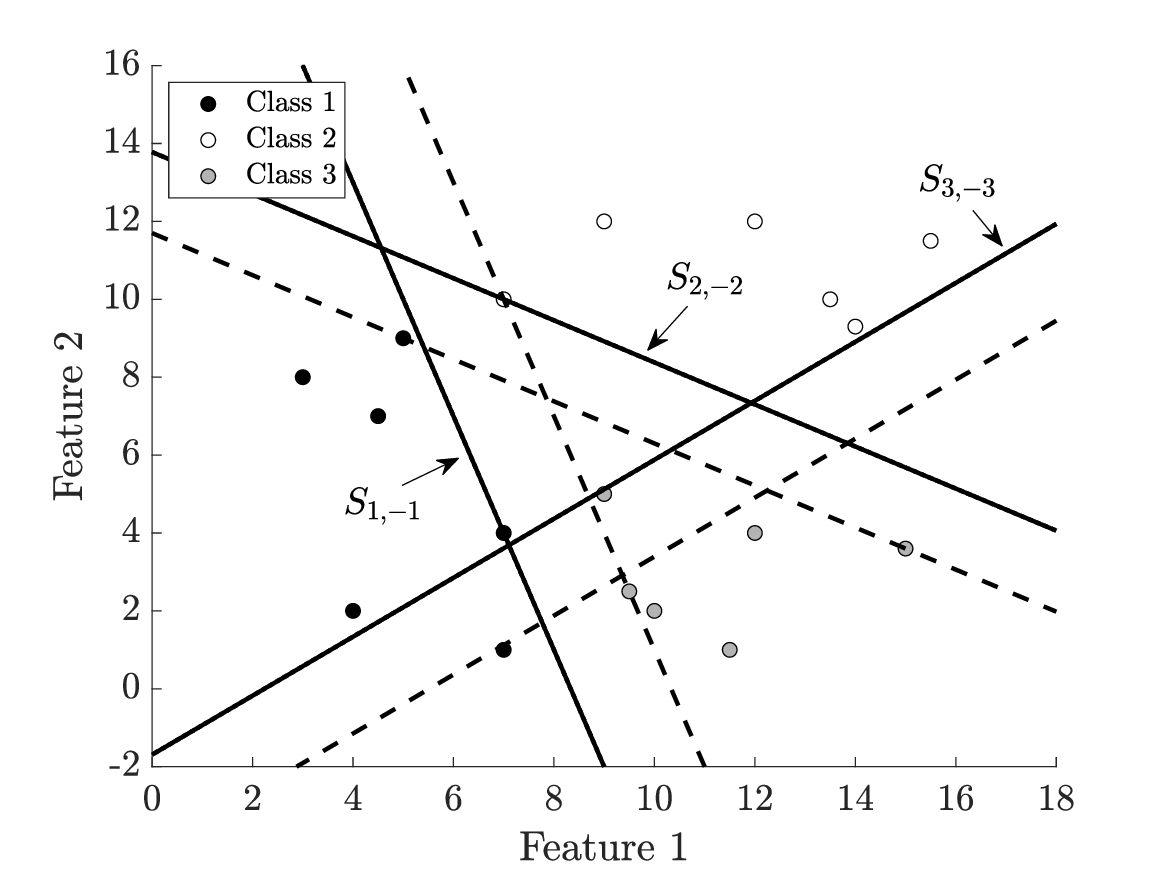}
         \caption{Linear kernel}
         \label{fig_toy_2d_det_lin_multiclass}
     \end{subfigure}
     \begin{subfigure}[b]{0.49\textwidth}
         \centering
\includegraphics[width=\textwidth]{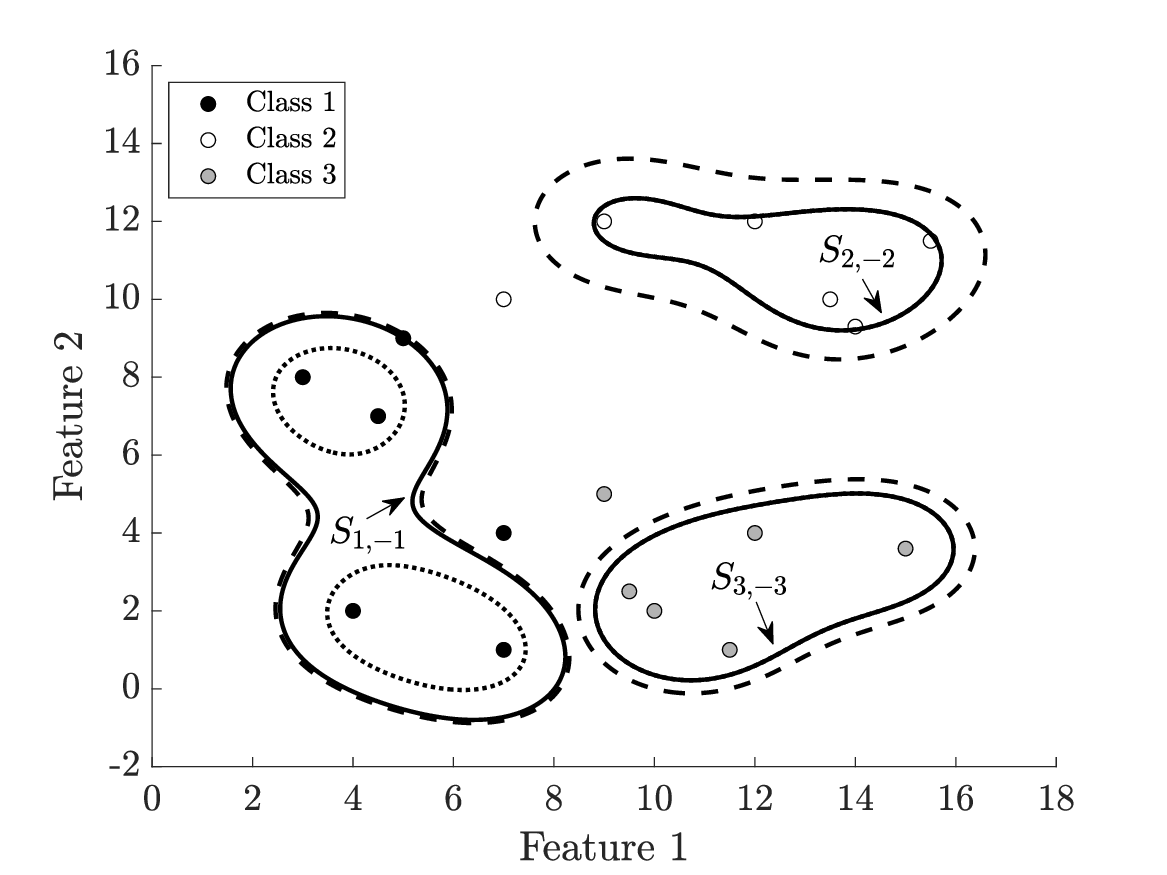}
         \caption{Gaussian RBF kernel ($\alpha=1.9$)}
         \label{fig_toy_2d_det_RBF_multiclass}
     \end{subfigure}
          \caption{Separating surfaces obtained with linear and Gaussian RBF kernel functions in the case of a three-classes classification task. For each class $l=1,2,3$, the dotted line and the dashed line represent respectively $S_{l}$ and $S_{-l}$.}
        \label{fig_toy_2d_det_multiclass}
\end{figure}

\section{A robust model for nonlinear SVM} \label{sec_robust_model}
In this section, we derive the robust counterpart of the deterministic approach discussed so far, when input data are plagued by uncertainties. According to the RO framework, we construct an uncertainty set around each observation and optimize against the worst-case realization across the entire uncertainty set (\cite{BerDunPawZhu2019}).

Contrariwise to RO models dealing with linear classification (see, for instance, \cite{FacMagPot2022}), in the nonlinear context data points $x^{(i)}$ are mapped into the feature space $\mathcal{H}$ via $\phi(\cdot)$ and uncertainty sets $\mathcal{U}_{\mathcal{H}}\big(\phi(x^{(i)})\big)$ have to be constructed. Unfortunately, a closed-form expression of $\phi(\cdot)$ is rarely available and an \emph{a priori} control about $\mathcal{U}_{\mathcal{H}}\big(\phi(x^{(i)})\big)$ is not possible. Therefore, further assumptions on the uncertainty set $\mathcal{U}_{\mathcal{H}}\big(\phi(x^{(i)}\big)$ in the feature space are necessary.

The remainder of the section is organized as follows. In Section \ref{subsec_uncertainty_set} bounded-by-$\ell_p$-norm uncertainty sets $\mathcal{U}_p(x^{(i)})$ are constructed, together with the corresponding ones $\mathcal{U}_{\mathcal{H}}\big(\phi(x^{(i)})\big)$ in the feature space. Bounds on the radii of $\mathcal{U}_{\mathcal{H}}\big(\phi(x^{(i)})\big)$ are derived in Section \ref{subsec_bounds}. Finally, in Section \ref{subsec_robust_model} the robust counterpart of models \eqref{GSVM} and \eqref{GSVM_multiclass} is rigorously deduced, together with computationally tractable reformulations.

\subsection{The construction of the uncertainty sets} \label{subsec_uncertainty_set}
We assume that each observation $x^{(i)}$ in the input space is subject to an additive and unknown perturbation vector $\sigma^{(i)}$, whose $\ell_p$-norm, with $p\in[1,\infty]$, is bounded by a nonnegative constant $\eta^{(i)}$. Consequently, the uncertainty set around $x^{(i)}$ has the following expression:
\begin{linenomath}
\begin{equation} \label{U_input_space}
\mathcal{U}_p(x^{(i)}):=\Set{ x\in\mathbb{R}^n : x=x^{(i)}+\sigma^{(i)}, \lVert \sigma^{(i)} \lVert_p \leq \eta^{(i)} }.
\end{equation}
\end{linenomath}

Parameter $\eta^{(i)}$ calibrates the degree of conservatism: if $\eta^{(i)}=0$, then $\sigma^{(i)}$ is the zero vector of $\mathbb{R}^n$ and $\mathcal{U}_p(x^{(i)})$ coincides with $x^{(i)}$. Popular choices for the $\ell_p$-norm in the RO literature are $p=1,2,\infty$, leading to polyhedral, spherical and box uncertainty sets, respectively.

In order to consider the extension towards the feature space, we now assume that, if $x$ belongs to $\mathcal{U}_p(x^{(i)})$, then:
\begin{linenomath}
\begin{equation*}
\phi(x)=\phi(x^{(i)}+\sigma^{(i)})=\phi(x^{(i)})+\zeta^{(i)},
\end{equation*}
\end{linenomath}

where the perturbation $\zeta^{(i)}$ belongs to the feature space $\mathcal{H}$ and its $\mathcal{H}$-norm is bounded a nonnegative constant $\delta^{(i)}$. The latter may be unknown but it depends on the known bound $\eta^{(i)}$ in the input space, i.e. $\delta^{(i)}=\delta^{(i)}\big(\eta^{(i)}\big)$. If no uncertainty occurs in the input space, no uncertainty will occur in the feature space too: $\eta^{(i)}=0$ implies $\delta^{(i)}=0$. Hence, the uncertainty set around $\phi(x^{(i)})$ in the feature space is modelled as:
\begin{linenomath}
\begin{equation} \label{U_feature_space}
\mathcal{U}_{\mathcal{H}}\big(\phi(x^{(i)})\big):=\Set{ z\in\mathcal{H} : z=\phi(x^{(i)})+\zeta^{(i)}, \lVert \zeta^{(i)} \lVert_{\mathcal{H}} \leq \delta^{(i)} }.
\end{equation}
\end{linenomath}

\subsection{Bounds on the uncertainty sets in the feature space} \label{subsec_bounds}
Let $k(\cdot,\cdot)$ be a symmetric and positive semidefinite kernel, with corresponding feature map $\phi(\cdot)$. 
In the following, we derive closed-form expressions for the radius $\delta^{(i)}$ in the feature space given the bound $\eta^{(i)}$ in the input space, when $k(\cdot,\cdot)$ is the polynomial kernel or the Gaussian RBF kernel. Below, we provide the results and relegate the proofs to  \ref{appendix_proofs}.

\begin{proposition}[\textbf{Polynomial kernel}] \label{lemma_inhom_pol}
Let $\mathcal{U}_p(x^{(i)})$ and $\mathcal{U}_{\mathcal{H}}\big(\phi(x^{(i)})\big)$ be the uncertainty sets in the input and in the feature space as in \eqref{U_input_space} and \eqref{U_feature_space}, respectively, with $p\in[1,\infty]$. Consider the inhomogeneous polynomial kernel of degree $d\in\mathbb{N}$ and additive constant $c\geq 0$, with radius $\delta^{(i)}\equiv\delta^{(i)}_{d,c}$, and:
\begin{linenomath}
\begin{equation*}
  C=C(n,p)= \begin{cases}
    1, & 1\leq p \leq 2\\
n^{\textstyle \frac{p-2}{2p}}, & p > 2.
  \end{cases}
\end{equation*}
\end{linenomath}

\begin{itemize}
\item[{(i)}] If $d=1$, then the radius of $\mathcal{U}_{\mathcal{H}}\big(\phi(x^{(i)})\big)$ is:
\begin{linenomath}
\begin{equation} \label{delta_ihhomogeneous_1}
\delta^{(i)}_{1,c} =  C\eta^{(i)}.
\end{equation}
\end{linenomath}

\item[{(ii)}] If $d > 1$, then:
\begin{linenomath}
\begin{equation} \label{delta_inhomogeneous}
\delta^{(i)}_{d,c} =  \sqrt{\big(\delta^{(i)}_{d,0}\big)^2+\sum_{k=1}^{d-1}\binom{d}{k}c^k\bigg[\sum_{j=1}^{d-k} \binom{d-k}{j} \norm{x^{(i)}}_2^{d-k-j}\big(C\eta^{(i)}\big)^j\bigg]^2},
\end{equation}
\end{linenomath}
where $\delta^{(i)}_{d,0}$ is the bound for the corresponding homogeneous polynomial kernel:
\begin{linenomath}
\begin{equation} \label{delta_homogeneous}
\delta^{(i)}_{d,0} =  \sum_{k=1}^d \binom{d}{k} \norm{x^{(i)}}_2^{d-k}\big(C\eta^{(i)}\big)^k.
\end{equation}
\end{linenomath}
\end{itemize}

Notice that when $c=0$, equation \eqref{delta_inhomogeneous} reduces to  \eqref{delta_homogeneous}.
\end{proposition}

\begin{proposition}[\textbf{Gaussian RBF kernel}] \label{lemma_RBF_gaussian}
Let $\mathcal{U}_p(x^{(i)})$ and $\mathcal{U}_{\mathcal{H}}\big(\phi(x^{(i)})\big)$ be the uncertainty sets in the input and in the feature space as in \eqref{U_input_space} and \eqref{U_feature_space}, respectively, with $p\in[1,\infty]$. Consider the Gaussian RBF kernel with parameter $\alpha > 0$ and radius $\delta^{(i)}\equiv\delta_{\alpha}^{(i)}$. If:
\begin{linenomath}
\begin{equation*}
  C=C(n,p)= \begin{cases}
    1, & 1\leq p \leq 2\\
n^{\textstyle \frac{p-2}{2p}}, & p > 2,
  \end{cases}
\end{equation*}
\end{linenomath}
then:
\begin{linenomath}
\begin{equation} \label{delta_RBF}
\delta^{(i)}_{\alpha} =  \sqrt{2-2\exp\bigg(\displaystyle -\frac{{(C\eta^{(i)})}^2}{2\alpha^2}\bigg)}.
\end{equation}
\end{linenomath}

\end{proposition}

We observe that Propositions \ref{lemma_inhom_pol}-\ref{lemma_RBF_gaussian} are consistent with Lemma 7 presented in \cite{XuCarMan2009}. However, in this paper we specify the bounds for particular choices of the kernel functions. In addition, we extend the result for a bounded-by-$\ell_p$-norm uncertainty set for a generic $p\in[1,\infty]$.

\subsection{The robust model} \label{subsec_robust_model}
Robustifying model \eqref{GSVM} against the uncertainty set $\mathcal{U}_p(x^{(i)})$ yields the following optimization program:
\begin{linenomath}
\begin{equation} \label{rob0_GSVM}
\begin{aligned}
\min_{u,\gamma,\xi} \quad &  \norm{u}_{q}^{q}+\nu \sum_{i=1}^m\xi_i\\
\text{s.t.} \quad & y^{(i)}\sum_{j=1}^m k(x,x^{(j)})y^{(j)}u_j\geq 1-\xi_i+y^{(i)}\gamma & \forall x \in \mathcal{U}_p(x^{(i)}),  \text{ }i=1,\ldots,m\\
 & \xi_i\geq 0 & \quad  i=1,\ldots,m.
\end{aligned}
\end{equation}
\end{linenomath}

Model \eqref{rob0_GSVM} cannot be solved in practice due to the infinite possibilities for choosing $x$ in $\mathcal{U}_p(x^{(i)})$. Nevertheless, it can be reformulated in a tractable form, as stated in the following theorem.

\begin{theorem}\label{teo_rob}
Let $\mathcal{U}_p(x^{(i)})$ and $\mathcal{U}_{\mathcal{H}}\big(\phi(x^{(i)})\big)$ be the uncertainty sets in the input and in the feature space as in \eqref{U_input_space} and \eqref{U_feature_space}, respectively, with $p\in[1,\infty]$. Model \eqref{rob0_GSVM} is equivalent to:
\begin{linenomath}
\begin{equation} \label{rob_GSVM}
\begin{aligned}
\min_{u,\gamma,\xi}  \quad & \norm{u}_{q}^{q}+\nu \sum_{i=1}^m\xi_i\\
\text{s.t.} \quad & y^{(i)}\sum_{j=1}^m K_{ij}y^{(j)}u_j -\delta^{(i)}\sum_{j=1}^{m} \sqrt{K_{jj}} \abs{u_j} \geq 1-\xi_i+y^{(i)}\gamma & i=1,\ldots,m\\
 & \xi_i\geq 0 &  i=1,\ldots,m.
\end{aligned}
\end{equation}
\end{linenomath}
\end{theorem}

\begin{proof}
The first set of constraints of model \eqref{rob0_GSVM} is equivalent to:
\begin{linenomath}
\begin{equation}\label{proof_teo_rob}
\begin{aligned}
\min_{x \in \mathcal{U}_p(x^{(i)})} \quad & y^{(i)}\sum_{j=1}^m k(x,x^{(j)})y^{(j)}u_j\geq 1-\xi_i+y^{(i)}\gamma \qquad i=1,\ldots,m.\\
\end{aligned}
\end{equation}
\end{linenomath}

Due to the definition of $\mathcal{U}_p(x^{(i)})$, for all $i=1,\ldots,m$ the left-hand side of \eqref{proof_teo_rob} can be re-stated as:
\begin{linenomath}
\begin{equation}\label{interm_rob_mod}
\begin{aligned}
\min_{\sigma^{(i)}}  \quad & y^{(i)}\sum_{j=1}^m k(x^{(i)}+\sigma^{(i)},x^{(j)})y^{(j)}u_j\\
\text{s.t.} \quad & \lVert \sigma^{(i)} \lVert_p \leq \eta^{(i)}.
\end{aligned}
\end{equation}
\end{linenomath}

According to the definition of the kernel function and the assumption on $\mathcal{U}_{\mathcal{H}}\big(\phi(x^{(i)})\big)$, we have that:
\begin{linenomath}
\begin{equation*}
k(x^{(i)}+\sigma^{(i)},x^{(j)})=\langle \phi(x^{(i)}+\sigma^{(i)}),\phi(x^{(j)}) \rangle = \langle\phi(x^{(i)})+\zeta^{(i)},\phi(x^{(j)})\rangle.
\end{equation*}
\end{linenomath}

Moreover, the linearity of the dot product in the feature space $\mathcal{H}$ implies that model \eqref{interm_rob_mod} can be written as:
\begin{linenomath}
\begin{equation} \label{rob_zeta}
\begin{aligned}
\min_{\zeta^{(i)}}  \quad & y^{(i)}\sum_{j=1}^m \langle\zeta^{(i)}, \phi(x^{(j)})\rangle\text{ } y^{(j)}u_j\\
\text{s.t.} \quad & \lVert \zeta^{(i)} \lVert_{\mathcal{H}} \leq \delta^{(i)},
\end{aligned}
\end{equation}
\end{linenomath}
where the term $y^{(i)}\sum_{j=1}^m \langle \phi(x^{(i)}), \phi(x^{(j)})\rangle\text{ } y^{(j)}u_j$ is equivalent to $y^{(i)}\sum_{j=1}^m K_{ij}y^{(j)}u_j$. Being independent of $\zeta^{(i)}$, it is moved to the right-hand side of \eqref{proof_teo_rob}.

Then, the modulus of the objective function of model \eqref{rob_zeta} can be bounded by $\sum_{j=1}^m \abs{\langle \zeta^{(i)}, \phi(x^{(j)}) \rangle} \cdot \abs{u_j}$. By applying the Cauchy-Schwarz inequality in $\mathcal{H}$ and the boundedness condition on $\norm{\zeta^{(i)}}_{\mathcal{H}}$, we get:
\begin{linenomath}
\begin{equation*}
\abs{\langle \zeta^{(i)}, \phi(x^{(j)}) \rangle} \leq \norm{\zeta^{(i)}}_{\mathcal{H}} \cdot \norm{\phi(x^{(j)})}_{\mathcal{H}} \leq \delta^{(i)} \cdot \sqrt{\langle \phi(x^{(j)}), \phi(x^{(j)}) \rangle}= \delta^{(i)} \cdot \sqrt{K_{jj}}.
\end{equation*}
\end{linenomath}

The value $K_{jj}$ is nonnegative, due to the positive semidefiniteness of the Gram matrix $K$. Therefore, we obtain:
\begin{linenomath}
\begin{equation} \label{last_line_teo}
\abs{y^{(i)}\sum_{j=1}^m \langle\zeta^{(i)}, \phi(x^{(j)})\rangle\text{ } y^{(j)}u_j } \leq \delta^{(i)}\sum_{j=1}^{m} \sqrt{K_{jj}} \abs{u_j}.
\end{equation}
\end{linenomath}

Thus, the optimal value of model \eqref{rob_zeta} is $-\delta^{(i)}\sum_{j=1}^{m} \sqrt{K_{jj}} \abs{u_j}$. By replacing the minimization term with this optimal value in the first set of constraints of \eqref{proof_teo_rob}, the thesis follows.
\end{proof}

When no uncertainty occurs in the data, $\delta^{(i)}=0$ for all $i=1,\ldots,m$ and the robust model \eqref{rob_GSVM} reduces to the deterministic formulation \eqref{GSVM}.

Model \eqref{rob_GSVM} is a convex nonlinear optimization model due to the presence of the $\ell_{q}$-norm of $u$. Nevertheless, it can be reformulated as a \emph{Linear Programming} (LP) problem when $q=1$ or $q=\infty$ and as a SOCP problem when $q=2$, as stated in the following result. The proof is provided in \ref{appendix_proofs}.

\begin{corollary} \label{coroll_LP_SOCP}
Model \eqref{rob_GSVM} can be expressed as a LP problem or as a SOCP problem in the following cases:
\begin{itemize}
\item[a)] Case $q=1$: LP problem
\begin{linenomath}
\begin{equation} \label{robust_q1}
\begin{aligned}
\min_{u,\gamma,\xi,s}  \quad & \sum_{i=1}^m s_i+\nu \sum_{i=1}^m\xi_i\\
\text{s.t.} \quad & y^{(i)}\sum_{j=1}^m K_{ij}y^{(j)}u_j -\delta^{(i)}\sum_{j=1}^{m} \sqrt{K_{jj}} s_j \geq 1-\xi_i+y^{(i)}\gamma & i=1,\ldots,m\\
& s_i \geq - u_i & i=1,\ldots,m\\
& s_i \geq u_i & i=1,\ldots,m\\
 & s_i,\xi_i\geq 0 & i=1,\ldots,m.
\end{aligned}
\end{equation}
\end{linenomath}

\item[b)] Case $q=2$: SOCP problem
\begin{linenomath}
\begin{equation} \label{robust_q2}
\begin{aligned}
\min_{u,\gamma,\xi,s,r,t,v}  \quad & r-v+\nu \sum_{i=1}^m\xi_i\\
\text{s.t.} \quad & y^{(i)}\sum_{j=1}^m K_{ij}y^{(j)}u_j -\delta^{(i)}\sum_{j=1}^{m} \sqrt{K_{jj}} s_j \geq 1-\xi_i+y^{(i)}\gamma & i=1,\ldots,m\\
& t \geq \norm{u}_2 \\
& r+v=1 \\
& r \geq \sqrt{t^2+v^2} \\
& s_i \geq - u_i & i=1,\ldots,m\\
& s_i \geq u_i & i=1,\ldots,m\\
 & s_i,\xi_i\geq 0 & i=1,\ldots,m.
\end{aligned}
\end{equation}
\end{linenomath}

\item[c)] Case $q=\infty$: LP problem
\begin{linenomath}
\begin{equation} \label{robust_qinfty}
\begin{aligned}
\min_{u,\gamma,\xi,s,s_\infty}  \quad & s_\infty+\nu \sum_{i=1}^m\xi_i\\
\text{s.t.} \quad & y^{(i)}\sum_{j=1}^m K_{ij}y^{(j)}u_j -\delta^{(i)}\sum_{j=1}^{m} \sqrt{K_{jj}} s_j \geq 1-\xi_i+y^{(i)}\gamma & i=1,\ldots,m\\
& s_\infty \geq - u_i & i=1,\ldots,m\\
& s_\infty \geq u_i & i=1,\ldots,m\\
& s_i \geq - u_i & i=1,\ldots,m\\
& s_i \geq u_i & i=1,\ldots,m\\
& s_\infty \geq 0\\
 & s_i,\xi_i\geq 0 & i=1,\ldots,m.
\end{aligned}
\end{equation}
\end{linenomath}
\end{itemize}
\end{corollary}

As in the deterministic setting, once $u$, $\gamma$ and $\xi$ are obtained as solutions of model \eqref{rob_GSVM}, then $\omega_{\mathcal{A}}$ and $\omega_{\mathcal{B}}$ are computed according to formulas \eqref{omega_nonlinear}. Finally, the optimal separating hypersurface $S=(u,b)$ is derived, where $b$ is the optimal solution of the following robust counterpart of problem $\eqref{linear_search_GSVM}$:
\begin{linenomath}
\begin{equation} \label{linear_search_robust}
\begin{aligned}
\min_{b}  \quad & \sum_{i=1}^m \mathbbm{1}\bigg[\bigg(y^{(i)}b-y^{(i)}\sum_{j=1}^m K_{ij}y^{(j)}u_j+\delta^{(i)}\sum_{j=1}^{m} \sqrt{K_{jj}} \abs{u_j}\bigg)_i\bigg]\\
\text{s.t.} \quad & \gamma+1-\omega_{\mathcal{B}}\leq b \leq \gamma-1+\omega_{\mathcal{A}}.
\end{aligned}
\end{equation}
\end{linenomath}

When dealing with a multiclass classification task, the robust extension of model \eqref{GSVM_multiclass} for the $l$-th class is given by:
\begin{linenomath}
\begin{equation} \label{GSVM_robust_multiclass}
\begin{aligned}
\min_{u_l,\gamma_l,\xi_l}  \quad & \norm{u_l}_{q}^q+\nu \sum_{i=1}^m\xi_{l,i}\\
\text{s.t.} \quad & \widehat{y}^{(i)}_l\sum_{j=1}^m K_{ij}\widehat{y}^{(j)}_lu_{l,j} -\delta^{(i)}\sum_{j=1}^{m} \sqrt{K_{jj}} \abs{u_{l,j}} \geq 1-\xi_{l,i}+\widehat{y}^{(i)}_l\gamma_l & \quad i=1,\ldots,m\\
 & \xi_{l,i}\geq 0 & \quad  i=1,\ldots,m.
\end{aligned}
\end{equation}
\end{linenomath}

The optimal parameter $b_l$ is the solution of:
\begin{linenomath}
\begin{equation} \label{linear_search_robust_multiclass}
\begin{aligned}
\min_{b_l}  \quad & \sum_{i=1}^m \mathbbm{1}\bigg[\bigg(\widehat{y}^{(i)}_lb_l-\widehat{y}^{(i)}_l\sum_{j=1}^m K_{ij}\widehat{y}^{(j)}_lu_{l,j}+\delta^{(i)}\sum_{j=1}^{m} \sqrt{K_{jj}} \abs{u_{l,j}}\bigg)_i\bigg]\\
\text{s.t.} \quad & \gamma_l+1-\omega_{-l}\leq b_l \leq \gamma_l-1+\omega_{l}.
\end{aligned}
\end{equation}
\end{linenomath}

Since the structural form of the robust models \eqref{rob_GSVM} and \eqref{GSVM_robust_multiclass} is the same as their deterministic equivalent, the time complexity analysis provides analogous results.

For the sake of illustration, we depict in Figure \ref{fig_kernel2_robust} the kernel-induced decision boundaries of the robust model \eqref{robust_q1}, considering the same dataset of Figure \ref{fig_toy_2d_det}. The model is trained for both spherical (see Figure \ref{fig_kernel_RBF_rob_p2}) and box (see Figure \ref{fig_kernel_RBF_rob_pinf}) uncertainty sets.

\begin{figure}[h!]
     \centering
  \begin{subfigure}[b]{0.49\textwidth}
         \centering
         \includegraphics[width=\textwidth]{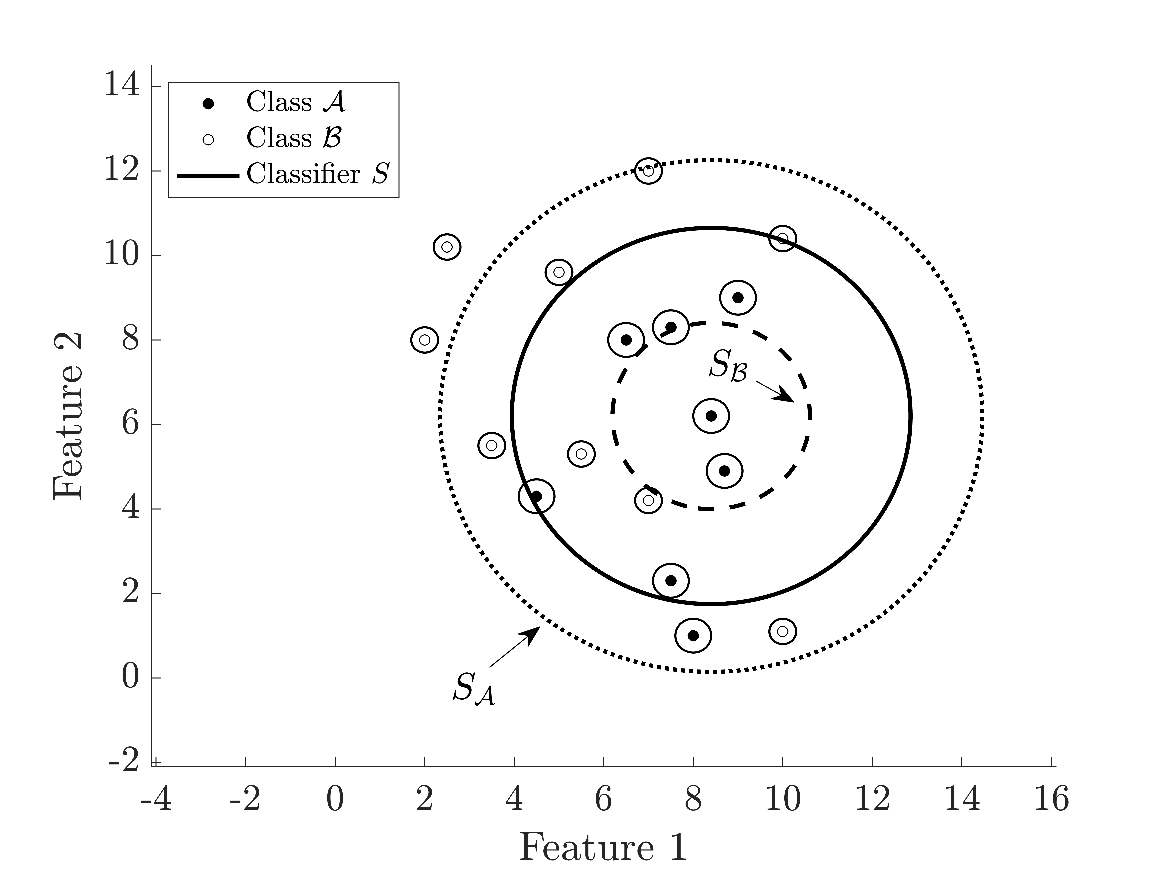}
         \caption{Gaussian RBF kernel ($\alpha=1.9$), $p=2$}
         \label{fig_kernel_RBF_rob_p2}
     \end{subfigure}
     \begin{subfigure}[b]{0.49\textwidth}
         \centering
\includegraphics[width=\textwidth]{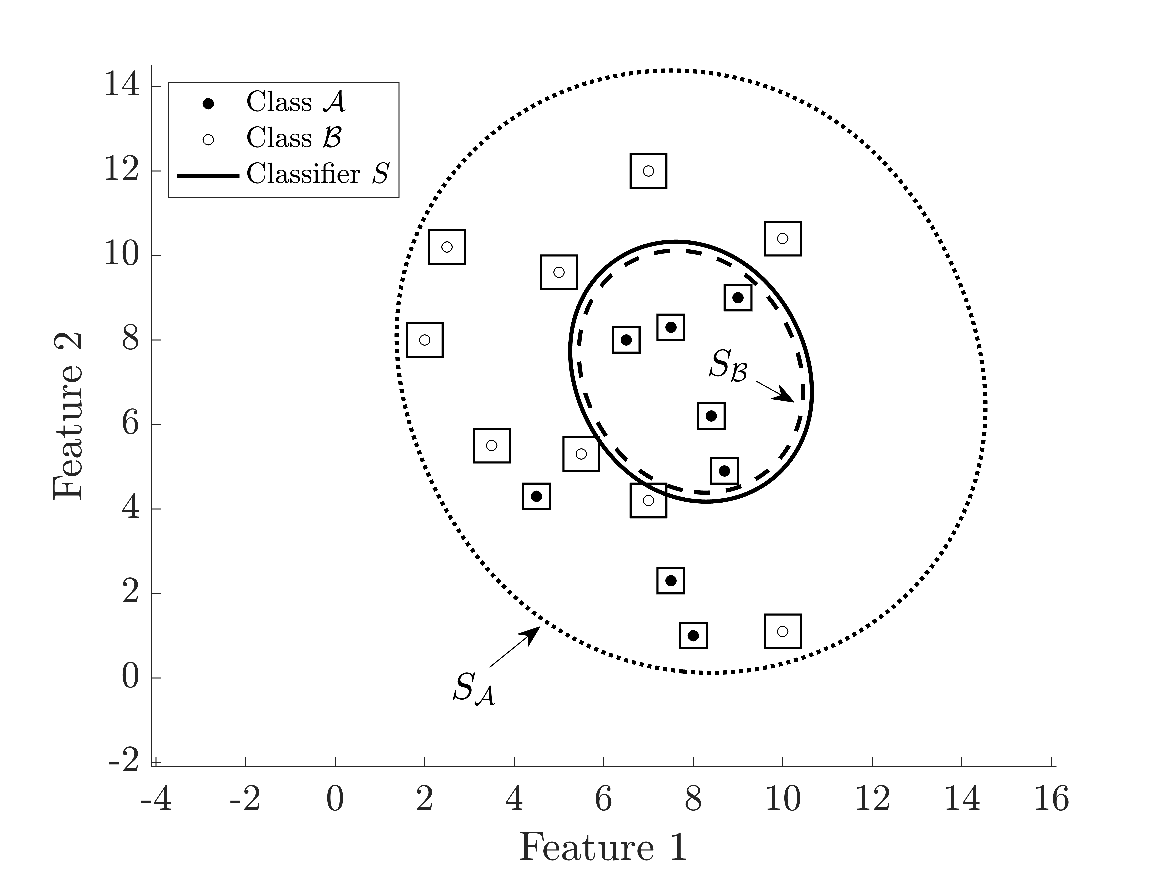}
         \caption{Gaussian RBF kernel ($\alpha=1.9$), $p=\infty$}
         \label{fig_kernel_RBF_rob_pinf}
     \end{subfigure}
          \caption{Separating surfaces obtained with Gaussian RBF kernel function from the robust model \eqref{robust_q1}. The $\ell_p$-norms defining the uncertainty sets are $p=2$ (on the left) and $p=\infty$ (on the right).}
        \label{fig_kernel2_robust}
\end{figure}

\section{Computational results} \label{sec_computational_results}
In this section, we evaluate the performance of the deterministic models presented in Section \ref{sec_novel_deterministic_model} and their robust counterparts of Section \ref{sec_robust_model} on a selection of 12 benchmark datasets taken from the UCI Machine Learning Repository (\cite{KelLonNot2023}). The models were implemented in MATLAB (v. 2021b) and solved using CVX (v. 2.2, see \cite{GraBoy2008,cvx2014}) and MOSEK solver (v. 9.1.9, see \cite{mosek}). All computational experiments were run on a MacBookPro17.1 with a chip Apple M1 of 8 cores and 16 GB of RAM memory. The MATLAB codes developed for the current proposal are made publicly available on GitHub (\url{https://github.com/aspinellibg/NonlinearSVM}).

The benchmark datasets are listed in the first column of Table \ref{tab_comp_res_7525}, along with the corresponding number of observations $m$ and of features $n$. In this study we examine 10 binary classification problems and 2 multiclass classification problems.

The experimental setting is as follows. Each dataset was split into \emph{training set}, composed by the $\beta\%$ of the observations, and \emph{testing set}, composed by the remaining $(100-\beta)\%$. We accounted for three different values of $\beta$, leading to the holdouts 75\%-25\%, 50\%-50\%, and 25\%-75\%. The partition was performed inline with the \emph{proportional random sampling} strategy (\cite{ChenTseYu2001}), meaning that the original class balance in the entire dataset was maintained in both training and testing set. Once the partition was complete, a kernel function $k(\cdot,\cdot)$ was chosen and the training set used to train the deterministic classifier for different values of input parameter $\nu$. Specifically, the deterministic formulation was solved on the training dataset through a grid-search strategy with five logarithmically spaced values of $\nu$ between $10^{-3}$ and $10^0$, and setting $N_{\max}=10^4$ as number of sub-intervals in the linear search procedure (see \cite{FacMagPot2022}). The optimal classifier was chosen among the five candidates as the one minimizing the misclassification error on the training set. Finally, the out-of-sample error on the testing set was computed, as the ratio between the total number of misclassified points in the testing set and its cardinality. In order to get stable results, the partition in training and testing set was performed 96 times in a \emph{repeated holdout} fashion (\cite{Kim2009}). The choice of this number is motivated by the use of the Parallel Computing Toolbox in MATLAB, since the code was parallelized on the 8 cores of the working machine. The final results were then averaged.

As in the original work of \cite{LiuPot2009} and in the robust linear extension presented in \cite{FacMagPot2022}, we considered $q=1$ in the objective function of the models. This choice provides a good compromise between structural risk minimization, related to the misclassification error, and parsimony since it automatically performs feature selection (\cite{LabMar-MerRod-Chi2019,LopMalCar2019,LeeYoonWon2022,LiaoDaiKuo2024}).

As far as it concerns the kernel function $k(\cdot,\cdot)$, we tested seven different alternatives: homogeneous linear ($d=1$, $c=0$), homogeneous quadratic ($d=2$, $c=0$), homogeneous cubic ($d=3$, $c=0$); inhomogeneous linear, inhomogeneous quadratic, inhomogeneous cubic; Gaussian RBF. For simplicity, parameter $\alpha$ in the Gaussian RBF kernel was set as the maximum value of the standard deviation across features for the dataset under consideration. Similarly for parameter $c$ in the inhomogeneous polynomial kernels.

Since models \eqref{GSVM}, \eqref{GSVM_multiclass} and their robust extensions \eqref{rob_GSVM}, \eqref{GSVM_robust_multiclass} are distance-based, imbalances in the order of magnitude of the features may result in distorted weights when classifying. For this reason, we considered \emph{min-max normalization} and \emph{standardization} as pre-processing techniques of data transformation (\cite{HanKamPei2011}). On one hand, in the min-max normalization each dataset was linearly scaled feature-wise into the $n$-dimensional hypercube $[0,1]^n$. On the other hand, in the standardization the values of a specific feature $j$, with $j=1,\ldots,n$, were normalized based on its mean $\mu_j$ and standard deviation $std_j$.

Among all the optimal deterministic classifiers found for each pair \textit{data transformation-kernel function}, the best configuration was chosen as the one minimizing the overall out-of-sample testing error. Within this choice of \textit{data transformation-kernel function}, the robust model was solved. The bounds $\eta^{(i)}$ on the perturbation vectors defining the uncertainty sets $\mathcal{U}_{p}(x^{(i)})$ were adjusted as:
\begin{linenomath}
\begin{eqnarray*}
\eta^{(i)}=\eta_{\mathcal{A}} := \rho_{\mathcal{A}} \max_{j=1,\ldots,n} std_{j,\mathcal{A}} \qquad \forall i:x^{(i)}\in\mathcal{A}
\\
\eta^{(i)}=\eta_{\mathcal{B}} := \rho_{\mathcal{B}} \max_{j=1,\ldots,n} std_{j,\mathcal{B}} \qquad \forall i:x^{(i)}\in\mathcal{B},
\end{eqnarray*}
\end{linenomath}
where $\rho_{\mathcal{A}}$ is a nonnegative parameter allowing the user to tailor the degree of conservatism and $\max_{j=1,\ldots,n} std_{j,\mathcal{A}}$ is the maximum standard deviation feature-wise for training points of class $\mathcal{A}$. Similarly for $\rho_{\mathcal{B}}$ and $\max_{j=1,\ldots,n} std_{j,\mathcal{B}}$. Once $\eta^{(i)}$ had been determined, the computation of the bound $\delta^{(i)}$ in the feature space was performed according to Propositions \ref{lemma_inhom_pol}-\ref{lemma_RBF_gaussian}. For simplicity, we set $\rho_{\mathcal{A}}=\rho_{\mathcal{B}}=\rho$, and considered 7 logarithmically spaced values for $\rho$ between $10^{-7}$ and $10^{-1}$. When the number of classes is greater than two, an analogous approach was applied class-wise. As in the deterministic setting, we averaged the out-of-sample testing errors for 96 random partitions of the dataset.

For each dataset, we report in Table \ref{tab_comp_res_7525} the best configuration \textit{data transformation-kernel function}, along with the average out-of-sample testing errors and standard deviations for the deterministic and robust models (holdout 75\% training set-25\% testing set). We considered three  types of uncertainty set, defined respectively by $\ell_1$-, $\ell_2$- and $\ell_{\infty}$-norm. Detailed results are reported in Tables \ref{tab_detailed_determ_7525}-\ref{tab_detailed_rob_2575_multiclass} in \ref{appendix_results}.

\begin{table}[h!]
\centering
\resizebox{\textwidth}{!}{
\begin{tabular}{llll *{3}l c}\toprule
Dataset & Data transformation & Kernel & Deterministic & \multicolumn{3}{c}{Robust} & Robust\\
$m\times n$& & & & $p=1$ & $p=2$ & $p=\infty$ & improvement ratio\\ \toprule
Arrhythmia  & $-$ & Gaussian RBF & $20.47\% \pm 0.07$ & $\underline{\mathbf{19.12\% \pm 0.08}}$ & $19.30\% \pm 0.07$ & $19.61\% \pm 0.07$ & {$6.60\%$}
\\
$68 \times 279$
\\
\hline
CPU time (s) & & & 0.289 & 0.290 & 0.288 & 0.295
\\
\toprule 
Parkinson & Min-max normalization & Hom. linear & $13.19\% \pm 0.03$ & $12.98\% \pm 0.03$ & $\underline{\mathbf{12.37\% \pm 0.03}}$ & $12.61\% \pm 0.04$ & {$6.22\%$}
\\
$195 \times 22$
\\
\hline
CPU time (s) & & & 3.626 & 3.421 & 3.454 & 3.418
\\
\toprule 
Heart Disease & Standardization & Inhom. linear & $17.48\% \pm 0.04$ & $16.84\% \pm 0.04$ & $17.53\% \pm 0.03$ & $\underline{\mathbf{16.36\% \pm 0.04}}$ & {$6.41\%$}
\\
$297 \times 13$
\\
\hline
CPU time (s) & & & 12.253 & 11.602 & 11.477 & 11.417
\\
\toprule 
Dermatology & $-$ & Inhom. quadratic & $1.64\% \pm 0.02$ & $1.65\% \pm 0.01$ & $1.57\% \pm 0.01$ & $\underline{\mathbf{0.55\% \pm 0.01}}$ & {$66.46\%$}
\\
$358 \times 34$
\\
\hline
CPU time (s) & & & 20.173 & 20.055 & 20.420 & 20.147
\\
\toprule 
Climate Model Crashes & $-$ & Hom. linear & $5.01\% \pm 0.02$ & $4.47\% \pm 0.02$ & $4.50\% \pm 0.01$ & $\underline{\mathbf{4.34\% \pm 0.01}}$ & {$13.37\%$}
\\
$540 \times 18$
\\
\hline
CPU time (s) & & & 68.069 & 66.762 & 67.169 & 67.381
\\
\toprule 
Breast Cancer Diagnostic & Min-max normalization & Inhom. quadratic & $3.02\% \pm 0.02$ & $2.63\% \pm 0.01$ & $2.65\% \pm 0.01$ & $\underline{\mathbf{{2.39\% \pm 0.01}}}$ & {$20.86\%$}
\\
$569 \times 30$
\\
\hline
CPU time (s) & & & 77.786 & 77.968 & 78.267 & 77.543
\\
\toprule 
Breast Cancer & Standardization & Hom. linear & $3.17\% \pm 0.01$ & $\underline{\mathbf{2.97\% \pm 0.01}}$ & $3.07\% \pm 0.01$ & $3.06\% \pm 0.01$ & {$6.31\%$}
\\
$683 \times 9$
\\
\hline
CPU time (s) & & & 135.765 & 135.651 & 137.039 & 136.286
\\
\toprule 
Blood Transfusion & Standardization & Inhom. cubic & $20.72\% \pm 0.02$ & $20.60\% \pm 0.02$ & $\underline{\mathbf{20.55\% \pm 0.02}}$ & $20.64\% \pm 0.02$ & {$0.82\%$}
\\
$748 \times 4$
\\
\hline
CPU time (s) & & & 178.136 & 178.751 & 179.682 & 180.083
\\
\toprule 
Mammographic Mass & Standardization & Inhom. quadratic & $15.71\% \pm 0.02$ & $15.49\% \pm 0.02$ & $\underline{\mathbf{15.42\% \pm 0.02}}$ & $15.54\% \pm 0.02$ & {$1.85\%$}
\\
$830 \times 5$
\\
\hline
CPU time (s) & & & 241.205 & 241.810 & 242.614 & 241.929
\\
\toprule 
{Qsar Biodegradation} & {Min-max normalization} & {Gaussian RBF} & {$12.88\% \pm 0.02$} & {$\underline{\mathbf{11.78\% \pm 0.01}}$} & {$12.72\% \pm 0.02$} & {$12.86\% \pm 0.02$} & {$8.54\%$}
\\
{$1055 \times 41$}
\\
\hline
{CPU time (s)} & & & {484.908} & {498.235} & {495.073} & {491.748}
\\
\toprule[1.8pt] 
{Iris} & {$-$} & {Gaussian RBF} & {$3.10\% \pm 0.03$} & {$3.07\% \pm 0.03$} & {$3.21\% \pm 0.03$} & {$\underline{\mathbf{2.87\% \pm 0.03}}$} & {$7.42\%$}
\\
{$150 \times 4$ (3 classes)} 
\\
\hline
{CPU time (s)} & & & {5.391} & {5.684} & {5.627} & {5.604}
\\
\toprule 
{Wine} & {Standardization} & {Inhom. linear} & {$2.77\% \pm 0.02$} & {$2.63\% \pm 0.02$} & {$2.63\% \pm 0.02$} & {$\underline{\mathbf{2.51\% \pm 0.02}}$} & {$9.39\%$}
\\
{$178 \times 13$ (3 classes)}
\\
\hline
CPU time (s) & & & {7.916} & {8.361} & {8.352} & {8.605}
\\
\bottomrule
\end{tabular}
}
\caption{Average out-of-sample testing errors and standard deviations over 96 runs for the deterministic and robust models. Best results are highlighted. The last column displays the robust improvement ratios over the deterministic counterparts. Holdout: 75\% training set-25\% testing set.} \label{tab_comp_res_7525}
\end{table}

We notice that all the considered robust formulations outperform the corresponding deterministic models. In 6 out of 12 datasets the best results are achieved by the box robust formulation ($p=\infty$). Since box uncertainty sets are the widest around data points, this implies that the proposed formulations benefit from a more conservative approach when treating uncertainties. The last column of Table \ref{tab_comp_res_7525} shows the robust \textit{Improvement Ratio} (IR) over the deterministic counterpart. The IR was computed as in \cite{FacMagPot2022} and according to the following formula:
\begin{equation*}
\text{IR}:=\frac{\tau^{\text{det}}-\tau^{\text{rob*}}}{\tau^{\text{det}}},
\end{equation*}
where $\tau^{\text{det}}$ and $\tau^{\text{rob*}}$ are the average out-of-sample testing errors of the deterministic and the best robust performing model, respectively. The results on the IR further confirm that robust methods provide superior accuracy when the uncertainty is handled in the classification process. Extensive results on the improvement ratio are reported in Table \ref{tab_impr_ratio} in \ref{appendix_results}.

For the sake of completeness, we explore in details the performance of the proposed models when classifying datasets ``Parkinson'' and ``Breast Cancer Diagnostic''. First of all, we discuss the results of the deterministic approach, with respect to both data transformation and kernel function. The out-of-sample testing errors for the holdout 75\%-25\% are depicted in Figure \ref{fig_parkbreast_datatransformation_kernel}, while detailed results are reported in Table \ref{tab_detailed_determ_7525} in \ref{appendix_results}. We note that the worst performance occurs when no data transformations are applied. Conversely, min-max normalization and standardization provide good and comparable results. Similar conclusions can be drawn for holdouts 50\%-50\% and 25\%-75\% (see Tables \ref{tab_detailed_determ_5050}-\ref{tab_detailed_determ_2575} in \ref{appendix_results}).

\begin{figure}[h!]
     \centering
     \begin{subfigure}[b]{0.496\textwidth}
         \centering
         \includegraphics[width=\textwidth]{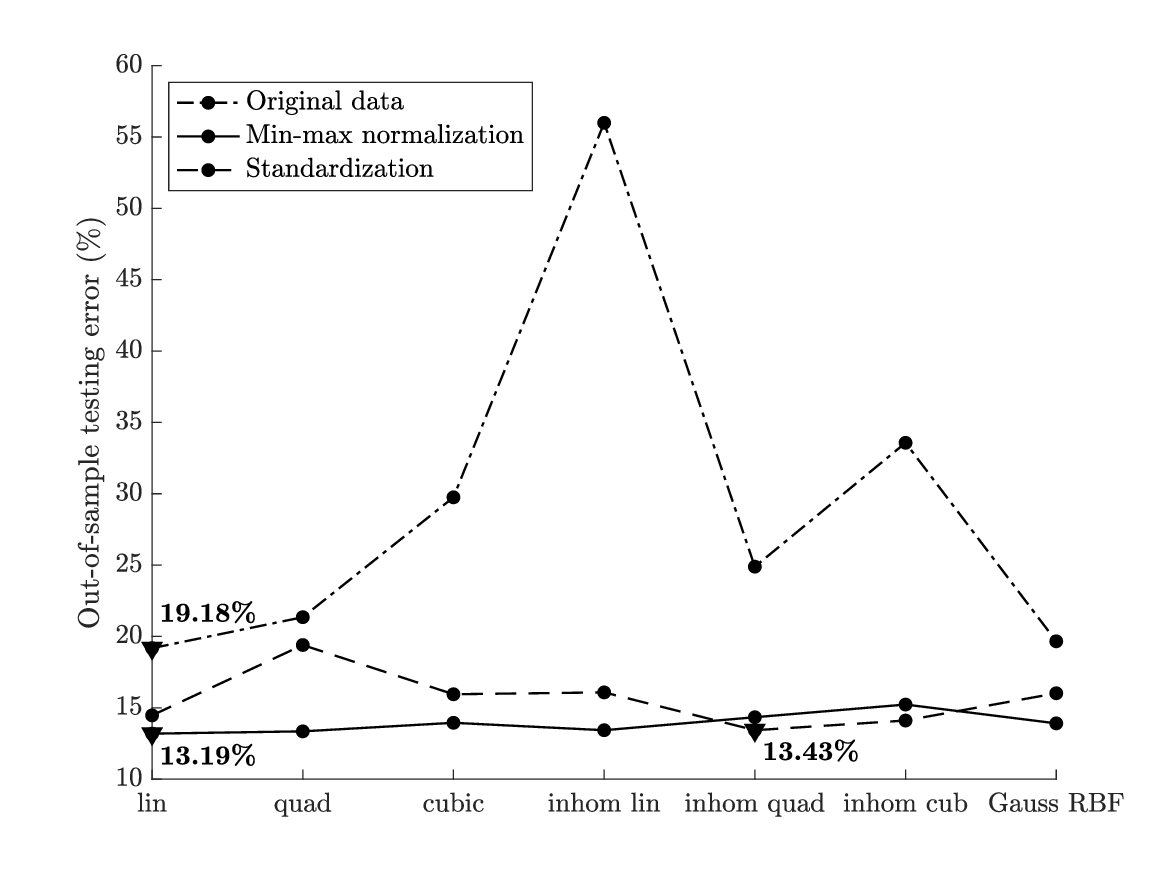}
         \caption{Parkinson.}
         \label{fig_parkinson_determ}
     \end{subfigure}
     \hfill
     \begin{subfigure}[b]{0.496\textwidth}
         \centering
         \includegraphics[width=\textwidth]{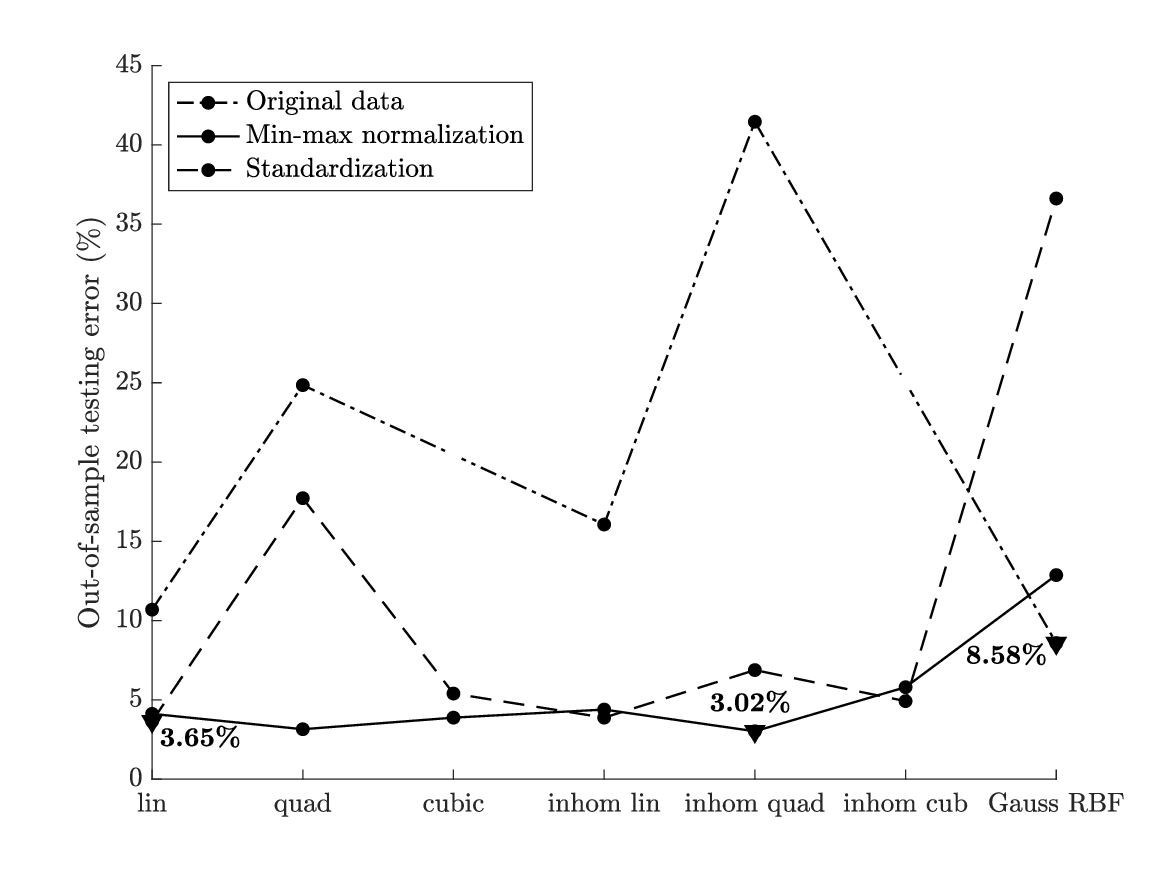}
         \caption{Breast Cancer Diagnostic.}
         \label{fig_breastcancerdiagn_determ}
     \end{subfigure}
          \caption{Out-of-sample testing error of the deterministic formulation applied to the datasets ``Parkinson'' and ``Breast Cancer Diagnostic''. Each triangle represents the lowest error for the corresponding data transformation technique. Holdout: 75\% training set-25\% testing set.}
        \label{fig_parkbreast_datatransformation_kernel}
\end{figure}

In order to evaluate the performance of the robust model, we consider 60 logarithmically spaced values of $\rho$ between $10^{-7}$ and $10^{-1}$. The results are depicted in Figure \ref{fig_parkinson_robust}. We notice that increasing the value of $\beta$ leads to better performance in terms of the overall out-of-sample testing error (see Figures \ref{fig_parkinson_robust_all}, \ref{fig_breastcancerdiagnostic_robust_all}), since more data points in the training set are available as input of the optimization model. In addition, when perturbations are included in the model, the performance improves with respect to the deterministic case. Indeed, the great majority of the points lies below the corresponding horizontal line, representing the out-of-sample testing error of the deterministic classifier. Interestingly, the increase of the uncertainty impacts differently on the two classes (see Figures \ref{fig_parkinson_robust_A_B}, \ref{fig_breastcancerdiagnostic_robust_A_B}). For instance, the ``Breast Cancer Diagnostic'' dataset is not able to bear high levels of uncertainty ($\rho>10^{-3}$) since all data points of class $\mathcal{A}$, representing patients with a malignant tumor, are misclassified. On the other hand, all observations in class $\mathcal{B}$ (patients with a benign tumor) are assigned to the correct category. From a practical perspective, given that classifying people with a malignant tumor as people with a benign tumor is worse than the opposite, robust models with low degree of perturbation should be considered in this case.

\begin{figure}[h!]
     \centering
     \begin{subfigure}[b]{0.496\textwidth}
         \centering
         \includegraphics[width=\textwidth]{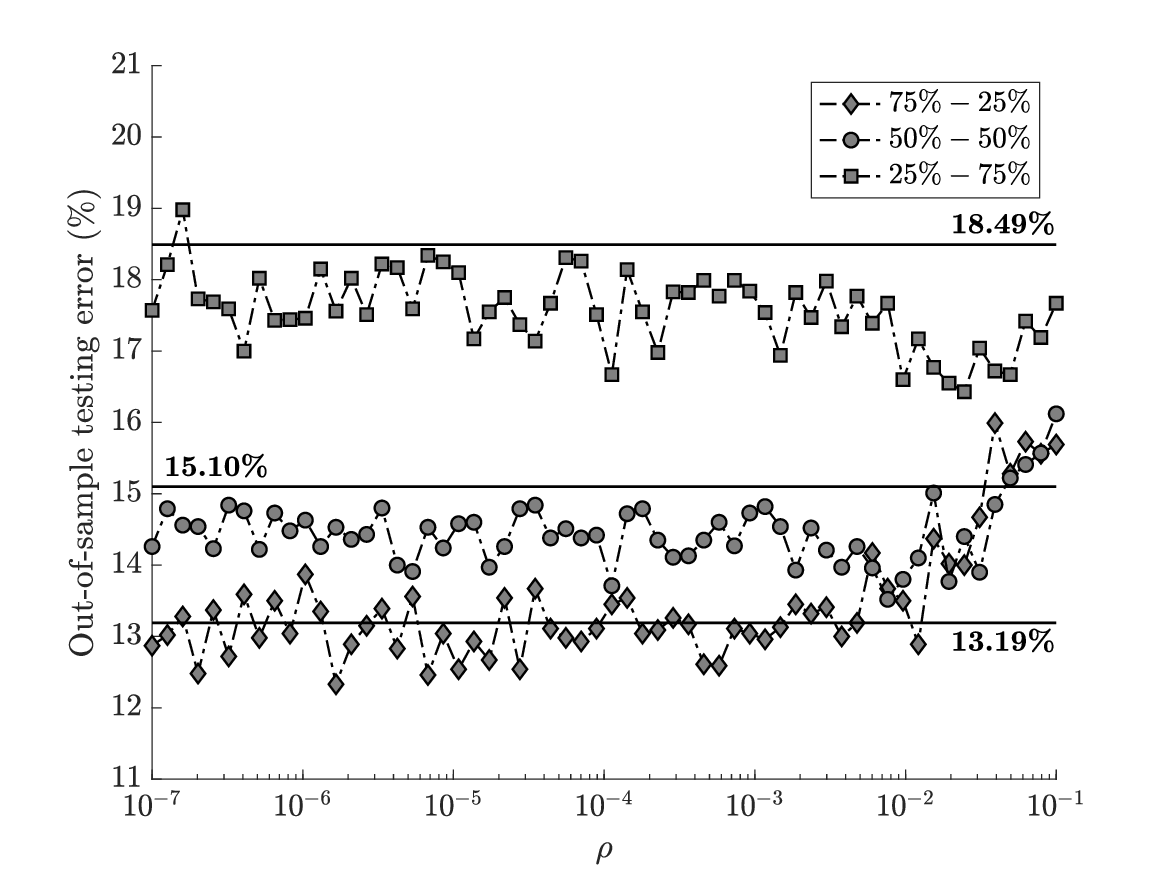}
         \caption{Overall results (Parkinson).}
         \label{fig_parkinson_robust_all}
     \end{subfigure}
     \hfill
     \begin{subfigure}[b]{0.496\textwidth}
         \centering
         \includegraphics[width=\textwidth]{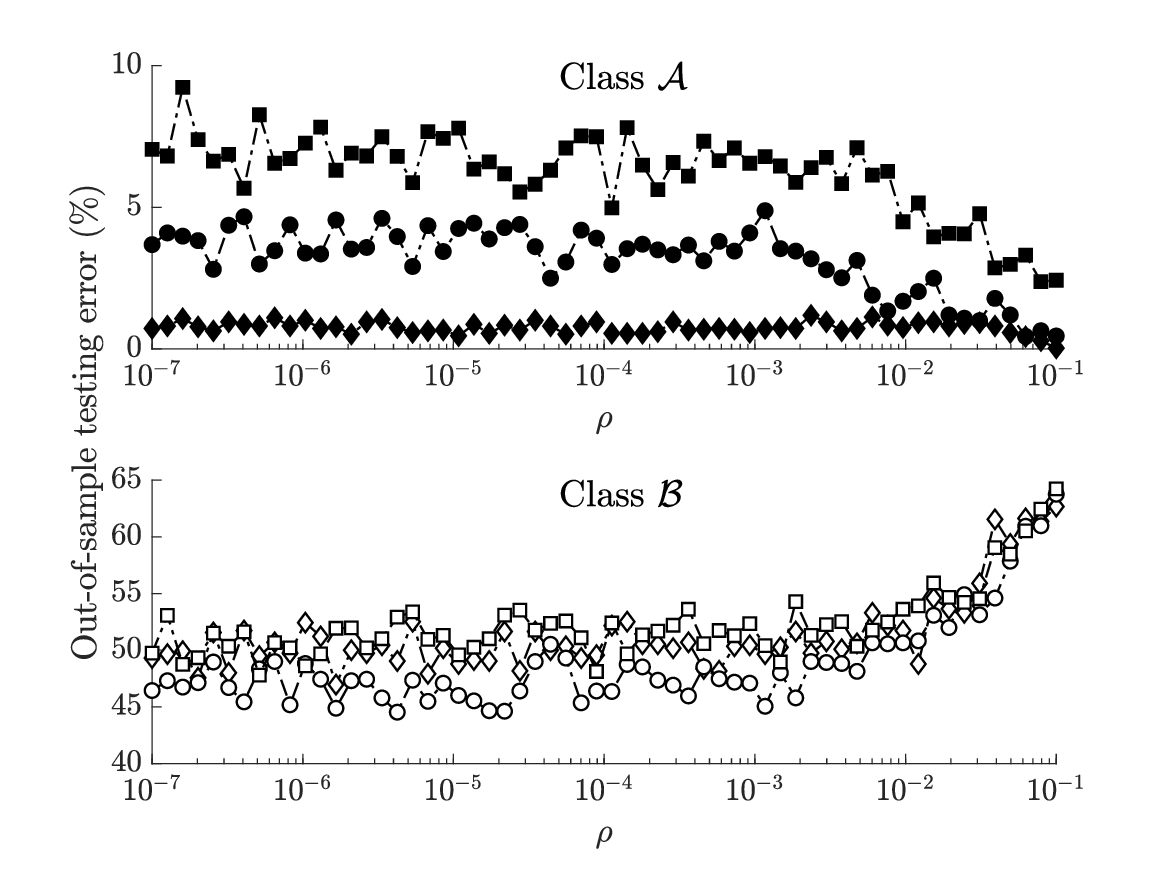}
         \caption{Results divided by class. (Parkinson).}
         \label{fig_parkinson_robust_A_B}
     \end{subfigure}
          \begin{subfigure}[b]{0.496\textwidth}
         \centering
         \includegraphics[width=\textwidth]{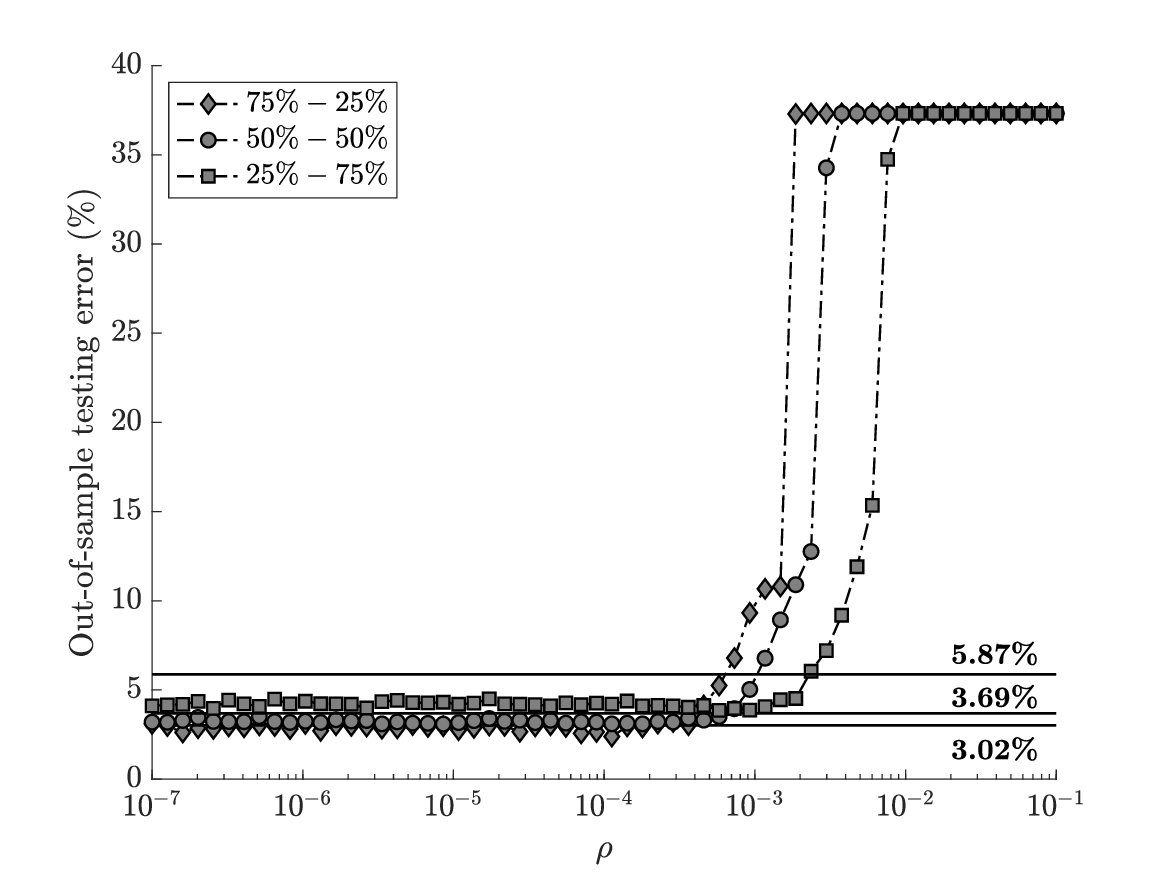}
         \caption{Overall results. (Breast Cancer Diagnostic).}
         \label{fig_breastcancerdiagnostic_robust_all}
     \end{subfigure}
     \hfill
     \begin{subfigure}[b]{0.496\textwidth}
         \centering
         \includegraphics[width=\textwidth]{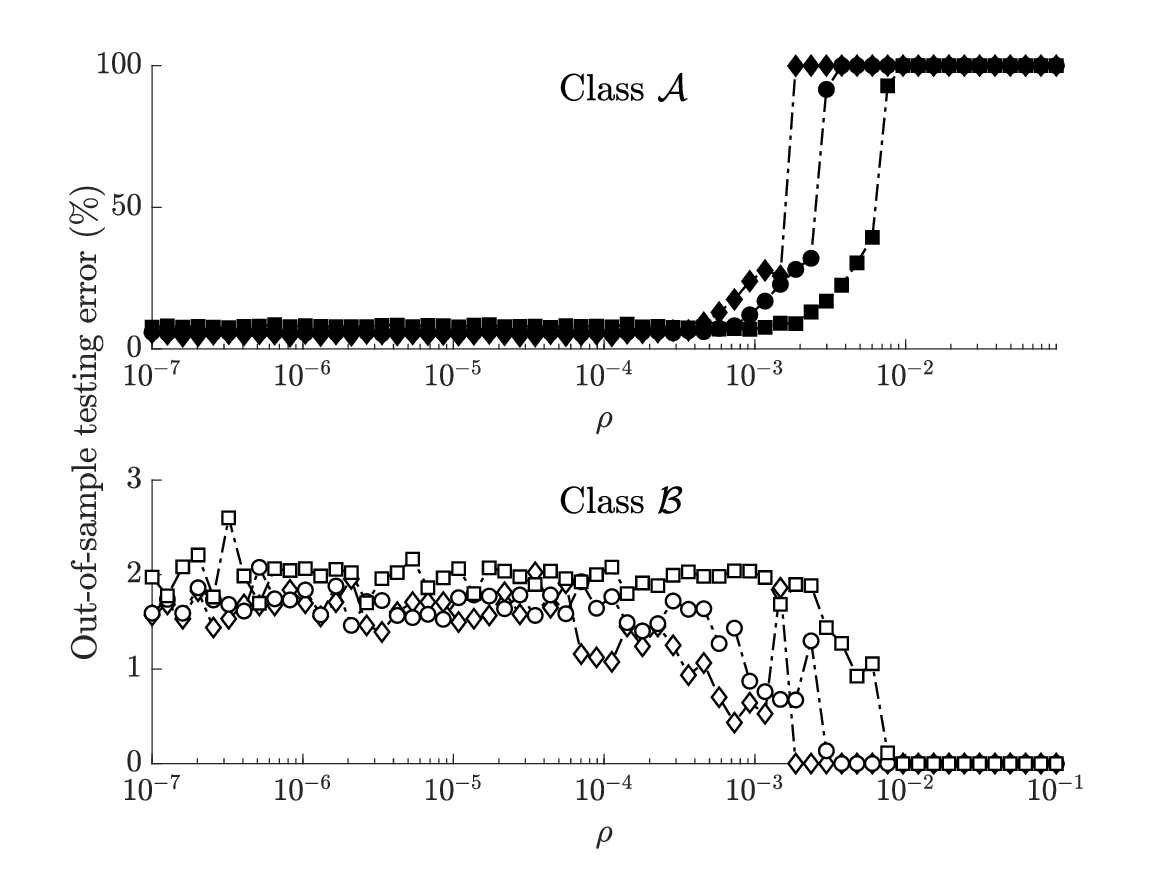}
         \caption{Results divided by class. (Breast Cancer Diagnostic).}
         \label{fig_breastcancerdiagnostic_robust_A_B}
     \end{subfigure}
          \caption{Out-of-sample testing error of the robust formulation applied to the datasets ``Parkinson'' and ``Breast Cancer Diagnostic''. Overall results are on the left, with the performance of the deterministic classifier depicted as horizontal line for each holdout. Results divided by class are on the right. The values of $\rho$ are in logarithmic scale.}
        \label{fig_parkinson_robust}
\end{figure}

In Table \ref{tab_comparison_scikitlearng} we report a comparison between the best results of Table \ref{tab_comp_res_7525} and the out-of-sample testing errors provided by the SVM classifier of \emph{scikit-learn}, a popular ML library implemented in Python (\cite{scikitlearn}). We tested the seven different kernels and reported in column 5 the best choice in terms of the lowest out-of-sample testing error. From column 6, it can be noted that in 8 out of 10 datasets the formulation proposed in this study outperforms the one implemented in the scikit-learn library for SVM.

\begin{table}[h!]
\centering
\resizebox{\textwidth}{!}{
\begin{tabular}{ll| *{2}l | *{2}l}\toprule
{Dataset} & {Data transformation} & \multicolumn{2}{c|}{{Table \ref{tab_comp_res_7525}}} & \multicolumn{2}{c}{{Scikit-learn SVM library}}
\\
& & {Best kernel} & {Result} & {Best kernel} & {Result}
\\
\hline
{Arrhythmia} & {$-$} & {Gaussian RBF} & {$\underline{\mathbf{19.12\% \pm 0.08}}$} & {Gaussian RBF} & {$19.48\% \pm 0.07$}
\\
{Parkinson} & {Min-max normalization} & {Hom. linear} & {$12.37\% \pm 0.03$} & {Inhom. cubic} & {$\underline{\mathbf{9.41\% \pm 0.04}}$}
\\
{Heart Disease} & {Standardization} & {Inhom. linear} & {$\underline{\mathbf{16.36\% \pm 0.04}}$} & {Inhom. linear} & {$16.63\% \pm 0.04$}
\\
{Dermatology} & {$-$} & {Inhom. quadratic} & {$0.55\% \pm 0.01$} & {Inhom. linear} & {$\underline{\mathbf{0.11\% \pm 0.01}}$}
\\
{Climate Model Crashes} & {$-$} & {Hom. linear} & {$\underline{\mathbf{4.34\% \pm 0.01}}$} & {Inhom. linear} & {$4.78\% \pm 0.01$}
\\
{Breast Cancer Diagnostic} & {Min-max normalization} & {Inhom. quadratic} & {$\underline{\mathbf{2.39\% \pm 0.01}}$} & {Hom. cubic} & {$2.78\% \pm 0.01$}
\\
{Breast Cancer} & {Standardization} & Hom. linear & {$\underline{\mathbf{2.97\% \pm 0.01}}$} & {Gaussian RBF} & {$3.04\% \pm 0.01$}
\\
{Blood Transfusion} & {Standardization} & {Inhom. cubic} & {$\underline{\mathbf{20.55\% \pm 0.02}}$} & {Inhom. cubic} & {$21.65\% \pm 0.02$}
\\
{Mammographic Mass}  & {Standardization} & {Inhom. quadratic} & {$\underline{\mathbf{15.42\% \pm 0.02}}$} & {Inhom. quadratic} & {$16.05\% \pm 0.02$}
\\
{Qsar Biodegradation} & {Min-max normalization} & {Gaussian RBF} & {$\underline{\mathbf{11.78\% \pm 0.01}}$} & {Inhom. quadratic} & {$12.57\% \pm 0.02$}
\\
\bottomrule
\end{tabular}}
\caption{Out-of-sample testing error comparison among best results of Table \ref{tab_comp_res_7525} and simulations from the scikit-learn SVM library (\cite{scikitlearn}). The lowest out-of-sample testing error within a dataset is highlighted.} \label{tab_comparison_scikitlearng}
\end{table}

In addition, we compare the performance of our proposal with results from other SVM formulations present in the ML literature (see Table \ref{tab_comparison_FacciniBertsimas}). Specifically, as deterministic models we consider the linear classifiers proposed in \cite{LiuPot2009}, \cite{BerDunPawZhu2019}, and \cite{JayKheCha2007}, as well as the kernelized TWSVM classifier from \cite{JayKheCha2007}. For all of these models, we tuned the hyperparameter in the objective function using the same grid-search strategy employed in this paper. Following \cite{Peng2011}, to prevent issues related to ill-conditioning, we included a regularization term in the objective function of the kernelized TWSVM approach (see \cite{Suman2018} for further details on the MATLAB implementation). Finally, our robust formulation was compared with the robust classifiers from \cite{FacMagPot2022} and \cite{BerDunPawZhu2019}. As shown in Table \ref{tab_comparison_FacciniBertsimas_DET}, in 5 out of 10 datasets the results of our deterministic classifiers outperform the other methods. Consequently, the linear approaches benefit from a generalization towards nonlinear classifier. Table \ref{tab_comparison_FacciniBertsimas_ROB} further shows that our robust formulation achieves even better accuracy in most of the cases.

\begin{table}[h!]
  \begin{subtable}[h]{\textwidth}
\centering
\resizebox{\textwidth}{!}{
\begin{tabular}{l| *{5}l }\toprule
Dataset & \multicolumn{5}{c}{SVM classifier}
\\
& This paper & Linear & Linear & Linear TWSVM & Kernelized TWSVM
\\
& & \cite{LiuPot2009} & \cite{BerDunPawZhu2019} & \cite{JayKheCha2007} & \cite{JayKheCha2007}
\\
\hline
Arrhythmia & $20.47\%$ & $25.65\%$ & $43.08\%$ & $\underline{20.34\%}$ & $24.33\%$
\\
Parkinson & $\underline{13.19\%}$ & $14.13\%$ & $14.36\%$ & $16.10\%$ & $15.71\%$
\\
Heart Disease & $17.48\%$ & $16.68\%$ & $\underline{15.93\%}$ & $16.96\%$ & $16.31\%$
\\
Dermatology & $1.64\%$ & $0.56\%$ & $3.38\%$ & $1.12\%$ & $\underline{0.18\%}$
\\
Climate Model Crashes & $5.01\%$ & $\underline{4.99\%}$ & $5.00\%$ & $13.67\%$ & $5.92\%$
\\
Breast Cancer Diagnostic & $\underline{3.02\%}$ & $4.89\%$ & $6.49\%$ & $3.62\%$ & $4.50\%$
\\
Breast Cancer & $\underline{3.17\%}$ & $3.49\%$ & $5.00\%$ & $4.08\%$ & $4.00\%$
\\
Blood Transfusion & $\underline{20.72\%}$ & $23.49\%$ & $23.62\%$ & $37.12\%$ & $23.22\%$
\\
Mammographic Mass & $\underline{15.71\%}$ & $-$ & $18.07\%$ & $17.32\%$ & $17.92\%$
\\
Qsar Biodegradation & $12.88\%$ & $-$ & $\underline{12.51\%}$ & $14.69\%$ & $13.24\%$
\\
\bottomrule
\end{tabular}
}
\caption{Deterministic formulations.}
\label{tab_comparison_FacciniBertsimas_DET}
\end{subtable}
\vfill \vspace*{0.3cm}
    \begin{subtable}[h]{\textwidth}
        \centering
        \resizebox{0.65\textwidth}{!}{
        \begin{tabular}{l| *{3}l}\toprule
Dataset & \multicolumn{3}{c}{SVM classifier}
\\
& This paper & Robust linear & Robust linear
\\
& & \cite{FacMagPot2022} & \cite{BerDunPawZhu2019}
\\
\hline
Arrhythmia & $\underline{19.12\%}$ & $23.00\%$ & $29.23\%$
\\
Parkinson & $\underline{12.37\%}$ & $13.00\%$ & $16.41\%$
\\
Heart Disease  & $16.36\%$ & $\underline{16.20\%}$ & $16.61\%$
\\
Dermatology & $0.55\%$ & $\underline{0.13\%}$ & $1.13\%$
\\
Climate Model Crashes  & $4.34\%$ & $4.34\%$ & $\underline{4.07\%}$
\\
Breast Cancer Diagnostic & $\underline{2.39\%}$ & $3.89\%$ & $4.04\%$
\\
Breast Cancer & $\underline{2.97\%}$ & $3.12\%$ & $4.26\%$
\\
Blood Transfusion & $\underline{20.55\%}$ & $22.55\% $ & $23.62\%$
\\
Mammographic Mass & $\underline{15.42\%}$ & $-$ & $19.28\%$
\\
{Qsar Biodegradation} &  $\underline{11.78\%}$ & $-$ & $12.42\%$
\\
\bottomrule
\end{tabular}
}
        \caption{Robust formulations.}
\label{tab_comparison_FacciniBertsimas_ROB}
\end{subtable}
   
\caption{Out-of-sample testing error comparison among deterministic and robust results obtained from SVM formulations in the literature. For each approach and dataset, the best result is underlined.} \label{tab_comparison_FacciniBertsimas}
\end{table}

To assess the good performance of the proposed approach over the other methods, we applied the Friedman test and the Holm test (\cite{Dem2006}). First of all, we computed the average rank $R_j$ for each of the methods on the basis of the out-of-sample testing error (see columns 2 and 4 in Table \ref{tab_holm_test}). Then, the Friedman test with Iman-Davenport correction is applied to verify whether such ranks are statistically similar (null hypothesis). The statistic $F_F$ associated with the test is given by:
\begin{linenomath}
\begin{equation*}
F_F=\frac{(N_d-1)\chi^2_F}{N_d(N_m-1)-\chi^2_F}, \quad \text{with} \quad \chi^2_F = \frac{12N_d}{N_m(N_m+1)}\bigg[\sum_{j=1}^{N_m} R_j^2-\frac{N_m(N_m+1)^2}{4}\bigg],
\end{equation*}
\end{linenomath}
where $N_d=8$ is the number of datasets (we excluded ``Mammographic Mass'' and ``Qsar Biodegradation'' since they were not considered in \cite{FacMagPot2022}) and $N_m$ is the number of methodologies (5 for the deterministic and 3 for the robust). Under the null hypothesis, $F_F$ is distributed according to the $F$-distribution with $N_m-1$ and $(N_m-1)(N_d-1)$ degrees of freedom. In our case, the $p$-values associated with the Friedman test are $0.243$ and $0.014$ for the deterministic and robust approach, respectively. This implies that for the robust classifiers the null hypothesis of equal ranks is rejected with a significance level lower than $\alpha_R=5\%$. Since such hypothesis does not hold, we performed pairwise comparisons between the robust classifier with the highest rank $R^*$ and those remaining. To this extent, we considered the Holm test (\cite{Dem2006}) whose statistic $z_j$ for comparing the best classifier with the $j$-th one is computed as:
\begin{linenomath}
\begin{equation*}
z_j=(R^*-R_j)\sqrt{\frac{6N_d}{N_m(N_m+1)}}.
\end{equation*}
\end{linenomath}
Under the null hypothesis of outperformance of the best method over the others, the test statistic is distributed as a standard normal distribution. The results of the Holm test are presented in Table \ref{tab_holm_test} (see columns 3-7). The null hypothesis is rejected when the $p$-value of the test is below the significance thresholds of column 6. It can be seen that the proposed model achieves the highest rank in both the deterministic and robust formulation, outperforming the robust linear SVM approach presented in \cite{BerDunPawZhu2019}. On the other hand, there are no statistically significant differences between our proposal and the robust method devised in \cite{FacMagPot2022}, even if in most cases the results confirm the good performance of the proposed methodology (see Table \ref{tab_comparison_FacciniBertsimas_ROB}).

\begin{table}[h!]
\centering
\resizebox{\textwidth}{!}{
\begin{tabular}{ll|lllll}\toprule
\multicolumn{2}{c|}{{Deterministic formulation}} & \multicolumn{5}{c}{{Robust formulation}} 
\\
{SVM classifier} & {Mean rank} & {SVM classifier} & {Mean rank} & {$p$-value} & {$\alpha_R/(j-1)$} & {Action}
\\
\hline
{This paper} & {2.250} & {This paper} & {1.438} & {-} & {-} & {-}
\\
{\cite{FacMagPot2022}} & {2.625} & {\cite{FacMagPot2022}} & {1.813} & {0.453} & {0.050} & {Not reject}
\\
{Kernelized TWSVM \cite{JayKheCha2007}} & {2.750} & {\cite{BerDunPawZhu2019}} & {2.750} & {0.009} & {0.025} & {Reject}
\\
{Linear TWSVM \cite{JayKheCha2007}} & {3.625}
\\
{\cite{BerDunPawZhu2019}} & {3.750}
\\
\bottomrule
\end{tabular}
}
\caption{{Mean ranks of the deterministic formulations (columns 1-2). Holm test for pairwise comparison of robust formulations, with $\alpha_R=0.05$ and $j=2,3$ (columns 3-7).}} \label{tab_holm_test}
\end{table}

From Table \ref{tab_comp_res_7525} it can be noticed that the choice of the best data transformation method strongly depends on the dataset. In order to guide the final user among the three possible techniques, we report in Table \ref{tab_data_transformation} in \ref{appendix_results} summary statistics on the 10 datasets deployed for binary classification task. Specifically, for each feature we compute the mean and the corresponding coefficient of variation, defined as the ratio between the standard deviation and the mean. In Table \ref{tab_data_transformation} we list the minimum and the maximum values of the two considered indices for each dataset, along with the corresponding best data transformation. We argue that, whenever the values of the observations are close, the best approach is to classify the original data without any transformation (see datasets \virg{Arrhythmia}, \virg{Dermatology} and \virg{Climate Model Crashes}). In the extreme case of constant features, pre-processing techniques of data transformation cannot be applied (see dataset \virg{Arrhythmia}). On the other hand, the min-max normalization is a suitable choice when the order of magnitude across the features varies a lot. For instance, in datasets \virg{Parkinson} and \virg{Breast Cancer Diagnostic} there are 7 and 5 orders of magnitude of difference between the minimum and the maximum value of the mean of the features, respectively. Finally, standardization is an appropriate technique in all other cases, where no significant differences occur among the orders of magnitude of the features (see datasets \virg{Heart Disease}, \virg{Breast Cancer}, \virg{Blood Transfusion} and \virg{Mammographic Mass}).

Finally, numerical results show that the computational time is significantly high for datasets with a large number of observations, especially when considering 75\% of the instances as training set (see Table \ref{tab_detailed_determ_7525} in \ref{appendix_results}). The performing speed benefits from a reduction of $\beta$, even if at the cost of worsening the accuracy. Nevertheless, when datasets are equally split in training and testing set, the out-of-sample testing error does not increase significantly if compared to the holdout 75\%-25\% (see Table \ref{tab_detailed_determ_5050}). A similar conclusion can be drawn for the robust model (see Tables \ref{tab_detailed_rob_7525}-\ref{tab_detailed_rob_5050_cont}). Conversely, from the time complexity analysis, it should be noticed that the number $N_{\max}$ of sub-intervals chosen to solve problem \eqref{linear_search_GSVM} and its variants impacts on the overall computational time especially when the number of observations is not significantly high. Therefore, the final user should properly choose the values of $\beta$ and $N_{\max}$ to guarantee high accuracy in a reasonable time.

\section{Conclusions} \label{sec_conclusions}
In this paper, we have proposed novel optimization models for solving binary and multiclass classification tasks through a Support Vector Machine (SVM) approach. From a methodological perspective, we have extended the techniques presented in \cite{LiuPot2009,FacMagPot2022} to the nonlinear context through the introduction of kernel functions. Data are mapped from the input space to the feature space where a first classification via kernelized SVM is performed. The optimal classifier is then constructed as the solution of a linear search procedure aiming to minimize the overall misclassification error.

Motivated by the uncertain nature of real-world data, we have adopted a Robust Optimization (RO) approach by constructing around each training data a bounded-by-$\ell_p$-norm uncertainty set, with $p\in[1,\infty]$. Perturbation propagates from the input space to the feature space through the feature map associated with the kernel function. To face this problem, we have rigorously derived closed-form expressions for the uncertainty set bounds in the feature space, extending the results present in the literature. Thanks to this, we have formulated the robust counterpart of the deterministic models in the case of nonlinear classifiers. To enhance generalization, in all the proposed formulations we have considered a $\ell_q$-norm with $q\in[1,\infty]$ as measure of the SVM-margin. Since the resulting robust problem turns to be convex but nonlinear, we have proved that in specific cases it can be reformulated as a LP or a SOCP problem, with clear advantages in terms of computational efficiency.

The proposed models have been tested on real-world datasets, considering different combinations of data transformations and kernel functions. The results show that our robust formulation outperforms other linear and kernelized SVM approaches in most cases. This has been confirmed by classical statistical tests deployed to compare the performance of machine learning techniques. Overall, the models benefit from including uncertainty in the training process. The accuracy is clearly affected by the choice of the kernel function and of the data transformation before training. Therefore, we have provided insights to guide the final user in choosing the best configuration.

Regarding future advancements, various streams of research can originate from this work. First of all, extend the approach to handle uncertainties in the labels of training data. This could increase the generalization capability of the models. Additionally, in this work we have followed the classical RO approach of including uncertainty during the training phase (see, for instance, \cite{BerDunPawZhu2019}). It could be noteworthy to consider perturbations both in the training and in the testing sets. However, this choice increases the complexity of the models and novel measures to quantify the accuracy have to be devised, since it is not obvious how to classify an entire uncertainty set in one class or another as opposed to the case of single data point. The main limitation of the current proposal is the complexity of the two-step procedure, leading to a time-consuming process. Further techniques could be employed to speed up the approach and to optimize the tuning phase of the model's parameters (see, for example, the Bayesian optimization in \cite{SnoLarAda2012}). Finally, different methodologies could be applied to further robustify the models. For instance, Chance-Constrained Programming and Distributionally Robust Optimization with ambiguity sets defined by moments, phi-divergences or Wasserstein distance merit further research too.

\section*{Acknowledgements}
This work has been supported by ``ULTRA OPTYMAL - Urban Logistics and sustainable TRAnsportation: OPtimization under uncertainTY and MAchine Learning'', a PRIN2020 project funded by the Italian University and Research Ministry (grant number 20207C8T9M).

This study was also carried out within the MOST - Sustainable Mobility National Research Center and received funding from the European Union Next-GenerationEU (PIANO NAZIONALE DI RIPRESA E RESILIENZA (PNRR) – MISSIONE 4 COMPONENTE 2, INVESTIMENTO 1.4 – D.D. 1033 17/06/2022, CN00000023), Spoke 5 ``Light Vehicle and Active Mobility''. This manuscript reflects only the authors' views and opinions, neither the European Union nor the European Commission can be considered responsible for them.

\section*{References}

\bibliography{svm_bib}

\begin{thebibliography}{93}
\expandafter\ifx\csname natexlab\endcsname\relax\def\natexlab#1{#1}\fi
\providecommand{\url}[1]{\texttt{#1}}
\providecommand{\href}[2]{#2}
\providecommand{\path}[1]{#1}
\providecommand{\DOIprefix}{doi:}
\providecommand{\ArXivprefix}{arXiv:}
\providecommand{\URLprefix}{URL: }
\providecommand{\Pubmedprefix}{pmid:}
\providecommand{\doi}[1]{\href{http://dx.doi.org/#1}{\path{#1}}}
\providecommand{\Pubmed}[1]{\href{pmid:#1}{\path{#1}}}
\providecommand{\bibinfo}[2]{#2}
\ifx\xfnm\relax \def\xfnm[#1]{\unskip,\space#1}\fi
\bibitem[{Ben-Tal et~al.(2012)Ben-Tal, Bhadra, Bhattacharyya \&
  Nemirovski}]{Ben-TalBhaNem2012}
\bibinfo{author}{Ben-Tal, A.}, \bibinfo{author}{Bhadra, S.},
  \bibinfo{author}{Bhattacharyya, C.}, \& \bibinfo{author}{Nemirovski, A.}
  (\bibinfo{year}{2012}).
\newblock \bibinfo{title}{Efficient methods for robust classification under
  uncertainty in kernel matrices}.
\newblock {\it \bibinfo{journal}{Journal of Machine Learning Research}\/},
  {\it \bibinfo{volume}{13}\/}, \bibinfo{pages}{2923--2954}.
\bibitem[{Ben-Tal et~al.(2009)Ben-Tal, {El Ghaoui} \&
  Nemirovski}]{Ben-TalElGNem2009}
\bibinfo{author}{Ben-Tal, A.}, \bibinfo{author}{{El Ghaoui}, L.}, \&
  \bibinfo{author}{Nemirovski, A.} (\bibinfo{year}{2009}).
\newblock \bibinfo{title}{Robust optimization}.
\newblock \bibinfo{publisher}{Princeton University Press}.
\bibitem[{Bengio et~al.(2021)Bengio, Lodi \& Prouvost}]{BenLodPro2021}
\bibinfo{author}{Bengio, Y.}, \bibinfo{author}{Lodi, A.}, \&
  \bibinfo{author}{Prouvost, A.} (\bibinfo{year}{2021}).
\newblock \bibinfo{title}{Machine learning for combinatorial optimization: A
  methodological tour d'horizon}.
\newblock {\it \bibinfo{journal}{European Journal of Operational Research}\/},
  {\it \bibinfo{volume}{290}\/}, \bibinfo{pages}{405--421}.
\bibitem[{Ben\'{i}tez-Pe{\~n}a et~al.(2024)Ben\'{i}tez-Pe{\~n}a, Blanquero,
  Carrizosa \& Ram\'{i}rez-Cobo}]{Ben-PenBlanCarRam-Cob2024}
\bibinfo{author}{Ben\'{i}tez-Pe{\~n}a, S.}, \bibinfo{author}{Blanquero, R.},
  \bibinfo{author}{Carrizosa, E.}, \& \bibinfo{author}{Ram\'{i}rez-Cobo, P.}
  (\bibinfo{year}{2024}).
\newblock \bibinfo{title}{Cost-sensitive probabilistic predictions for support
  vector machines}.
\newblock {\it \bibinfo{journal}{European Journal of Operational Research}\/},
  {\it \bibinfo{volume}{314}\/}, \bibinfo{pages}{268--279}.
\bibitem[{Bennett \& Mangasarian(1992)}]{BenMan1992}
\bibinfo{author}{Bennett, K.~P.}, \& \bibinfo{author}{Mangasarian, O.~L.}
  (\bibinfo{year}{1992}).
\newblock \bibinfo{title}{Robust linear programming discrimination of two
  linearly inseparable sets}.
\newblock {\it \bibinfo{journal}{Optimization Methods \& Software}\/},  {\it
  \bibinfo{volume}{1}\/}, \bibinfo{pages}{23--34}.
\bibitem[{Bertsimas et~al.(2011)Bertsimas, Brown \& Caramanis}]{BerBroCar2010}
\bibinfo{author}{Bertsimas, D.}, \bibinfo{author}{Brown, D.~B.}, \&
  \bibinfo{author}{Caramanis, C.} (\bibinfo{year}{2011}).
\newblock \bibinfo{title}{Theory and applications of robust optimization}.
\newblock {\it \bibinfo{journal}{SIAM review}\/},  {\it
  \bibinfo{volume}{53}\/}, \bibinfo{pages}{464--501}.
\bibitem[{Bertsimas et~al.(2019)Bertsimas, Dunn, Pawlowski \&
  Zhuo}]{BerDunPawZhu2019}
\bibinfo{author}{Bertsimas, D.}, \bibinfo{author}{Dunn, J.},
  \bibinfo{author}{Pawlowski, C.}, \& \bibinfo{author}{Zhuo, Y.~D.}
  (\bibinfo{year}{2019}).
\newblock \bibinfo{title}{Robust classification}.
\newblock {\it \bibinfo{journal}{INFORMS Journal of Optimization}\/},  {\it
  \bibinfo{volume}{1}\/}, \bibinfo{pages}{2--34}.
\bibitem[{Bhadra et~al.(2010)Bhadra, Bhattacharya, Bhattacharyya \&
  Ben-Tal}]{BhaBhaBen-Tal2010}
\bibinfo{author}{Bhadra, S.}, \bibinfo{author}{Bhattacharya, S.},
  \bibinfo{author}{Bhattacharyya, C.}, \& \bibinfo{author}{Ben-Tal, A.}
  (\bibinfo{year}{2010}).
\newblock \bibinfo{title}{Robust formulations for handling uncertainty in
  kernel matrices}.
\newblock {\it \bibinfo{journal}{Proceedings for the 27th International
  Conference on Machine Learning}\/},  (pp. \bibinfo{pages}{71--78}).
\bibitem[{Bhattacharyya(2004)}]{Bhat2004}
\bibinfo{author}{Bhattacharyya, C.} (\bibinfo{year}{2004}).
\newblock \bibinfo{title}{Robust classification of noisy data using second
  order cone programming approach}.
\newblock In {\it \bibinfo{booktitle}{International Conference on Intelligent
  Sensing and Information Processing, 2004}\/} (pp. \bibinfo{pages}{433--438}).
\bibitem[{Bi \& Zhang(2005)}]{BiZha2005}
\bibinfo{author}{Bi, J.}, \& \bibinfo{author}{Zhang, T.}
  (\bibinfo{year}{2005}).
\newblock \bibinfo{title}{Support vector classification with input data
  uncertainty}.
\newblock In {\it \bibinfo{booktitle}{Advances in neural information processing
  systems}\/} (pp. \bibinfo{pages}{161--168}).
\bibitem[{Blanco et~al.(2020)Blanco, Puerto \&
  Rodr\'{i}guez-Ch\'{i}a}]{BlaPueRod-Chi2020}
\bibinfo{author}{Blanco, V.}, \bibinfo{author}{Puerto, J.}, \&
  \bibinfo{author}{Rodr\'{i}guez-Ch\'{i}a, A.~M.} (\bibinfo{year}{2020}).
\newblock \bibinfo{title}{On lp-support vector machines and multidimensional
  kernels}.
\newblock {\it \bibinfo{journal}{Journal of Machine Learning Research}\/},
  {\it \bibinfo{volume}{21}\/}, \bibinfo{pages}{1--29}.
\bibitem[{Boser et~al.(1992)Boser, Guyon \& Vapnik}]{BosGuyVap1992}
\bibinfo{author}{Boser, B.~E.}, \bibinfo{author}{Guyon, I.}, \&
  \bibinfo{author}{Vapnik, V.~N.} (\bibinfo{year}{1992}).
\newblock \bibinfo{title}{A training algorithm for optimal margin classifiers}.
\newblock {\it \bibinfo{journal}{Proceedings of the Fifth Annual Workshop of
  Computational Learning Theory}\/},  {\it \bibinfo{volume}{5}\/},
  \bibinfo{pages}{144--152}.
\bibitem[{Cervantes et~al.(2020)Cervantes, Garcia-Lamont, Rodr\'{i}guez-Mazahua
  \& Lopez}]{CerGarRodAsd2020}
\bibinfo{author}{Cervantes, J.}, \bibinfo{author}{Garcia-Lamont, F.},
  \bibinfo{author}{Rodr\'{i}guez-Mazahua, L.}, \& \bibinfo{author}{Lopez, A.}
  (\bibinfo{year}{2020}).
\newblock \bibinfo{title}{A comprehensive survey on support vector machine
  classification: Applications, challenges and trends}.
\newblock {\it \bibinfo{journal}{Neurocomputing}\/},  {\it
  \bibinfo{volume}{408}\/}, \bibinfo{pages}{189--215}.
\bibitem[{Chen et~al.(2001)Chen, Tse \& Yu}]{ChenTseYu2001}
\bibinfo{author}{Chen, T.~Y.}, \bibinfo{author}{Tse, T.~H.}, \&
  \bibinfo{author}{Yu, Y.-T.} (\bibinfo{year}{2001}).
\newblock \bibinfo{title}{Proportional sampling strategy: a compendium and some
  insights}.
\newblock {\it \bibinfo{journal}{The Journal of Systems and Software}\/},  {\it
  \bibinfo{volume}{58}\/}, \bibinfo{pages}{65--81}.
\bibitem[{Chen et~al.(2012)Chen, Fan \& Sun}]{ChenFanSun2012}
\bibinfo{author}{Chen, Z.-Y.}, \bibinfo{author}{Fan, Z.-P.}, \&
  \bibinfo{author}{Sun, M.} (\bibinfo{year}{2012}).
\newblock \bibinfo{title}{A hierarchical multiple kernel support vector machine
  for customer churn prediction using longitudinal behavioral data}.
\newblock {\it \bibinfo{journal}{European Journal of Operational Research}\/},
  {\it \bibinfo{volume}{223}\/}, \bibinfo{pages}{461--472}.
\bibitem[{Cortes \& Vapnik(1995)}]{CorVap1995}
\bibinfo{author}{Cortes, C.}, \& \bibinfo{author}{Vapnik, V.~N.}
  (\bibinfo{year}{1995}).
\newblock \bibinfo{title}{Support-vector networks}.
\newblock {\it \bibinfo{journal}{Machine Learning}\/},  {\it
  \bibinfo{volume}{20}\/}, \bibinfo{pages}{273--297}.
\bibitem[{{De Bock} et~al.(2023){De Bock}, Coussement, Caigny, S\l{}owi\'nski,
  Baesens, Boute, Choi, Delen, Kraus, Lessmann, Maldonado, Martens,
  \'Oskarsd\'ottir, Vairetti, Verbeke \& Weber}]{DeBocketal2023}
\bibinfo{author}{{De Bock}, K.~W.}, \bibinfo{author}{Coussement, K.},
  \bibinfo{author}{Caigny, A.~D.}, \bibinfo{author}{S\l{}owi\'nski, R.},
  \bibinfo{author}{Baesens, B.}, \bibinfo{author}{Boute, R.~N.},
  \bibinfo{author}{Choi, T.-M.}, \bibinfo{author}{Delen, D.},
  \bibinfo{author}{Kraus, M.}, \bibinfo{author}{Lessmann, S.},
  \bibinfo{author}{Maldonado, S.}, \bibinfo{author}{Martens, D.},
  \bibinfo{author}{\'Oskarsd\'ottir, M.}, \bibinfo{author}{Vairetti, C.},
  \bibinfo{author}{Verbeke, W.}, \& \bibinfo{author}{Weber, R.}
  (\bibinfo{year}{2023}).
\newblock \bibinfo{title}{Explainable ai for operational research: A defining
  framework, methods, applications, and a research agenda}.
\newblock {\it \bibinfo{journal}{European Journal of Operational Research}\/},
  {\it \bibinfo{volume}{in press}\/}.
\bibitem[{De~Leone et~al.(2023)De~Leone, Maggioni \&
  Spinelli}]{DelMagSpi2023TPMSVM}
\bibinfo{author}{De~Leone, R.}, \bibinfo{author}{Maggioni, F.}, \&
  \bibinfo{author}{Spinelli, A.} (\bibinfo{year}{2023}).
\newblock \bibinfo{title}{A robust twin parametric margin support vector
  machine for multiclass classification}.
\newblock \URLprefix \url{https://arxiv.org/abs/2306.06213}.
\bibitem[{De~Leone et~al.(2024)De~Leone, Maggioni \& Spinelli}]{DeLMagSpi2024}
\bibinfo{author}{De~Leone, R.}, \bibinfo{author}{Maggioni, F.}, \&
  \bibinfo{author}{Spinelli, A.} (\bibinfo{year}{2024}).
\newblock \bibinfo{title}{A multiclass robust twin parametric margin support
  vector machine with an application to vehicles emissions}.
\newblock In \bibinfo{editor}{G.~Nicosia}, \bibinfo{editor}{V.~Ojha},
  \bibinfo{editor}{E.~La~Malfa}, \bibinfo{editor}{G.~La~Malfa},
  \bibinfo{editor}{P.~M. Pardalos}, \& \bibinfo{editor}{R.~Umeton} (Eds.), {\it
  \bibinfo{booktitle}{Machine Learning, Optimization, and Data Science}\/} (pp.
  \bibinfo{pages}{299--310}).
\newblock \bibinfo{address}{Cham}: \bibinfo{publisher}{Springer Nature
  Switzerland} volume \bibinfo{volume}{14506} of {\it \bibinfo{series}{Lecture
  Notes in Computer Science}\/}.
\newblock \DOIprefix\doi{https://doi.org/10.1007/978-3-031-53966-4_22}.
\bibitem[{Dem{\v{s}}ar(2006)}]{Dem2006}
\bibinfo{author}{Dem{\v{s}}ar, J.} (\bibinfo{year}{2006}).
\newblock \bibinfo{title}{Statistical comparisons of classifiers over multiple
  data sets}.
\newblock {\it \bibinfo{journal}{Journal of Machine Learning Research}\/},
  {\it \bibinfo{volume}{7}\/}, \bibinfo{pages}{1--30}.
\bibitem[{Ding \& Hua(2014)}]{DingHua2014}
\bibinfo{author}{Ding, S.}, \& \bibinfo{author}{Hua, X.}
  (\bibinfo{year}{2014}).
\newblock \bibinfo{title}{Recursive least squares projection twin support
  vector machines for nonlinear classification}.
\newblock {\it \bibinfo{journal}{Neurocomputing}\/},  {\it
  \bibinfo{volume}{130}\/}, \bibinfo{pages}{3--9}.
\newblock \bibinfo{note}{Track on Intelligent Computing and Applications
  Complex Learning in Connectionist Networks}.
\bibitem[{Ding et~al.(2019)Ding, Zhao, Zhang, Zhang \&
  Xue}]{DingZhaZhaZhaXue2019}
\bibinfo{author}{Ding, S.}, \bibinfo{author}{Zhao, X.}, \bibinfo{author}{Zhang,
  J.}, \bibinfo{author}{Zhang, X.}, \& \bibinfo{author}{Xue, Y.}
  (\bibinfo{year}{2019}).
\newblock \bibinfo{title}{A review on multi-class twsvm}.
\newblock {\it \bibinfo{journal}{Artificial Intelligence Review}\/},  {\it
  \bibinfo{volume}{52}\/}, \bibinfo{pages}{775--801}.
\bibitem[{Doumpos et~al.(2023)Doumpos, Zopounidis, Gounopoulos, Platanakis \&
  Zhang}]{DouZopGouPlaZha2023}
\bibinfo{author}{Doumpos, M.}, \bibinfo{author}{Zopounidis, C.},
  \bibinfo{author}{Gounopoulos, D.}, \bibinfo{author}{Platanakis, E.}, \&
  \bibinfo{author}{Zhang, W.} (\bibinfo{year}{2023}).
\newblock \bibinfo{title}{Operational research and artificial intelligence
  methods in banking}.
\newblock {\it \bibinfo{journal}{European Journal of Operational Research}\/},
  {\it \bibinfo{volume}{306}\/}, \bibinfo{pages}{1--16}.
\bibitem[{Du et~al.(2021)Du, Zhang, Chen, Sun, Chen \&
  Shao}]{DuZhaCheSunCheShao2021}
\bibinfo{author}{Du, S.-W.}, \bibinfo{author}{Zhang, M.-C.},
  \bibinfo{author}{Chen, P.}, \bibinfo{author}{Sun, H.-F.},
  \bibinfo{author}{Chen, W.-J.}, \& \bibinfo{author}{Shao, Y.-H.}
  (\bibinfo{year}{2021}).
\newblock \bibinfo{title}{A multiclass nonparallel parametric-margin support
  vector machine}.
\newblock {\it \bibinfo{journal}{Information}\/},  {\it
  \bibinfo{volume}{12}\/}, \bibinfo{pages}{515--533}.
\bibitem[{El~Ghaoui et~al.(2003)El~Ghaoui, Lanckriet, Natsoulis
  et~al.}]{ElGLan2003}
\bibinfo{author}{El~Ghaoui, L.}, \bibinfo{author}{Lanckriet, G. R.~G.},
  \bibinfo{author}{Natsoulis, G.} et~al. (\bibinfo{year}{2003}).
\newblock \bibinfo{title}{Robust classification with interval data}.
\newblock In {\it \bibinfo{booktitle}{Computer Science Division, University of
  California Berkeley}\/}.
\bibitem[{Faccini et~al.(2022)Faccini, Maggioni \& Potra}]{FacMagPot2022}
\bibinfo{author}{Faccini, D.}, \bibinfo{author}{Maggioni, F.}, \&
  \bibinfo{author}{Potra, F.~A.} (\bibinfo{year}{2022}).
\newblock \bibinfo{title}{Robust and distributionally robust optimization
  models for linear support vector machine}.
\newblock {\it \bibinfo{journal}{Computers and Operations Research}\/},  {\it
  \bibinfo{volume}{147}\/}, \bibinfo{pages}{105930}.
\bibitem[{Fan et~al.(2014)Fan, Sadeghi \& Pardalos}]{FanSadPar2014}
\bibinfo{author}{Fan, N.}, \bibinfo{author}{Sadeghi, E.}, \&
  \bibinfo{author}{Pardalos, P.~M.} (\bibinfo{year}{2014}).
\newblock \bibinfo{title}{Robust support vector machines with polyhedral
  uncertainty of the input data}.
\newblock In {\it \bibinfo{booktitle}{Learning and Intelligent Optimization.
  International Conference on Learning and Intelligent Optimization}\/} (pp.
  \bibinfo{pages}{291--305}).
\newblock \bibinfo{publisher}{Springer-Verlag}.
\bibitem[{Fung et~al.(2002)Fung, Mangasarian \& Shavlik}]{FunManSha2002}
\bibinfo{author}{Fung, G.}, \bibinfo{author}{Mangasarian, O.~L.}, \&
  \bibinfo{author}{Shavlik, J.~W.} (\bibinfo{year}{2002}).
\newblock \bibinfo{title}{Knowledge-based support vector machine classifiers}.
\newblock In {\it \bibinfo{booktitle}{NIPS}\/} (pp. \bibinfo{pages}{521--528}).
\bibitem[{Gambella et~al.(2021)Gambella, Ghaddar \&
  Naoum-Sawaya}]{GamGhaSaw2021}
\bibinfo{author}{Gambella, C.}, \bibinfo{author}{Ghaddar, B.}, \&
  \bibinfo{author}{Naoum-Sawaya, J.} (\bibinfo{year}{2021}).
\newblock \bibinfo{title}{Optimization problems for machine learning: A
  survey}.
\newblock {\it \bibinfo{journal}{European Journal of Operational Research}\/},
  {\it \bibinfo{volume}{290}\/}, \bibinfo{pages}{807--828}.
\bibitem[{Gao et~al.(2021)Gao, Fang, Luo \& Medhin}]{GaoFangLuoMed2021}
\bibinfo{author}{Gao, Z.}, \bibinfo{author}{Fang, S.-C.}, \bibinfo{author}{Luo,
  J.}, \& \bibinfo{author}{Medhin, N.} (\bibinfo{year}{2021}).
\newblock \bibinfo{title}{A kernel-free double well potential support vector
  machine with applications}.
\newblock {\it \bibinfo{journal}{European Journal of Operational Research}\/},
  {\it \bibinfo{volume}{290}\/}, \bibinfo{pages}{248--262}.
\bibitem[{Grant \& Boyd(2008)}]{GraBoy2008}
\bibinfo{author}{Grant, M.}, \& \bibinfo{author}{Boyd, S.}
  (\bibinfo{year}{2008}).
\newblock \bibinfo{title}{Graph implementations for nonsmooth convex programs}.
\newblock In \bibinfo{editor}{V.~Blondel}, \bibinfo{editor}{S.~Boyd}, \&
  \bibinfo{editor}{H.~Kimura} (Eds.), {\it \bibinfo{booktitle}{Recent Advances
  in Learning and Control}\/} Lecture Notes in Control and Information Sciences
  (pp. \bibinfo{pages}{95--110}).
\newblock \bibinfo{publisher}{Springer-Verlag Limited}.
\newblock \bibinfo{note}{\url{http://stanford.edu/~boyd/graph_dcp.html}}.
\bibitem[{Grant \& Boyd(2014)}]{cvx2014}
\bibinfo{author}{Grant, M.}, \& \bibinfo{author}{Boyd, S.}
  (\bibinfo{year}{2014}).
\newblock \bibinfo{title}{{CVX}: Matlab software for disciplined convex
  programming, version 2.1}.
\newblock \bibinfo{howpublished}{\url{http://cvxr.com/cvx}}.
\bibitem[{Gunnarsson et~al.(2021)Gunnarsson, {vanden Broucke}, Baesens,
  \'Oskarsd\'ottir \& Lemahieu}]{GunSepBaeOskLem2021}
\bibinfo{author}{Gunnarsson, B.~R.}, \bibinfo{author}{{vanden Broucke}, S.},
  \bibinfo{author}{Baesens, B.}, \bibinfo{author}{\'Oskarsd\'ottir, M.}, \&
  \bibinfo{author}{Lemahieu, W.} (\bibinfo{year}{2021}).
\newblock \bibinfo{title}{Deep learning for credit scoring: Do or don't?}
\newblock {\it \bibinfo{journal}{European Journal of Operational Research}\/},
  {\it \bibinfo{volume}{295}\/}, \bibinfo{pages}{292--305}.
\bibitem[{Han et~al.(2011)Han, Kamber \& Pei}]{HanKamPei2011}
\bibinfo{author}{Han, J.}, \bibinfo{author}{Kamber, M.}, \&
  \bibinfo{author}{Pei, J.} (\bibinfo{year}{2011}).
\newblock {\it \bibinfo{title}{Data mining: concepts and techniques - 3rd
  edition}\/}.
\newblock \bibinfo{publisher}{Morgan Kaufmann}.
\bibitem[{Hao(2010)}]{Hao2010}
\bibinfo{author}{Hao, P.-Y.} (\bibinfo{year}{2010}).
\newblock \bibinfo{title}{New support vector algorithms with parametric
  insensitive/margin model}.
\newblock {\it \bibinfo{journal}{Neural networks : the official journal of the
  International Neural Network Society}\/},  {\it \bibinfo{volume}{23}\/},
  \bibinfo{pages}{60--73}.
\bibitem[{Jayadeva et~al.(2007)Jayadeva, Khemchandani \&
  Chandra}]{JayKheCha2007}
\bibinfo{author}{Jayadeva}, \bibinfo{author}{Khemchandani, R.}, \&
  \bibinfo{author}{Chandra, S.} (\bibinfo{year}{2007}).
\newblock \bibinfo{title}{Twin support vector machines for pattern
  classification}.
\newblock {\it \bibinfo{journal}{IEEE Transactions on Pattern Analysis and
  Machine Intelligence}\/},  {\it \bibinfo{volume}{29}\/},
  \bibinfo{pages}{905--910}.
\bibitem[{Jiang \& Peng(2024)}]{JiaPen2024}
\bibinfo{author}{Jiang, J.}, \& \bibinfo{author}{Peng, S.}
  (\bibinfo{year}{2024}).
\newblock \bibinfo{title}{Mathematical programs with distributionally robust
  chance constraints: Statistical robustness, discretization and
  reformulation}.
\newblock {\it \bibinfo{journal}{European Journal of Operational Research}\/},
  {\it \bibinfo{volume}{313}\/}, \bibinfo{pages}{616--627}.
\bibitem[{Jim\'{e}nez-Cordero et~al.(2021)Jim\'{e}nez-Cordero, Morales \&
  Pineda}]{Jim-CorMorPin2021}
\bibinfo{author}{Jim\'{e}nez-Cordero, A.}, \bibinfo{author}{Morales, J.~M.}, \&
  \bibinfo{author}{Pineda, S.} (\bibinfo{year}{2021}).
\newblock \bibinfo{title}{A novel embedded min-max approach for feature
  selection in nonlinear support vector machine classification}.
\newblock {\it \bibinfo{journal}{European Journal of Operational Research}\/},
  {\it \bibinfo{volume}{293}\/}, \bibinfo{pages}{24--35}.
\bibitem[{Ju \& Tian(2012)}]{JuTian2012}
\bibinfo{author}{Ju, X.}, \& \bibinfo{author}{Tian, Y.} (\bibinfo{year}{2012}).
\newblock \bibinfo{title}{Knowledge-based support vector machine classifiers
  via nearest points}.
\newblock {\it \bibinfo{journal}{Procedia Computer Science}\/},  {\it
  \bibinfo{volume}{9}\/}, \bibinfo{pages}{1240--1248}.
\bibitem[{Katsafados et~al.(2024)Katsafados, Leledakis, Pyrgiotakis,
  Androutsopoulos \& Fergadiotis}]{KatLelPyrAndFer2024}
\bibinfo{author}{Katsafados, A.~G.}, \bibinfo{author}{Leledakis, G.~N.},
  \bibinfo{author}{Pyrgiotakis, E.~G.}, \bibinfo{author}{Androutsopoulos, I.},
  \& \bibinfo{author}{Fergadiotis, M.} (\bibinfo{year}{2024}).
\newblock \bibinfo{title}{Machine learning in bank merger prediction: A
  text-based approach}.
\newblock {\it \bibinfo{journal}{European Journal of Operational Research}\/},
  {\it \bibinfo{volume}{312}\/}, \bibinfo{pages}{783--797}.
\bibitem[{Kelly et~al.(2023)Kelly, Longjohn \& Nottingham}]{KelLonNot2023}
\bibinfo{author}{Kelly, M.}, \bibinfo{author}{Longjohn, R.}, \&
  \bibinfo{author}{Nottingham, K.} (\bibinfo{year}{2023}).
\newblock \bibinfo{title}{{UCI} machine learning repository}.
\newblock \URLprefix \url{http://archive.ics.uci.edu/ml}.
\bibitem[{Ketkov(2024)}]{Ket2024}
\bibinfo{author}{Ketkov, S.~S.} (\bibinfo{year}{2024}).
\newblock \bibinfo{title}{A study of distributionally robust mixed-integer
  programming with wasserstein metric: on the value of incomplete data}.
\newblock {\it \bibinfo{journal}{European Journal of Operational Research}\/},
  {\it \bibinfo{volume}{313}\/}, \bibinfo{pages}{602--615}.
\bibitem[{Khanjani-Shiraz et~al.(2023)Khanjani-Shiraz, Babapour-Azar,
  Hosseini-Nodeh \& Pardalos}]{Kha-ShiBab-AzaHos-NodPar2023}
\bibinfo{author}{Khanjani-Shiraz, R.}, \bibinfo{author}{Babapour-Azar, A.},
  \bibinfo{author}{Hosseini-Nodeh, Z.}, \& \bibinfo{author}{Pardalos, P.~M.}
  (\bibinfo{year}{2023}).
\newblock \bibinfo{title}{Distributionally robust joint chance-constrained
  support vector machines}.
\newblock {\it \bibinfo{journal}{Optimization Letters}\/},  {\it
  \bibinfo{volume}{17}\/}, \bibinfo{pages}{299--332}.
\bibitem[{Kim(2009)}]{Kim2009}
\bibinfo{author}{Kim, J.-H.} (\bibinfo{year}{2009}).
\newblock \bibinfo{title}{Estimating classification error rate: Repeated
  cross-validation, repeated hold-out and bootstrap}.
\newblock {\it \bibinfo{journal}{Computational Statistics and Data
  Analysis}\/},  {\it \bibinfo{volume}{53}\/}, \bibinfo{pages}{3735--3745}.
\bibitem[{Labb\'e et~al.(2019)Labb\'e, Mart\'inez-Merino \&
  Rodr\'iguez-Ch\'ia}]{LabMar-MerRod-Chi2019}
\bibinfo{author}{Labb\'e, M.}, \bibinfo{author}{Mart\'inez-Merino, L.~I.}, \&
  \bibinfo{author}{Rodr\'iguez-Ch\'ia, A.~M.} (\bibinfo{year}{2019}).
\newblock \bibinfo{title}{Mixed integer linear programming for feature
  selection in support vector machine}.
\newblock {\it \bibinfo{journal}{Discrete Applied Mathematics}\/},  {\it
  \bibinfo{volume}{261}\/}, \bibinfo{pages}{276--304}.
\bibitem[{Lanckriet et~al.(2002)Lanckriet, Ghaoui, Bhattacharyya \&
  Jordan}]{LanGhaBhaJor2002}
\bibinfo{author}{Lanckriet, G. R.~G.}, \bibinfo{author}{Ghaoui, L.~E.},
  \bibinfo{author}{Bhattacharyya, C.}, \& \bibinfo{author}{Jordan, M.~I.}
  (\bibinfo{year}{2002}).
\newblock \bibinfo{title}{A robust minimax approach to classification}.
\newblock {\it \bibinfo{journal}{Journal of Machine Learning Research}\/},
  {\it \bibinfo{volume}{3}\/}, \bibinfo{pages}{555--582}.
\bibitem[{Lee et~al.(2022)Lee, Yoon \& Won}]{LeeYoonWon2022}
\bibinfo{author}{Lee, I.~G.}, \bibinfo{author}{Yoon, S.~W.}, \&
  \bibinfo{author}{Won, D.} (\bibinfo{year}{2022}).
\newblock \bibinfo{title}{A mixed integer linear programming support vector
  machine for cost-effective group feature selection: Branch-cut-and-price
  approach}.
\newblock {\it \bibinfo{journal}{European Journal of Operational Research}\/},
  {\it \bibinfo{volume}{299}\/}, \bibinfo{pages}{1055--1068}.
\bibitem[{Lee et~al.(2000)Lee, Mangasarian \& Wolberg}]{LeeManWol2000}
\bibinfo{author}{Lee, Y.-J.}, \bibinfo{author}{Mangasarian, O.~L.}, \&
  \bibinfo{author}{Wolberg, W.~H.} (\bibinfo{year}{2000}).
\newblock \bibinfo{title}{Breast cancer survival and chemotherapy: a support
  vector machine analysis}.
\newblock {\it \bibinfo{journal}{Discrete mathematical problems with medical
  applications}\/},  {\it \bibinfo{volume}{55}\/}, \bibinfo{pages}{1--10}.
\bibitem[{Li et~al.(2009)Li, Liang \& Xu}]{LiLiXu2009}
\bibinfo{author}{Li, H.}, \bibinfo{author}{Liang, Y.}, \& \bibinfo{author}{Xu,
  Q.} (\bibinfo{year}{2009}).
\newblock \bibinfo{title}{Support vector machines and its applications in
  chemistry}.
\newblock {\it \bibinfo{journal}{Chemometrics and Intelligent Laboratory
  Systems}\/},  {\it \bibinfo{volume}{95}\/}, \bibinfo{pages}{188--198}.
\bibitem[{Liao et~al.(2024)Liao, Dai \& Kuosmanen}]{LiaoDaiKuo2024}
\bibinfo{author}{Liao, Z.}, \bibinfo{author}{Dai, S.}, \&
  \bibinfo{author}{Kuosmanen, T.} (\bibinfo{year}{2024}).
\newblock \bibinfo{title}{Convex support vector regression}.
\newblock {\it \bibinfo{journal}{European Journal of Operational Research}\/},
  {\it \bibinfo{volume}{313}\/}, \bibinfo{pages}{858--870}.
\bibitem[{Lin et~al.(2024{\natexlab{a}})Lin, Fang, Fang \&
  Gao}]{LinFangFangGao2024}
\bibinfo{author}{Lin, F.}, \bibinfo{author}{Fang, S.-C.},
  \bibinfo{author}{Fang, X.}, \& \bibinfo{author}{Gao, Z.}
  (\bibinfo{year}{2024}{\natexlab{a}}).
\newblock \bibinfo{title}{Distributionally robust chance-constrained
  kernel-based support vector machine}.
\newblock {\it \bibinfo{journal}{Computers \& Operations Research}\/},  {\it
  \bibinfo{volume}{170}\/}, \bibinfo{pages}{106755}.
\bibitem[{Lin et~al.(2024{\natexlab{b}})Lin, Fang, Fang, Gao \&
  Luo}]{LinFangFangGaoLuo2024}
\bibinfo{author}{Lin, F.}, \bibinfo{author}{Fang, S.-C.},
  \bibinfo{author}{Fang, X.}, \bibinfo{author}{Gao, Z.}, \&
  \bibinfo{author}{Luo, J.} (\bibinfo{year}{2024}{\natexlab{b}}).
\newblock \bibinfo{title}{A distributionally robust chance-constrained
  kernel-free quadratic surface support vector machine}.
\newblock {\it \bibinfo{journal}{European Journal of Operational Research}\/},
  {\it \bibinfo{volume}{316}\/}, \bibinfo{pages}{46--60}.
\bibitem[{Liu \& Potra(2009)}]{LiuPot2009}
\bibinfo{author}{Liu, X.}, \& \bibinfo{author}{Potra, F.~A.}
  (\bibinfo{year}{2009}).
\newblock \bibinfo{title}{Pattern separation and prediction via linear and
  semidefinite programming}.
\newblock {\it \bibinfo{journal}{Studies in Informatics and Control}\/},  {\it
  \bibinfo{volume}{18}\/}, \bibinfo{pages}{71--82}.
\bibitem[{L{\'o}pez et~al.(2017)L{\'o}pez, Maldonado \&
  Carrasco}]{LopMalCar2017}
\bibinfo{author}{L{\'o}pez, J.}, \bibinfo{author}{Maldonado, S.}, \&
  \bibinfo{author}{Carrasco, M.} (\bibinfo{year}{2017}).
\newblock \bibinfo{title}{A robust formulation for twin multiclass support
  vector machine}.
\newblock {\it \bibinfo{journal}{Applied Intelligence}\/},  {\it
  \bibinfo{volume}{47}\/}, \bibinfo{pages}{1031--1043}.
\bibitem[{L{\'o}pez et~al.(2018)L{\'o}pez, Maldonado \&
  Carrasco}]{LopMalCar2018}
\bibinfo{author}{L{\'o}pez, J.}, \bibinfo{author}{Maldonado, S.}, \&
  \bibinfo{author}{Carrasco, M.} (\bibinfo{year}{2018}).
\newblock \bibinfo{title}{Double regularization methods for robust feature
  selection and svm classification via dc programming}.
\newblock {\it \bibinfo{journal}{Information Sciences}\/},  {\it
  \bibinfo{volume}{429}\/}, \bibinfo{pages}{377--389}.
\bibitem[{L{\'o}pez et~al.(2019)L{\'o}pez, Maldonado \&
  Carrasco}]{LopMalCar2019}
\bibinfo{author}{L{\'o}pez, J.}, \bibinfo{author}{Maldonado, S.}, \&
  \bibinfo{author}{Carrasco, M.} (\bibinfo{year}{2019}).
\newblock \bibinfo{title}{Robust nonparallel support vector machines via
  second-order cone programming}.
\newblock {\it \bibinfo{journal}{Neurocomputing}\/},  {\it
  \bibinfo{volume}{364}\/}, \bibinfo{pages}{227--238}.
\bibitem[{Luo et~al.(2020)Luo, Yan \& Tian}]{LuoYanTian2020}
\bibinfo{author}{Luo, J.}, \bibinfo{author}{Yan, X.}, \& \bibinfo{author}{Tian,
  Y.} (\bibinfo{year}{2020}).
\newblock \bibinfo{title}{Unsupervised quadratic surface support vector machine
  with application to credit risk assessment}.
\newblock {\it \bibinfo{journal}{European Journal of Operational Research}\/},
  {\it \bibinfo{volume}{280}\/}, \bibinfo{pages}{1008--1017}.
\bibitem[{Maggioni et~al.(2023)Maggioni, Faccini, Gheza, Manelli \&
  Bonetti}]{MagFacGheManBonORAHS2022}
\bibinfo{author}{Maggioni, F.}, \bibinfo{author}{Faccini, D.},
  \bibinfo{author}{Gheza, F.}, \bibinfo{author}{Manelli, F.}, \&
  \bibinfo{author}{Bonetti, G.} (\bibinfo{year}{2023}).
\newblock \bibinfo{title}{Machine learning based classification models for
  covid-19 patients}.
\newblock In \bibinfo{editor}{R.~Aringhieri}, \bibinfo{editor}{F.~Maggioni},
  \bibinfo{editor}{E.~Lanzarone}, \bibinfo{editor}{M.~Reuter-Oppermann},
  \bibinfo{editor}{G.~Righini}, \& \bibinfo{editor}{M.~T. Vespucci} (Eds.),
  {\it \bibinfo{booktitle}{Operations Research for Health Care in Red Zone}\/}
  (pp. \bibinfo{pages}{35--46}).
\newblock \bibinfo{address}{Cham}: \bibinfo{publisher}{Springer International
  Publishing}.
\bibitem[{Maggioni et~al.(2009)Maggioni, Potra, Bertocchi \&
  Allevi}]{MagPotBerAll2009}
\bibinfo{author}{Maggioni, F.}, \bibinfo{author}{Potra, F.~A.},
  \bibinfo{author}{Bertocchi, M.}, \& \bibinfo{author}{Allevi, E.}
  (\bibinfo{year}{2009}).
\newblock \bibinfo{title}{Stochastic second-order cone programming in mobile ad
  hoc networks}.
\newblock {\it \bibinfo{journal}{Journal of Optimization Theory and
  Applications}\/},  {\it \bibinfo{volume}{143}\/}, \bibinfo{pages}{309--328}.
\bibitem[{Maggioni \& Spinelli(2024)}]{MagSpi2024}
\bibinfo{author}{Maggioni, F.}, \& \bibinfo{author}{Spinelli, A.}
  (\bibinfo{year}{2024}).
\newblock \bibinfo{title}{A robust nonlinear support vector machine approach
  for vehicles smog rating classification}.
\newblock In \bibinfo{editor}{M.~Bruglieri}, \bibinfo{editor}{P.~Festa},
  \bibinfo{editor}{G.~Macrina}, \& \bibinfo{editor}{O.~Pisacane} (Eds.), {\it
  \bibinfo{booktitle}{Optimization in Green Sustainability and Ecological
  Transition}\/} AIRO Springer Series.
\newblock \bibinfo{publisher}{Springer Cham}.
\newblock \DOIprefix\doi{https://doi.org/10.1007/978-3-031-47686-0_19}.
\bibitem[{Maldonado et~al.(2022)Maldonado, L\'opez \& Carrasco}]{MalLopCar2022}
\bibinfo{author}{Maldonado, S.}, \bibinfo{author}{L\'opez, J.}, \&
  \bibinfo{author}{Carrasco, M.} (\bibinfo{year}{2022}).
\newblock \bibinfo{title}{The cobb-douglas learning machine}.
\newblock {\it \bibinfo{journal}{Pattern Recognition}\/},  {\it
  \bibinfo{volume}{128}\/}, \bibinfo{pages}{108701}.
\bibitem[{Maldonado et~al.(2020)Maldonado, L\'opez \& Vairetti}]{MalLopVai2020}
\bibinfo{author}{Maldonado, S.}, \bibinfo{author}{L\'opez, J.}, \&
  \bibinfo{author}{Vairetti, C.} (\bibinfo{year}{2020}).
\newblock \bibinfo{title}{Profit-based churn prediction based on minimax
  probability machines}.
\newblock {\it \bibinfo{journal}{European Journal of Operational Research}\/},
  {\it \bibinfo{volume}{284}\/}, \bibinfo{pages}{273--284}.
\bibitem[{Mangasarian(1998)}]{Man1998}
\bibinfo{author}{Mangasarian, O.~L.} (\bibinfo{year}{1998}).
\newblock \bibinfo{title}{Generalized support vector machines}.
\newblock In {\it \bibinfo{booktitle}{Advances in Large Margin Classifiers}\/}
  (pp. \bibinfo{pages}{135--146}).
\newblock \bibinfo{publisher}{MIT Press}.
\bibitem[{Marcelli \& De~Leone(2020)}]{MarDeLNic2020}
\bibinfo{author}{Marcelli, E.}, \& \bibinfo{author}{De~Leone, R.}
  (\bibinfo{year}{2020}).
\newblock \bibinfo{title}{Multi-kernel covariance terms in multi-output support
  vector machines}.
\newblock In \bibinfo{editor}{G.~Nicosia}, \bibinfo{editor}{V.~Ojha},
  \bibinfo{editor}{E.~La~Malfa}, \bibinfo{editor}{G.~Jansen},
  \bibinfo{editor}{V.~Sciacca}, \bibinfo{editor}{P.~Pardalos},
  \bibinfo{editor}{G.~Giuffrida}, \& \bibinfo{editor}{R.~Umeton} (Eds.), {\it
  \bibinfo{booktitle}{Machine Learning, Optimization, and Data Science}\/} (pp.
  \bibinfo{pages}{1--11}).
\newblock \bibinfo{address}{Cham}: \bibinfo{publisher}{Springer International
  Publishing}.
\bibitem[{Mi et~al.(2019)Mi, Wang, Mi, Huang, Zhang, Yang, Jiang \&
  Octavian}]{MiWangetal2019}
\bibinfo{author}{Mi, C.}, \bibinfo{author}{Wang, J.}, \bibinfo{author}{Mi, W.},
  \bibinfo{author}{Huang, Y.}, \bibinfo{author}{Zhang, Z.},
  \bibinfo{author}{Yang, Y.}, \bibinfo{author}{Jiang, J.}, \&
  \bibinfo{author}{Octavian, P.} (\bibinfo{year}{2019}).
\newblock \bibinfo{title}{Research on regional clustering and two-stage svm
  method for container truck recognition}.
\newblock {\it \bibinfo{journal}{Discrete and Continuous Dynamical Systems -
  S}\/},  {\it \bibinfo{volume}{12}\/}, \bibinfo{pages}{1117--1133}.
\bibitem[{{MOSEK ApS}(2019)}]{mosek}
\bibinfo{author}{{MOSEK ApS}} (\bibinfo{year}{2019}).
\newblock {\it \bibinfo{title}{The MOSEK optimization toolbox for MATLAB
  manual. Version 9.1}\/}.
\newblock \URLprefix \url{http://docs.mosek.com/9.1/toolbox/index.html}.
\bibitem[{Pedregosa et~al.(2011)Pedregosa, Varoquaux, Gramfort, Michel,
  Thirion, Grisel, Blondel, Prettenhofer, Weiss, Dubourg, Vanderplas, Passos,
  Cournapeau, Brucher, Perrot \& Duchesnay}]{scikitlearn}
\bibinfo{author}{Pedregosa, F.}, \bibinfo{author}{Varoquaux, G.},
  \bibinfo{author}{Gramfort, A.}, \bibinfo{author}{Michel, V.},
  \bibinfo{author}{Thirion, B.}, \bibinfo{author}{Grisel, O.},
  \bibinfo{author}{Blondel, M.}, \bibinfo{author}{Prettenhofer, P.},
  \bibinfo{author}{Weiss, R.}, \bibinfo{author}{Dubourg, V.},
  \bibinfo{author}{Vanderplas, J.}, \bibinfo{author}{Passos, A.},
  \bibinfo{author}{Cournapeau, D.}, \bibinfo{author}{Brucher, M.},
  \bibinfo{author}{Perrot, M.}, \& \bibinfo{author}{Duchesnay, E.}
  (\bibinfo{year}{2011}).
\newblock \bibinfo{title}{Scikit-learn: Machine learning in {P}ython}.
\newblock {\it \bibinfo{journal}{Journal of Machine Learning Research}\/},
  {\it \bibinfo{volume}{12}\/}, \bibinfo{pages}{2825--2830}.
\bibitem[{Peng(2011)}]{Peng2011}
\bibinfo{author}{Peng, X.} (\bibinfo{year}{2011}).
\newblock \bibinfo{title}{Tpmsvm: A novel twin parametric-margin support vector
  machine for pattern recognition}.
\newblock {\it \bibinfo{journal}{Pattern Recognition}\/},  {\it
  \bibinfo{volume}{44}\/}, \bibinfo{pages}{2678--2692}.
\bibitem[{Peng \& Xu(2013)}]{PengXu2013}
\bibinfo{author}{Peng, X.}, \& \bibinfo{author}{Xu, D.} (\bibinfo{year}{2013}).
\newblock \bibinfo{title}{Robust minimum class variance twin support vector
  machine classifier}.
\newblock {\it \bibinfo{journal}{Neural Computing and Applications}\/},  {\it
  \bibinfo{volume}{22}\/}, \bibinfo{pages}{999--1011}.
\bibitem[{Piccialli \& Sciandrone(2018)}]{PicSci2018}
\bibinfo{author}{Piccialli, V.}, \& \bibinfo{author}{Sciandrone, M.}
  (\bibinfo{year}{2018}).
\newblock \bibinfo{title}{Nonlinear optimization and support vector machines}.
\newblock {\it \bibinfo{journal}{4OR - A Quarterly Journal of Operations
  Research}\/},  {\it \bibinfo{volume}{16}\/}, \bibinfo{pages}{111--149}.
\bibitem[{Qi et~al.(2013)Qi, Tian \& Shi}]{QiTianShi2013}
\bibinfo{author}{Qi, Z.}, \bibinfo{author}{Tian, Y.}, \& \bibinfo{author}{Shi,
  Y.} (\bibinfo{year}{2013}).
\newblock \bibinfo{title}{Robust twin support vector machine for pattern
  classification}.
\newblock {\it \bibinfo{journal}{Pattern Recognition}\/},  {\it
  \bibinfo{volume}{46}\/}, \bibinfo{pages}{305--316}.
\bibitem[{Raeesi et~al.(2023)Raeesi, Sahebjamnia \& Mansouri}]{RaeSahMan2023}
\bibinfo{author}{Raeesi, R.}, \bibinfo{author}{Sahebjamnia, N.}, \&
  \bibinfo{author}{Mansouri, S.~A.} (\bibinfo{year}{2023}).
\newblock \bibinfo{title}{The synergistic effect of operational research and
  big data analytics in greening container terminal operations: A review and
  future directions}.
\newblock {\it \bibinfo{journal}{European Journal of Operational Research}\/},
  {\it \bibinfo{volume}{310}\/}, \bibinfo{pages}{943--973}.
\bibitem[{Rudin(1987)}]{Rud1987}
\bibinfo{author}{Rudin, W.} (\bibinfo{year}{1987}).
\newblock {\it \bibinfo{title}{Real and complex analysis}\/}.
\newblock \bibinfo{publisher}{McGraw-Hill}.
\bibitem[{Sch{\"o}lkopf et~al.(2000)Sch{\"o}lkopf, Smola, Williamson \&
  Bartlett}]{SchSmoWilBar2000}
\bibinfo{author}{Sch{\"o}lkopf, B.}, \bibinfo{author}{Smola, A.},
  \bibinfo{author}{Williamson, R.~C.}, \& \bibinfo{author}{Bartlett, P.~L.}
  (\bibinfo{year}{2000}).
\newblock \bibinfo{title}{New support vector algorithms}.
\newblock {\it \bibinfo{journal}{Neural Computation}\/},  {\it
  \bibinfo{volume}{12}\/}, \bibinfo{pages}{1207--1245}.
\bibitem[{Sch{\"o}lkopf \& Smola(2001)}]{SchSmo2001}
\bibinfo{author}{Sch{\"o}lkopf, B.}, \& \bibinfo{author}{Smola, A.~J.}
  (\bibinfo{year}{2001}).
\newblock {\it \bibinfo{title}{Learning with Kernels: Support Vector Machines,
  regularization, optimization, and beyond}\/}.
\newblock \bibinfo{publisher}{MIT press}.
\bibitem[{Singla et~al.(2020)Singla, Ghosh \& Shukla}]{SinGhoShu2020}
\bibinfo{author}{Singla, M.}, \bibinfo{author}{Ghosh, D.}, \&
  \bibinfo{author}{Shukla, K.~K.} (\bibinfo{year}{2020}).
\newblock \bibinfo{title}{A survey of robust optimization based machine
  learning with special reference to support vector machines}.
\newblock {\it \bibinfo{journal}{International Journal of Machine Learning and
  Cybernetics}\/},  {\it \bibinfo{volume}{11}\/}, \bibinfo{pages}{1359--1385}.
\bibitem[{Snoek et~al.(2012)Snoek, Larochelle \& Adams}]{SnoLarAda2012}
\bibinfo{author}{Snoek, J.}, \bibinfo{author}{Larochelle, H.}, \&
  \bibinfo{author}{Adams, R.~P.} (\bibinfo{year}{2012}).
\newblock \bibinfo{title}{Practical bayesian optimization of machine learning
  algorithms}.
\newblock In \bibinfo{editor}{F.~Pereira}, \bibinfo{editor}{C.~Burges},
  \bibinfo{editor}{L.~Bottou}, \& \bibinfo{editor}{K.~Weinberger} (Eds.), {\it
  \bibinfo{booktitle}{Advances in Neural Information Processing Systems}\/}.
\newblock \bibinfo{publisher}{Curran Associates, Inc.}
  volume~\bibinfo{volume}{25}.
\bibitem[{Suman(2018)}]{Suman2018}
\bibinfo{author}{Suman, S.} (\bibinfo{year}{2018}).
\newblock \bibinfo{title}{{TwinSVM}}.
\newblock \URLprefix \url{https://github.com/sumitsomans/TwinSVM}.
\bibitem[{Szelag \& S\l{}owi\'nski(2024)}]{SzeSlo2024}
\bibinfo{author}{Szelag, M.}, \& \bibinfo{author}{S\l{}owi\'nski, R.}
  (\bibinfo{year}{2024}).
\newblock \bibinfo{title}{Explaining and predicting customer churn by monotonic
  rules induced from ordinal data}.
\newblock {\it \bibinfo{journal}{European Journal of Operational Research}\/},
  {\it \bibinfo{volume}{317}\/}, \bibinfo{pages}{414--424}.
\bibitem[{Tanveer et~al.(2022)Tanveer, Rajani, Rastogi \&
  Shao}]{TanRajRasShaoGan2022}
\bibinfo{author}{Tanveer, M.}, \bibinfo{author}{Rajani, T.},
  \bibinfo{author}{Rastogi, R.}, \& \bibinfo{author}{Shao, Y.}
  (\bibinfo{year}{2022}).
\newblock \bibinfo{title}{Comprehensive review on twin support vector
  machines}.
\newblock {\it \bibinfo{journal}{Annals of Operations Research}\/},  (pp.
  \bibinfo{pages}{1--46}).
\bibitem[{Tay \& Cao(2001)}]{TayCao2001}
\bibinfo{author}{Tay, F.~E.}, \& \bibinfo{author}{Cao, L.}
  (\bibinfo{year}{2001}).
\newblock \bibinfo{title}{Application of support vector machines in financial
  time series forecasting}.
\newblock {\it \bibinfo{journal}{Omega}\/},  {\it \bibinfo{volume}{29}\/},
  \bibinfo{pages}{309--317}.
\bibitem[{Trafalis \& Alwazzi(2010)}]{TraAlw2010}
\bibinfo{author}{Trafalis, T.~B.}, \& \bibinfo{author}{Alwazzi, S.~A.}
  (\bibinfo{year}{2010}).
\newblock \bibinfo{title}{Support vector machine classification with noisy
  data: a second order cone programming approach}.
\newblock {\it \bibinfo{journal}{International Journal of General Systems}\/},
  {\it \bibinfo{volume}{39}\/}, \bibinfo{pages}{757--781}.
\bibitem[{Trafalis \& Gilbert(2006)}]{TraGil2006}
\bibinfo{author}{Trafalis, T.~B.}, \& \bibinfo{author}{Gilbert, R.~C.}
  (\bibinfo{year}{2006}).
\newblock \bibinfo{title}{Robust classification and regression using support
  vector machines}.
\newblock {\it \bibinfo{journal}{European Journal of Operational Research}\/},
  {\it \bibinfo{volume}{173}\/}, \bibinfo{pages}{893--909}.
\bibitem[{Vapnik(1995)}]{Vap1995}
\bibinfo{author}{Vapnik, V.~N.} (\bibinfo{year}{1995}).
\newblock {\it \bibinfo{title}{The nature of statistical learning theory}\/}.
\newblock \bibinfo{publisher}{Springer-Verlag}.
\bibitem[{Vapnik \& Chervonenkis(1974)}]{VapChe1974}
\bibinfo{author}{Vapnik, V.~N.}, \& \bibinfo{author}{Chervonenkis, A.~Y.}
  (\bibinfo{year}{1974}).
\newblock {\it \bibinfo{title}{Theory of Pattern Recognition}\/}.
\newblock \bibinfo{publisher}{Nauka, Moscow}.
\bibitem[{Wang et~al.(2018{\natexlab{a}})Wang, Zheng, Yoon \&
  Ko}]{WanZheWoKo2018}
\bibinfo{author}{Wang, H.}, \bibinfo{author}{Zheng, B.}, \bibinfo{author}{Yoon,
  S.~W.}, \& \bibinfo{author}{Ko, H.~S.} (\bibinfo{year}{2018}{\natexlab{a}}).
\newblock \bibinfo{title}{A support vector machine-based ensemble algorithm for
  breast cancer diagnosis}.
\newblock {\it \bibinfo{journal}{European Journal of Operational Research}\/},
  {\it \bibinfo{volume}{267}\/}, \bibinfo{pages}{687--699}.
\bibitem[{Wang et~al.(2018{\natexlab{b}})Wang, Fan \& Pardalos}]{WanFanPar2018}
\bibinfo{author}{Wang, X.}, \bibinfo{author}{Fan, N.}, \&
  \bibinfo{author}{Pardalos, P.~M.} (\bibinfo{year}{2018}{\natexlab{b}}).
\newblock \bibinfo{title}{Robust chance-constrained support vector machines
  with second-order moment information}.
\newblock {\it \bibinfo{journal}{Annals of Operations Research}\/},  {\it
  \bibinfo{volume}{263}\/}, \bibinfo{pages}{45--68}.
\bibitem[{Wang \& Pardalos(2014)}]{WanPar2014}
\bibinfo{author}{Wang, X.}, \& \bibinfo{author}{Pardalos, P.~M.}
  (\bibinfo{year}{2014}).
\newblock \bibinfo{title}{A survey of support vector machines with
  uncertainties}.
\newblock {\it \bibinfo{journal}{Annals of Data Science}\/},  {\it
  \bibinfo{volume}{1}\/}, \bibinfo{pages}{293--309}.
\bibitem[{Wei et~al.(2023)Wei, Hao, Ren \& Glover}]{WeiHaoRenGlov2023}
\bibinfo{author}{Wei, Z.}, \bibinfo{author}{Hao, J.-K.}, \bibinfo{author}{Ren,
  J.}, \& \bibinfo{author}{Glover, F.} (\bibinfo{year}{2023}).
\newblock \bibinfo{title}{Responsive strategic oscillation for solving the
  disjunctively constrained knapsack problem}.
\newblock {\it \bibinfo{journal}{European Journal of Operational Research}\/},
  {\it \bibinfo{volume}{309}\/}, \bibinfo{pages}{993--1009}.
\bibitem[{Weston \& Watkins(1998)}]{WesWat1998}
\bibinfo{author}{Weston, J.}, \& \bibinfo{author}{Watkins, C.}
  (\bibinfo{year}{1998}).
\newblock {\it \bibinfo{title}{Multi-class Support Vector Machines}\/}.
\newblock \bibinfo{type}{Technical Report} Royal Holloway, University of
  London.
\bibitem[{Xu et~al.(2009)Xu, Caramanis \& Mannor}]{XuCarMan2009}
\bibinfo{author}{Xu, H.}, \bibinfo{author}{Caramanis, C.}, \&
  \bibinfo{author}{Mannor, S.} (\bibinfo{year}{2009}).
\newblock \bibinfo{title}{Robustness and regularization of support vector
  machines}.
\newblock {\it \bibinfo{journal}{Journal of Machine Learning Research}\/},
  {\it \bibinfo{volume}{10}\/}, \bibinfo{pages}{1485--1510}.
\bibitem[{Yajima(2005)}]{Yaj2005}
\bibinfo{author}{Yajima, Y.} (\bibinfo{year}{2005}).
\newblock \bibinfo{title}{Linear programming approaches for multicategory
  support vector machines}.
\newblock {\it \bibinfo{journal}{European Journal of Operational Research}\/},
  {\it \bibinfo{volume}{162}\/}, \bibinfo{pages}{514--531}.
\bibitem[{Yao et~al.(2017)Yao, Crook \& Andreeva}]{YaoCroAnd2017}
\bibinfo{author}{Yao, X.}, \bibinfo{author}{Crook, J.}, \&
  \bibinfo{author}{Andreeva, G.} (\bibinfo{year}{2017}).
\newblock \bibinfo{title}{Enhancing two-stage modelling methodology for loss
  given default with support vector machines}.
\newblock {\it \bibinfo{journal}{European Journal of Operational Research}\/},
  {\it \bibinfo{volume}{263}\/}, \bibinfo{pages}{679--689}.

\end{thebibliography}

\newpage
\appendix

\section{Supplementary proofs} \label{appendix_proofs}
We first recall a lemma that will be useful to prove Propositions \ref{lemma_inhom_pol}-\ref{lemma_RBF_gaussian}.

\begin{lemma}[\textbf{Inequalities in $\ell_p$-norm}] \label{lemma_equivalence_norm}
Let $x$ be a vector in $\mathbb{R}^n$. If $1\leq p \leq q \leq \infty$, then:
\begin{linenomath}
\begin{equation} \label{norm_inequality_qpq}
\norm{x}_q \leq \norm{x}_p \leq n^{ \textstyle \frac{1}{p}-\frac{1}{q}}\norm{x}_q.
\end{equation}
\end{linenomath}
\end{lemma}

\begin{proof}
We consider the two inequalities separately, starting from $\norm{x}_q \leq \norm{x}_p$. First of all, if $x=0$, then the inequality is obviously true. Otherwise, let $y\in\mathbb{R}^n$ such that $y_i:=\abs{x_i}/\norm{x}_q$ for $i=1,\ldots,n$. Therefore, $0\leq y_i\leq 1$. Indeed:
\begin{linenomath}
\begin{equation*}
\norm{x}_q^q =\sum_{i=1}^n \abs{x_i}^q \geq \abs{x_i}^q,
\end{equation*}
\end{linenomath}

for all $i=1,\ldots,n$ and thus $\abs{x_i}/\norm{x}_q\leq 1$. The hypothesis $p\leq q$ and the decreasing property of the exponential function with basis lower than one imply that:
\begin{linenomath}
\begin{equation*}
y_i^p \geq y_i^q, \qquad i=1,\ldots,n.
\end{equation*}
\end{linenomath}

By summing we have:
\begin{linenomath}
\begin{equation*}
\norm{y}_p \geq \norm{y}_q.
\end{equation*}
\end{linenomath}

Finally, by definition of $y$ we derive that:
\begin{linenomath}
\begin{equation*}
\frac{\norm{x}_p}{\norm{x}_q}\geq \frac{\norm{x}_q}{\norm{x}_q}=1,
\end{equation*}
\end{linenomath}

from which the thesis follows.

On the other hand, to prove the second inequality we recall the H\"older inequality (see, for instance, \cite{Rud1987}). Let $a$ and $b$ be in $\mathbb{R}^n$. If $r$ and $r'$ are conjugate exponents, i.e. $\textstyle \frac{1}{r}+\frac{1}{r'}=1$, with $1\leq r,r'\leq \infty$, then:
\begin{linenomath}
\begin{equation*}
\norm{ab}_1\leq \norm{a}_r \cdot \norm{b}_{r'},
\end{equation*}
\end{linenomath}

or, equivalently:
\begin{linenomath}
\begin{equation}\label{holder_inequality}
\sum_{i=1}^n \abs{a_i}\abs{b_i} \leq \bigg(\sum_{i=1}^n \abs{a_i}^r\bigg)^{\textstyle \frac{1}{r}} \cdot \bigg(\sum_{i=1}^n \abs{b_i}^{r'}\bigg)^{\textstyle \frac{1}{r'}}.
\end{equation}
\end{linenomath}

First of all, we rewrite the $\ell_p$-norm of $x$ as:
\begin{linenomath}
\begin{equation*}
\norm{x}_p^p=\sum_{i=1}^n \abs{x_i}^p=\sum_{i=1}^n \abs{x_i}^p \cdot 1.
\end{equation*}
\end{linenomath}

In the H\"older inequality \eqref{holder_inequality}, let $a=x$ and $b=e$ and consider as conjugate exponents $r=\textstyle \frac{q}{p}$ and $r'=\textstyle \frac{q}{q-p}$. Both $r$ and $r'$ are greater than or equal to 1 because, by hypothesis, $p\leq q$. Consequently, we can bound the $\ell_p$-norm of $x$ by:
\begin{linenomath}
\begin{equation*}
\norm{x}_p^p \leq \bigg(\sum_{i=1}^n \big(\abs{x_i}^p\big)^{\textstyle \frac{q}{p}}\bigg)^{\textstyle \frac{p}{q}} \cdot \bigg(\sum_{i=1}^n 1^{\textstyle \frac{q}{q-p}}\bigg)^{1 - \textstyle \frac{p}{q}}= \bigg(\sum_{i=1}^n \abs{x_i}^q\bigg)^{\textstyle \frac{p}{q}}n^{1-\textstyle \frac{p}{q}}= \norm{x}_q^{p}n^{1-\textstyle \frac{p}{q}}.
\end{equation*}
\end{linenomath}

Finally, the thesis follows by taking the $p$-th root of both sides of the inequality.
\end{proof}

A graphical representation of inequality \eqref{norm_inequality_qpq} is depicted in Figure \ref{fig_norm_inequality}.

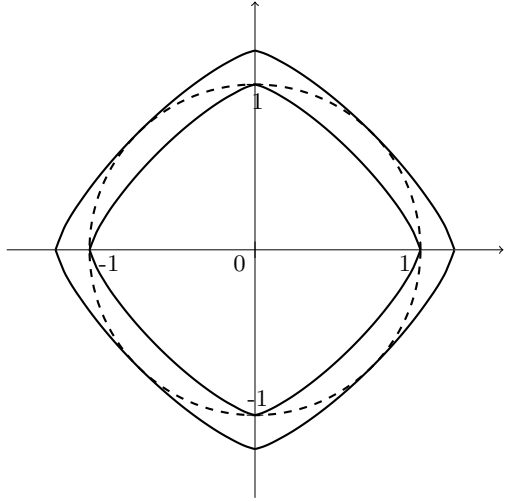
\begin{figure}[h!]
\begin{center}
\begin{tikzpicture}[scale=2.2]
    \draw[very thin,->] (-1.5,0) -- (1.5,0);
    \draw[very thin,->] (0,-1.5) -- (0,1.5);

    \foreach \x in {-1} \draw (\x,0.05) -- (\x,-0.05) node[below=0.5ex,right] {\small\x};
    \foreach \x in {1} \draw (\x,0.05) -- (\x,-0.05) node[below=0.5ex,left] {\small\x};
   \foreach \y in {0} \draw (\y,0.05) -- (\y,-0.05) node[below=0.5ex,left] {\small 0};
    \foreach \y in {-1} \draw (-0.05,\y) -- (0.05,\y) node[left=0.5ex,above] {\small\y};
     \foreach \y in {1} \draw (-0.05,\y) -- (0.05,\y) node[left=0.5ex,below] {\small\y};
    
    \draw[black,dashed,thick] (0,0) circle (1); 
    \draw[domain=0:1.,smooth,variable=\z,thick] plot ({\z},{pow(1-\z^1.3,0.76923)});
	 \draw[domain=0:1.,smooth,variable=\z,thick] plot ({\z},{-pow(1-\z^1.3,0.76923)});
	 \draw[domain=-1:0.,smooth,variable=\z,thick] plot ({\z},{pow(1+\z^1.3,0.76923)});
	  \draw[domain=-1:0.,smooth,variable=\z,thick] plot ({\z},{-pow(1+\z^1.3,0.76923)});
	  \draw[domain=0:1.20516,smooth,variable=\z,thick] plot ({\z},{pow(1.274561-\z^1.3,0.76923)});
	 \draw[domain=0:1.20516,smooth,variable=\z,thick] plot ({\z},{-pow(1.274561-\z^1.3,0.76923)});
	 \draw[domain=-1.20516:0,smooth,variable=\z,thick] plot ({\z},{pow(1.274561+\z^1.3,0.76923)});
	 \draw[domain=-1.20516:0.,smooth,variable=\z,thick] plot ({\z},{-pow(1.274561+\z^1.3,0.76923)});
\end{tikzpicture}
\end{center}
\caption{Graphical representation of Lemma \ref{lemma_equivalence_norm} in the case of $p=1.3$, $q=2$, $n=2$. The dashed $\ell_2$ unit ball lies between the $\ell_{1.3}$ unit ball and the $\ell_{1.3}$ ball with radius $2^{\frac{1}{1.3}-\frac{1}{2}}\approx 1.205$.} \label{fig_norm_inequality}
\end{figure}

As special cases, Lemma \ref{lemma_equivalence_norm} implies that, whenever $1\leq p \leq 2$, then:
\begin{linenomath}
\begin{equation} \label{norm_inequality_p2}
\norm{x}_2 \leq \norm{x}_p.
\end{equation}
\end{linenomath}

Conversely, if $p > 2$, then:
\begin{linenomath}
\begin{equation} \label{norm_inequality_2p}
\norm{x}_2 \leq n^{\textstyle \frac{p-2}{2p}} \norm{x}_p.
\end{equation}
\end{linenomath}

Thus, combining these results, we can write:
\begin{linenomath}
\begin{equation*}
\norm{x}_2 \leq C \norm{x}_p,
\end{equation*}
\end{linenomath}

with:
\begin{linenomath}
\begin{equation}
  C=C(n,p)= \begin{cases}
    1, & 1\leq p \leq 2\\
n^{\textstyle \frac{p-2}{2p}}, & p > 2.
  \end{cases} \label{constant_C}
\end{equation}
\end{linenomath}

\newpage
\noindent\textbf{{Proof of Proposition \ref{lemma_inhom_pol}}}

\begin{proof}
The $\mathcal{H}$-norm of the vector of perturbation $\zeta^{(i)}$ in the feature space can be expanded as:
\begin{linenomath}
	\begin{equation} \label{zeta_norm}
		\begin{split}
			\norm{\zeta^{(i)}}^2_{\mathcal{H}} &= \norm{\phi(x)-\phi(x^{(i)})}^2_\mathcal{H}\\
			&= \norm{\phi(x^{(i)}+\sigma^{(i)})-\phi(x^{(i)})}^2_\mathcal{H}\\
			&=\langle \phi(x^{(i)}+\sigma^{(i)})-\phi(x^{(i)}), \phi(x^{(i)}+\sigma^{(i)})-\phi(x^{(i)}) \rangle\\
			&= \langle \phi(x^{(i)}+\sigma^{(i)}), \phi(x^{(i)}+\sigma^{(i)})\rangle-2 \langle \phi(x^{(i)}+\sigma^{(i)}), \phi(x^{(i)}\rangle  + \langle \phi(x^{(i)}), \phi(x^{(i)}) \rangle\\
			&= k(x^{(i)}+\sigma^{(i)}, x^{(i)}+\sigma^{(i)}) -2k(x^{(i)}+\sigma^{(i)}, x^{(i)})+k(x^{(i)}, x^{(i)}).
		\end{split}
	\end{equation}
\end{linenomath}

By definition of the inhomogeneous polynomial kernel of degree $d$, the last right-hand side of \eqref{zeta_norm} becomes:
\begin{linenomath}
\begin{equation*}
\begin{split}
\norm{\zeta^{(i)}}^2_{\mathcal{H}} &= \bigg(\norm{x^{(i)}+\sigma^{(i)}}^{2}_2+c\bigg)^d -2 \big(\langle x^{(i)}+\sigma^{(i)}, x^{(i)} \rangle +c \big)^d + \bigg(\norm{x^{(i)}}^{2}_2+c\bigg)^d\\
&\!=\! \bigg(\!\norm{x^{(i)}}_2^2\!+\!\norm{\sigma^{(i)}}_2^2\!+\!2\text{ }\langle \sigma^{(i)}, x^{(i)} \rangle+c\bigg)^d \!\!-2 \bigg(\norm{x^{(i)}}_2^2+\langle \sigma^{(i)}, x^{(i)} \rangle +c \bigg)^d \!+\bigg(\norm{x^{(i)}}^{2}_2+c\bigg)^d.
\end{split}
\end{equation*}
\end{linenomath}

By applying the Cauchy-Schwarz inequality in $\mathbb{R}^n$ to the terms containing the dot product, the previous expression simplifies further, leading to:
\begin{linenomath}
\begin{equation*}
\begin{split}
\norm{\zeta^{(i)}}^2_{\mathcal{H}} &\!\!\!\leq \!\!\bigg(\!\!\norm{x^{(i)}}_2^2\!+\!\norm{\sigma^{(i)}}_2^2\!+\!2\norm{\sigma^{(i)}}_2 \norm{x^{(i)}}_2\!+c\bigg)^d\!\!\!\! -2 \bigg(\!\!\norm{x^{(i)}}_2^2+\norm{\sigma^{(i)}}_2\! \norm{x^{(i)}}_2 \!\!+c\!\bigg)^d\!\!\!\! +\!\!\bigg(\!\norm{x^{(i)}}_2^2\!\!+c\!\bigg)^d\\
& = \bigg[\bigg(\norm{x^{(i)}}_2+\norm{\sigma^{(i)}}_2\bigg)^2+c\bigg]^d \!-2 \bigg[\norm{x^{(i)}}_2\bigg(\norm{x^{(i)}}_2+\norm{\sigma^{(i)}}_2\bigg) \!+c\bigg]^d +\bigg(\!\norm{x^{(i)}}_2^2\!+c\bigg)^d.
\end{split}
\end{equation*}
\end{linenomath}

Applying the binomial expansion to three $d$-th powers implies that:
\begin{linenomath}
\begin{equation*}
\begin{split}
\norm{\zeta^{(i)}}^2_{\mathcal{H}} & \!\leq\! \sum_{k=0}^d \binom{d}{k} c^k \bigg(\norm{x^{(i)}}_2+\norm{\sigma^{(i)}}_2\bigg)^{2(d-k)}\!\!\!\!\!\! -2\sum_{k=0}^d \binom{d}{k} c^k \norm{x^{(i)}}_2^{d-k}\bigg(\norm{x^{(i)}}_2+\norm{\sigma^{(i)}}_2\bigg)^{d-k}+\\
& \quad + \sum_{k=0}^d \binom{d}{k}c^k\norm{x^{(i)}}_2^{2(d-k)}.
\end{split}
\end{equation*}
\end{linenomath}

We now split all the three sums by considering separately the cases when $k=0$, $k=d$ and, then, all the intermediate cases. Firstly, let us call $a_0$ the addendum of the sum corresponding to $k=0$. Therefore:
\begin{linenomath}
\begin{equation*}
\begin{split}
a_0 &= \bigg(\norm{x^{(i)}}_2+\norm{\sigma^{(i)}}_2\bigg)^{2d}-2 \norm{x^{(i)}}_2^d \bigg(\norm{x^{(i)}}_2+\norm{\sigma^{(i)}}_2 \bigg)^d+\norm{x^{(i)}}_2^{2d}\\
	  &= \bigg[\bigg(\norm{x^{(i)}}_2+\norm{\sigma^{(i)}}_2\bigg)^{d}-\norm{x^{(i)}}_2^{d}\bigg]^2\\
	  &= \bigg[\sum_{k=0}^d \binom{d}{k}\norm{x^{(i)}}^{d-k}_2\norm{\sigma^{(i)}}_2^k-\norm{x^{(i)}}_2^{d}\bigg]^2\\
	  &= \bigg[\sum_{k=1}^d \binom{d}{k}\norm{x^{(i)}}^{d-k}_2\norm{\sigma^{(i)}}_2^k+\norm{x^{(i)}}_2^{d}-\norm{x^{(i)}}_2^{d}\bigg]^2 = \bigg[\sum_{k=1}^d \binom{d}{k}\norm{x^{(i)}}^{d-k}_2\norm{\sigma^{(i)}}_2^k\bigg]^2.
\end{split}
\end{equation*}
\end{linenomath}

We notice that $a_0$ is the only addendum of the sum that does not contain $c$. This implies that $a_0$ is related to the bound $\delta^{(i)}_{d,0}$ for the homogeneous polynomial kernel.\\
Secondly, if $k=d$, we have no contribution because $c^d-2c^d+c^d=0.$ Before considering the cases $k=1,\ldots,d-1$, we now investigate what happens when the degree $d$ is equal to 1. Here, the index $k$ of the sums goes from 0 to 1, and therefore, as seen before:
\begin{linenomath}
\begin{equation*}
\norm{\zeta^{(i)}}^2_{\mathcal{H}} \leq \big(\delta^{(i)}_{\text{hom}}\big)^2=\big(C\eta^{(i)}\big)^2.
\end{equation*}
\end{linenomath}

Hence, when $d=1$, then $\delta^{(i)}_{1,c}=C\eta^{(i)}$. Conversely, when $d > 1$, we have all the addenda between $k=1$ and $k=d-1$. Thus, by combining all the three sums together we have:
\begin{linenomath}
\begin{equation*}
\begin{split}
\norm{\zeta^{(i)}}^2_{\mathcal{H}} & \!\!\leq\! a_0+\!\sum_{k=1}^{d-1} \!\binom{d}{k}c^k \bigg[\!\bigg(\!\norm{x^{(i)}}_2\!+\norm{\sigma^{(i)}}_2\!\bigg)^{2(d-k)}\!\!\!\!\!\!\!-2\norm{x^{(i)}}_2^{d-k}\!\!\bigg(\!\norm{x^{(i)}}_2\!+\norm{\sigma^{(i)}}_2\!\!\bigg)^{d-k}\!\!\!\!\!\!+\norm{x^{(i)}}_2^{2(d-k)}\!\!\bigg]\\
& = a_0 +\sum_{k=1}^{d-1} \binom{d}{k} c^k \bigg[\bigg(\norm{x^{(i)}}_2+\norm{\sigma^{(i)}}_2\bigg)^{d-k}\!\!-\norm{x^{(i)}}_2^{d-k}\bigg]^2.
\end{split}
\end{equation*}
\end{linenomath}

Again, by applying the binomial expansion to the $(d-k)$-th power of $\big(\norm{x^{(i)}}_2+\norm{\sigma^{(i)}}_2\big)$ and by splitting the sum, we are able to simplify the last term. Hence:
\begin{linenomath}
\begin{equation*}
\begin{split}
\norm{\zeta^{(i)}}^2_{\mathcal{H}} & \leq a_0+\sum_{k=1}^{d-1} \binom{d}{k} c^k \bigg[\sum_{j=0}^{d-k}\binom{d-k}{j} \norm{x^{(i)}}_2^{d-k-j}\norm{\sigma^{(i)}}_2^j-\norm{x^{(i)}}_2^{d-k}\bigg]^2\\
& = a_0+\sum_{k=1}^{d-1} \binom{d}{k} c^k \bigg[\sum_{j=1}^{d-k}\binom{d-k}{j} \norm{x^{(i)}}_2^{d-k-j}\norm{\sigma^{(i)}}_2^j\bigg]^2.
\end{split}
\end{equation*}
\end{linenomath}

Therefore, by taking the square root:
\begin{linenomath}
\begin{equation*}
\norm{\zeta^{(i)}}_{\mathcal{H}} \leq \sqrt{a_0+\sum_{k=1}^{d-1} \binom{d}{k} c^k \bigg[\sum_{j=1}^{d-k}\binom{d-k}{j} \norm{x^{(i)}}_2^{d-k-j}\norm{\sigma^{(i)}}_2^j\bigg]^2}.
\end{equation*}
\end{linenomath}

According to inequalities \eqref{norm_inequality_p2}$-$\eqref{norm_inequality_2p} and to hypothesis $\norm{\sigma^{(i)}}_p \leq \eta^{(i)}$, we obtain that:
\begin{linenomath}
\begin{displaymath}
\norm{\sigma^{(i)}}_2 \leq \left\{
\begin{array}{ll}
\norm{\sigma^{(i)}}_p \leq \eta^{(i)}, & 1\leq p \leq 2\\
\\
n^{\textstyle \frac{p-2}{2p}} \norm{\sigma^{(i)}}_p \leq n^{\textstyle \frac{p-2}{2p}} \eta^{(i)}, & p > 2.
\end{array} \right.
\end{displaymath}
\end{linenomath}

Finally, whenever $1\leq p \leq 2$, we have that:
\begin{linenomath}
\begin{equation*}
a_0 \leq \bigg[\sum_{k=1}^d \binom{d}{k}\norm{x^{(i)}}^{d-k}_2\norm{\sigma^{(i)}}_p^k\bigg]^2 
\leq \bigg[\sum_{k=1}^d \binom{d}{k}\norm{x^{(i)}}^{d-k}_2\big(\eta^{(i)}\big)^k\bigg]^2 = \big(\delta^{(i)}_{d,0}\big)^2,
\end{equation*}
\end{linenomath}

and the second addendum in the square root can be bounded by:
\begin{linenomath}
\begin{equation*}
\sum_{k=1}^{d-1} \binom{d}{k} c^k \bigg[\sum_{j=1}^{d-k}\binom{d-k}{j} \norm{x^{(i)}}_2^{d-k-j}\big(\eta^{(i)}\big)^j\bigg]^2.
\end{equation*}
\end{linenomath}

On the other hand, if $p>2$, then:
\begin{linenomath}
\begin{equation*}
a_0 \leq \bigg[\sum_{k=1}^d \binom{d}{k}\norm{x^{(i)}}^{d-k}_2n^{\textstyle \frac{k(p-2)}{2p}}\norm{\sigma^{(i)}}_p^k\bigg]^2  \leq \bigg[\sum_{k=1}^d \binom{d}{k}\norm{x^{(i)}}^{d-k}_2\bigg(n^{\textstyle \frac{p-2}{2p}}\eta^{(i)}\bigg)^k\bigg]^2 = \big(\delta^{(i)}_{d,0}\big)^2,
\end{equation*}
\end{linenomath}

and similarly the second addendum in the square root is always less than or equal to:
\begin{linenomath}
\begin{equation*}
\sum_{k=1}^{d-1} \binom{d}{k} c^k \bigg[\sum_{j=1}^{d-k}\binom{d-k}{j} \norm{x^{(i)}}_2^{d-k-j}\bigg(n^{\textstyle \frac{p-2}{2p}}\eta^{(i)}\bigg)^j\bigg]^2.
\end{equation*}
\end{linenomath}
\end{proof}

\noindent\textbf{{Proof of Proposition \ref{lemma_RBF_gaussian}}}

\begin{proof}
For all $x$ in $\mathbb{R}^n$, we have that $k(x,x)=1$ and, thus, equation \eqref{zeta_norm} reduces to:
\begin{linenomath}
\begin{equation*}
\norm{\zeta^{(i)}}^2_{\mathcal{H}} = 1-2\exp\bigg(\displaystyle -\frac{\norm{x^{(i)}+\sigma^{(i)}-x^{(i)}}^2_2}{2\alpha^2}\bigg)+1 = 2-2\exp\bigg(\displaystyle -\frac{\norm{\sigma^{(i)}}^2_2}{2\alpha^2}\bigg).
\end{equation*}
\end{linenomath}

Therefore:
\begin{linenomath}
\begin{equation*}
\norm{\zeta^{(i)}}_{\mathcal{H}} = \sqrt{2-2\exp\bigg(\displaystyle -\frac{\norm{\sigma^{(i)}}^2_2}{2\alpha^2}\bigg)}.
\end{equation*}
\end{linenomath}

The thesis follows by applying inequalities \eqref{norm_inequality_p2}$-$\eqref{norm_inequality_2p} and by considering the monotonicity of function $g(x)=-\exp(-x^2)$ when $x>0$.

\end{proof}

\newpage
\noindent\textbf{{Proof of Corollary \ref{coroll_LP_SOCP}}}

\begin{proof}

\begin{itemize}
\item[a)] If $q=1$, model \eqref{rob_GSVM} can be rewritten as model \eqref{robust_q1} by introducing an auxiliary vector $s \in \mathbb{R}^m$ such that each component $s_{i}$ is equal to $\abs{u_i}$ and adding the constraints $s_{i} \geq 0$, $s_{i} \geq - u_{i}$ and $s_{i} \geq u_{i}$ for all $i=1,\ldots,m$.
\item[b)] If $q=2$, the quadratic term $\norm{u}_2^2$ can be transformed from the objective function to the set of constraints by introducing auxiliary variables $r,t,v\in \mathbb{R}$ such that $t \geq \norm{u}_2$, $r+v=1$ and $r\geq \sqrt{t^2+v^2}$ (\cite{QiTianShi2013}). With the same reasoning at point a), model \eqref{rob_GSVM} reduces to model \eqref{robust_q2}.
\item[c)] If $q=\infty$, by introducing an auxiliary variable $s_\infty \geq 0$ equal to $\norm{u}_{\infty}$, and adding the constraints $s_\infty \geq -u_i$ and $s_\infty \geq u_i$ for all $i=1,\ldots,m$, model \eqref{rob_GSVM} is equivalent to model \eqref{robust_qinfty} with the same reasoning at point a).
\end{itemize}
\end{proof}

\newpage

\section{Supplementary results} \label{appendix_results}

\begin{table}[h!]
\centering
\resizebox{\textwidth}{!}{
\begin{tabular}{ll *{7}l}\toprule
Dataset & Data transformation & \multicolumn{7}{c}{Kernel}\\
& & Hom. linear & Hom. quadratic & Hom. cubic & Inhom. linear & Inhom. quadratic & Inhom. cubic & Gaussian RBF \\ \toprule
\multirow{6}{*}{Arrhythmia} & $-$ & $21.94\% \pm 0.10$ & $53.43\% \pm 0.20$ & $-$ & $21.88\% \pm 0.10$ & $53.31\% \pm 0.20$ & $-$ & $\underline{\mathbf{20.47\% \pm 0.07}}$
\\
& CPU time (s)& 0.298 & 0.297 & $-$ & 0.309 & 0.290 & $-$ & 0.289
\\
\cmidrule{2-9}
& Min-max normalization & $-$ & $-$ & $-$ & $-$ & $-$ & $-$ & $-$
\\
& CPU time (s) & $-$ & $-$ & $-$ & $-$ & $-$ & $-$ & $-$
\\
\cmidrule{2-9}
& Standardization & $-$ & $-$ & $-$ & $-$ & $-$ & $-$ & $-$
\\
& CPU time (s)& $-$ & $-$ & $-$ & $-$ & $-$ & $-$ & $-$
\\
\bottomrule
\end{tabular}
}
\caption{Detailed results of average out-of-sample testing errors and standard deviations over 96 runs of the deterministic model. Holdout: 75\% training set-25\% testing set.} \label{tab_detailed_determ_7525}
\end{table}

\begin{table}[h!]
\centering
\resizebox{\textwidth}{!}{
\begin{tabular}{ll *{7}l}\toprule
Dataset & Data transformation & \multicolumn{7}{c}{Kernel}\\
& & Hom. linear & Hom. quadratic & Hom. cubic & Inhom. linear & Inhom. quadratic & Inhom. cubic & Gaussian RBF \\ \toprule
\multirow{6}{*}{Arrhythmia} & $-$ & $23.90\% \pm 0.06$ & $54.23\% \pm 0.02$ & $-$ & $24.08\% \pm 0.05$ & $51.72\% \pm 0.21$ & $-$ & $\underline{\mathbf{23.59\% \pm 0.04}}$
\\
& CPU time (s)&  0.194 & 0.180 & $-$ & 0.191 & 0.216 & $-$ & 0.181
\\
\cmidrule{2-9}
& Min-max normalization & $-$ & $-$ & $-$ & $-$ & $-$ & $-$ & $-$
\\
& CPU time (s) & $-$ & $-$ & $-$ & $-$ & $-$ & $-$ & $-$
\\
\cmidrule{2-9}
& Standardization & $-$ & $-$ & $-$ & $-$ & $-$ & $-$ & $-$
\\
& CPU time (s)& $-$ & $-$ & $-$ & $-$ & $-$ & $-$ & $-$
\\
\bottomrule
\end{tabular}
}
\caption{Detailed results of average out-of-sample testing errors and standard deviations over 96 runs of the deterministic model. Holdout: 50\% training set-50\% testing set.} \label{tab_detailed_determ_5050}
\end{table}

\begin{table}[h!]
\centering
\resizebox{\textwidth}{!}{
\begin{tabular}{ll *{7}l}\toprule
Dataset & Data transformation & \multicolumn{7}{c}{Kernel}\\
& & Hom. linear & Hom. quadratic & Hom. cubic & Inhom. linear & Inhom. quadratic & Inhom. cubic & Gaussian RBF \\ \toprule
\multirow{6}{*}{Arrhythmia} & $-$ & $28.66\% \pm 0.07$ & $53.45\% \pm 0.20$ & $-$ & $\underline{\mathbf{27.31\% \pm 0.07}}$ & $53.84\% \pm 0.20$ & $-$ & $28.29\% \pm 0.02$ 
\\
& CPU time (s)& 0.142	&0.165	& $-$ &	0.144&	0.148& $-$ &0.153
\\
\cmidrule{2-9}
& Min-max normalization & $-$ & $-$ & $-$ & $-$ & $-$ & $-$ & $-$
\\
& CPU time (s) & $-$ & $-$ & $-$ & $-$ & $-$ & $-$ & $-$
\\
\cmidrule{2-9}
& Standardization & $-$ & $-$ & $-$ & $-$ & $-$ & $-$ & $-$
\\
& CPU time (s)& $-$ & $-$ & $-$ & $-$ & $-$ & $-$ & $-$
\\
\bottomrule
\end{tabular}
}
\caption{Detailed results of average out-of-sample testing errors and standard deviations over 96 runs of the deterministic model. Holdout: 25\% training set-75\% testing set.} \label{tab_detailed_determ_2575}
\end{table}

\clearpage
\newpage

\begin{table}[h!]
\centering
\resizebox{\textwidth}{!}{
\begin{tabular}{ll *{7}l}\toprule
{Dataset} & {Data transformation} & \multicolumn{7}{c}{{Kernel}}\\
& & {Hom. linear} & {Hom. quadratic} & {Hom. cubic} & {Inhom. linear} & {Inhom. quadratic} & {Inhom. cubic} & {Gaussian RBF} \\ \toprule
\multirow{6}{*}{{Iris}} & $-$ & {$6.42\% \pm 0.04$} &	{$3.89\%\pm 0.03$} & {$	5.83\% \pm 0.03$}	& {$6.53\% \pm 0.04$} & {$4.28\% \pm 0.04$} & {$6.39\% \pm 0.04$} & {$\underline{\mathbf{3.10\% \pm 0.03}}$}
\\
& {CPU time (s)}& {5.470} & {5.547} & {5.485} & {5.461} & {5.454} & {5.497} & {5.391}
\\
\cmidrule{2-9}
& {Min-max normalization} & {$10.92\% \pm 0.08$} & {$8.31\% \pm 0.07$} & {$6.02\% \pm 0.04$} & {$11.20\% \pm 0.08$} & {$5.74\% \pm 0.03$} & {$5.01\% \pm 0.03$} & {$\underline{4.28\% \pm 0.03}$}
\\
& {CPU time (s)} & {5.131} & {5.130} & {5.031} & {5.104} & {5.123} & {5.088} & {5.122}
\\
\cmidrule{2-9}
& {Standardization} & {$8.70\% \pm 0.05$} & {$15.20\% \pm 0.06$} & {$8.61\% \pm 0.5$} & {$8.70\% \pm 0.05$} & {$6.11\% \pm 0.04$} & {$7.74\% \pm 0.04$} & {$\underline{5.32\% \pm 0.03}$}
\\
& {CPU time (s)} & {5.476} & {5.374} & {5.307} & {5.087} & {5.172} & {5.105} & {5.121}
\\
\bottomrule
\end{tabular}
}
\caption{Detailed results of average out-of-sample testing errors and standard deviations over 96 runs of the deterministic multiclass model. Holdout: 75\% training set-25\% testing set.} \label{tab_detailed_determ_7525_multiclass}
\end{table}

\begin{table}[h!]
\centering
\resizebox{\textwidth}{!}{
\begin{tabular}{ll *{7}l}\toprule
{Dataset} & {Data transformation} & \multicolumn{7}{c}{{Kernel}}\\
& & {Hom. linear} & {Hom. quadratic} & {Hom. cubic} & {Inhom. linear} & {Inhom. quadratic} & {Inhom. cubic} & {Gaussian RBF} \\ \toprule
\multirow{6}{*}{{Iris}} & $-$ & {$6.79\% \pm 0.05$} & {$5.08\% \pm 0.02$} & {$6.33\% \pm 0.03$} & {$6.63\% \pm 0.05$} & {$4.63\% \pm 0.03$} & {$5.97\% \pm 0.03$} & {$\underline{4.58\%\pm 0.03}$}
\\
& {CPU time (s)} & {1.938} & {1.942} & {2.042} & {1.936} & {2.123} & {1.982} & {1.958}
\\
\cmidrule{2-9}
& {Min-max normalization} & {$22.04\% \pm 0.11$} & {$23.43\% \pm 0.10$} & {$9.21\% \pm 0.06$} & {	$20.13\% \pm 0.11$} & {$12.74\% \pm 0.09$} & {$5.44\% \pm 0.02$} & {$\underline{\mathbf{4.47\% \pm 0.02}}$}
\\
& {CPU time (s)} & {1.808} & {1.798} & {1.784} & {1.794} & {1.819} & {1.793} & {1.828}
\\
\cmidrule{2-9}
& {Standardization} & {$10.89\% \pm 0.06$} & {$19.56\% \pm 0.04$} & {$10.06\% \pm 0.05$} & {$9.65\% \pm 0.05$} & {$6.92\% \pm 0.02$} & {$7.60\% \pm 0.03$} & {$\underline{5.75\% \pm 0.03}$}
\\
& {CPU time (s)} & {1.913} & {1.906} & {1.841} & {1.832} & {1.797} & {1.979} & {1.951}
\\
\bottomrule
\end{tabular}
}
\caption{Detailed results of average out-of-sample testing errors and standard deviations over 96 runs of the deterministic multiclass model. Holdout: 50\% training set-50\% testing set.} \label{tab_detailed_determ_5050_multiclass}
\end{table}

\begin{table}[h!]
\centering
\resizebox{\textwidth}{!}{
\begin{tabular}{ll *{7}l}\toprule
{Dataset} & {Data transformation} & \multicolumn{7}{c}{{Kernel}}\\
& & {Hom. linear} & {Hom. quadratic} & {Hom. cubic} & {Inhom. linear} & {Inhom. quadratic} & {Inhom. cubic} & {Gaussian RBF} \\ \toprule
\multirow{6}{*}{{Iris}} & $-$ & {$11.90\% \pm 0.09$} & {$5.30\% \pm 0.03$} & {$5.51\% \pm 0.03$} & {$11.21\% \pm 0.08$} & {$\underline{\mathbf{5.05\% \pm 0.03}}$} & {$5.57\% \pm 0.03$} & {$8.28\% \pm 0.03$}
\\
& {CPU time (s)} & {0.645} & {0.635} & {0.603} & {	0.591} & {0.728} & {0.645} & {0.606}
\\
\cmidrule{2-9}
& {Min-max normalization} & {$28.25\% \pm 0.08$} & {$30.50\% \pm 0.06$} & {$27.10\% \pm 0.09$} & {$28.51\% \pm 0.08$} & {$24.73\% \pm 0.09$} & {$10.75\% \pm 0.07$} & {$\underline{7.33\% \pm 0.04}$}
\\
& {CPU time (s)} & {0.589} & {0.590} & {0.567} & {0.685} & {0.559} & {0.560} & {0.570}
\\
\cmidrule{2-9}
& {Standardization} & {$13.45\% \pm 0.08$} & {$25.29\% \pm 0.05$} & {$15.34\% \pm 0.06$} & {$13.24\% \pm 0.07$} & {$\underline{7.70\% \pm 0.03}$} & {$8.66\% \pm 0.03$} & {$8.54\% \pm 0.04$}
\\
& {CPU time (s)} & {0.590} & {0.611} & {0.628} & {0.590} & {0.698} & {0.574} & {0.581}
\\
\bottomrule
\end{tabular}
}
\caption{Detailed results of average out-of-sample testing errors and standard deviations over 96 runs of the deterministic multiclass model. Holdout: 25\% training set-75\% testing set.} \label{tab_detailed_determ_2575_multiclass}
\end{table}

\begin{table}[h!]
\centering
\resizebox{\textwidth}{!}{
\begin{tabular}{llll *{3}l}\toprule
Dataset & Data transformation & Kernel & $\rho$ & \multicolumn{3}{c}{Robust}\\
& & & & $p=1$ & $p=2$ & $p=\infty$\\ \toprule
\multirow{7}{*}{Arrhythmia} & \multirow{7}{*}{$-$} & \multirow{7}{*}{Gaussian RBF} & $10^{-7}$ & $19.18\% \pm 0.07$ & $19.42\% \pm 0.07$ & $19.67\% \pm 0.07$
\\
& & & $10^{-6}$ & $19.30\% \pm 0.06$ & $19.61\% \pm 0.07$ & $20.40\% \pm 0.08$
\\
& & & $10^{-5}$ & $20.47\% \pm 0.07$ & $19.91\% \pm 0.07$ & $19.73\% \pm 0.06$
\\
& & & $10^{-4}$ & $20.10\% \pm 0.07$ & $19.55\% \pm 0.07$ & $\underline{19.61\% \pm 0.07}$
\\
& & & $10^{-3}$ & $	\underline{\mathbf{19.12\% \pm 0.08}}$ & $\underline{19.30\% \pm 0.07}$ & $23.28\% \pm 0.06$
\\
& & & $10^{-2}$ & $19.30\% \pm 0.07$ & $20.83\% \pm 0.08$ & $29.41\% \pm 0.00$
\\
\bottomrule
\end{tabular}
}
\caption{Average out-of-sample testing errors and standard deviations over 96 runs of the robust model. Holdout: 75\% training set-25\% testing set.} \label{tab_detailed_rob_7525}
\end{table}
\begin{table}[h!]
\centering
\resizebox{\textwidth}{!}{
\begin{tabular}{llll *{3}l}\toprule
Dataset & Data transformation & Kernel & $\rho$ & \multicolumn{3}{c}{Robust}\\
& & & & $p=1$ & $p=2$ & $p=\infty$\\ \toprule
\multirow{7}{*}{Breast Cancer Diagnostic} & \multirow{7}{*}{Min-max normalization} & \multirow{7}{*}{Inhom. quadratic} & $10^{-7}$ & $2.80\% \pm 0.01$ & $\underline{2.65\% \pm 0.01}$ & $2.80\% \pm 0.01$
\\
& & & $10^{-6}$ & $2.96\% \pm 0.02$ & $2.70\% \pm 0.01$ & $2.96\% \pm 0.02$
\\
& & & $10^{-5}$ & $\underline{2.63\% \pm 0.01}$ & $2.99\% \pm 0.01$ & $2.66\% \pm 0.01$
\\
& & & $10^{-4}$ & $2.88\% \pm 0.01$ & $2.88\% \pm 0.01$ & $\underline{\mathbf{{2.39\% \pm 0.01}}}$
\\
& & & $10^{-3}$ & $2.91\% \pm 0.01$ & $3.19\% \pm 0.01$ & $9.76\% \pm 0.03$
\\
& & & $10^{-2}$ & $37.32\% \pm 0.00$ & $37.32\% \pm 0.00$ & $37.32\% \pm 0.00$
\\
\bottomrule
\end{tabular}
}
\caption{Average out-of-sample testing errors and standard deviations over 96 runs of the robust model. Holdout: 75\% training set-25\% testing set (continued).} \label{tab_detailed_rob_7525_cont}
\end{table}

\begin{table}[h!]
\centering
\resizebox{\textwidth}{!}{
\begin{tabular}{llll *{3}l}\toprule
Dataset & Data transformation & Kernel & $\rho$ & \multicolumn{3}{c}{Robust}\\
& & & & $p=1$ & $p=2$ & $p=\infty$\\ \toprule
\multirow{7}{*}{Arrhythmia} & \multirow{7}{*}{$-$} & \multirow{7}{*}{Gaussian RBF} & $10^{-7}$ & $24.66\% \pm 0.04$ & $23.99\% \pm 0.05$ & $\underline{\mathbf{23.44\% \pm 0.04}}$
\\
& & & $10^{-6}$ & $24.11\% \pm 0.05$ & $24.63\% \pm 0.04$ & $23.77\% \pm 0.05$
\\
& & & $10^{-5}$ & $23.81\% \pm 0.05$ & $24.08\% \pm 0.04$ & $24.11\% \pm 0.05$
\\
& & & $10^{-4}$ & $\underline{23.77\% \pm 0.05}$ & $\underline{23.74\% \pm 0.05}$ & $24.33\% \pm 0.05$
\\
& & & $10^{-3}$ & $24.51\% \pm 0.04$ & $24.60\% \pm 0.05$ & $26.10\% \pm 0.04$
\\
& & & $10^{-2}$ & $24.36\% \pm 0.04$ & $23.77\% \pm 0.05$ & $29.41\% \pm 0.00$
\\
\bottomrule
\end{tabular}
}
\caption{Average out-of-sample testing errors and standard deviations over 96 runs of the robust model. Holdout: 50\% training set-50\% testing set.} \label{tab_detailed_rob_5050}
\end{table}
\begin{table}[h!]
\centering
\resizebox{\textwidth}{!}{
\begin{tabular}{llll *{3}l}\toprule
Dataset & Data transformation & Kernel & $\rho$ & \multicolumn{3}{c}{Robust}\\
& & & & $p=1$ & $p=2$ & $p=\infty$\\ \toprule
\multirow{7}{*}{Breast Cancer Diagnostic} & \multirow{7}{*}{Min-max normalization} & \multirow{7}{*}{Inhom. quadratic} & $10^{-7}$  & $3.06\% \pm 0.01$ & $3.19\% \pm 0.01$ & $3.06\% \pm 0.01$
\\
& & & $10^{-6}$ & $3.19\% \pm 0.01$ & $3.19\% \pm 0.01$ & $3.18\% \pm 0.01$
\\
& & & $10^{-5}$ & $\underline{2.87\% \pm 0.01}$ & $3.17\% \pm 0.01$ & $\underline{\mathbf{2.86\% \pm 0.01}}$
\\
& & & $10^{-4}$ & $3.26\% \pm 0.01$ & $\underline{2.99\% \pm 0.01}$ & $3.21\% \pm 0.01$
\\
& & & $10^{-3}$ & $2.90\% \pm 0.01$ & $3.29\% \pm 0.01$ & $5.67\% \pm 0.01$
\\
& & & $10^{-2}$ & $11.14\% \pm 0.03$ & $10.74\% \pm 0.03$ & $37.32\% \pm 0.00$
\\
\bottomrule
\end{tabular}
}
\caption{Average out-of-sample testing errors and standard deviations over 96 runs of the robust model. Holdout: 50\% training set-50\% testing set (continued).} \label{tab_detailed_rob_5050_cont}
\end{table}

\begin{table}[h!]
\centering
\resizebox{\textwidth}{!}{
\begin{tabular}{llll *{3}l}\toprule
Dataset & Data transformation & Kernel & $\rho$ & \multicolumn{3}{c}{Robust}\\
& & & & $p=1$ & $p=2$ & $p=\infty$\\ \toprule
\multirow{7}{*}{Arrhythmia} & \multirow{7}{*}{$-$} & \multirow{7}{*}{Inhom. linear} & $10^{-7}$ & $\underline{\mathbf{26.70\% \pm 0.07}}$  & $27.47\% \pm 0.06$  & $28.82\% \pm 0.06$
\\
& & & $10^{-6}$ & $28.00\% \pm 0.06$  & $\underline{27.21\% \pm 0.06}$  & $28.15\% \pm 0.07$ 
\\
& & & $10^{-5}$ & $28.66\% \pm 0.06$  & $28.29\% \pm 0.06$  & $28.10\% \pm 0.06$
\\
& & & $10^{-4}$ & $28.78\% \pm 0.06$  & $28.04\% \pm 0.06$  & $\underline{27.84\% \pm 0.07}$
\\
& & & $10^{-3}$ & $27.41\% \pm 0.06$  & $28.68\% \pm 0.07$ & $32.29\% \pm 0.06$
\\
& & & $10^{-2}$ & $31.56\% \pm 0.07$  & $30.43\% \pm 0.07$  & $29.41\% \pm 0.00$
\\
\bottomrule
\end{tabular}
}
\caption{Average out-of-sample testing errors and standard deviations over 96 runs of the robust model. Holdout: 25\% training set-75\% testing set.} \label{tab_detailed_rob_2575}
\end{table}
\begin{table}[h!]
\centering
\resizebox{\textwidth}{!}{
\begin{tabular}{llll *{3}l}\toprule
Dataset & Data transformation & Kernel & $\rho$ & \multicolumn{3}{c}{Robust}\\
& & & & $p=1$ & $p=2$ & $p=\infty$\\ \toprule
\multirow{7}{*}{Breast Cancer Diagnostic} & \multirow{7}{*}{Standardization} & \multirow{7}{*}{Inhom. linear} & $10^{-7}$ & $4.78\% \pm 0.01$ & $4.81\% \pm 0.01$ & $4.60\% \pm 0.01$
\\
& & & $10^{-6}$ & $4.84\% \pm 0.01$ & $4.74\% \pm 0.02$ & $4.94\% \pm 0.01$
\\
& & & $10^{-5}$ & $4.65\% \pm 0.01$ & $4.85\% \pm 0.01$ & $4.76\% \pm 0.01$
\\
& & & $10^{-4}$ & $4.86\% \pm 0.01$ & $4.86\% \pm 0.01$ & $4.82\% \pm 0.01$
\\
& & & $10^{-3}$ & $4.89\% \pm 0.01$ & $4.76\% \pm 0.01$ & $4.79\% \pm 0.01$
\\
& & & $10^{-2}$ & $4.22\% \pm 0.01$ & $4.72\% \pm 0.02$ & $\underline{3.91\% \pm 0.01}$
\\
\bottomrule
\end{tabular}
}
\caption{Average out-of-sample testing errors and standard deviations over 96 runs of the robust model. Holdout: 25\% training set-75\% testing set (continued).} \label{tab_detailed_rob_2575_cont}
\end{table}

\clearpage
\newpage

\begin{table}[h!]
\centering
\resizebox{0.88\textwidth}{!}{
\begin{tabular}{llll *{3}l}\toprule
{Dataset} & {Data transformation} & {Kernel} & {$\rho$} & \multicolumn{3}{c}{{Robust}}\\
& & & & {$p=1$} & {$p=2$} & {$p=\infty$}\\ \toprule
\multirow{7}{*}{{Iris}} & \multirow{7}{*}{{$-$}} & \multirow{7}{*}{{Gaussian RBF}} & {$10^{-7}$} & {$4.00\% \pm 0.03$} & {$3.86\% \pm 0.03$} & {$3.43\% \pm 0.03$}
\\
& & & {$10^{-6}$} & {$3.55\% \pm 0.03$} & {$3.43\% \pm 0.03$} & {$\underline{\mathbf{2.87\% \pm 0.03}}$}
\\
& & & {$10^{-5}$} & {$3.46\% \pm 0.03$} & {$3.43\% \pm 0.03$} & {$3.27\% \pm 0.03$}
\\
& & & {$10^{-4}$} & {$3.15\% \pm 0.03$} & {$\underline{3.21\% \pm 0.03}$} & {$3.01\% \pm 0.03$}
\\
& & & {$10^{-3}$} & {$\underline{3.07\% \pm 0.03}$} & {$3.63\% \pm 0.03$} & {$3.52\% \pm 0.03$}
\\
& & & {$10^{-2}$} & {$3.21\% \pm 0.03$} & {$3.46\% \pm 0.03$} & {$3.27\% \pm 0.03$}
\\
& & & {$10^{-1}$} & {$5.83\% \pm 0.04$} & {$6.31\% \pm 0.04$} & {$10.87\% \pm 0.08$}
\\
\cmidrule{4-7}
&&& {CPU time (s)} & {5.684} & {5.627} & {5.604}
\\
\toprule
\multirow{7}{*}{{Wine}} & \multirow{7}{*}{{Standardization}} & \multirow{7}{*}{{Inhom. linear}} & {$10^{-7}$} & {$3.48\% \pm 0.02$} & {$\underline{2.63\% \pm 0.02}$} & {$2.89\% \pm 0.02$}
\\
& & & {$10^{-6}$} & {$2.91\% \pm 0.02$} & {$3.13\% \pm 0.03$} & {$3.15\% \pm 0.02$}
\\
& & & {$10^{-5}$} & {$3.29\% \pm 0.02$} & {$3.46\% \pm 0.03$} & {$3.15\% \pm 0.03$}
\\
& & & {$10^{-4}$} & {$2.79\% \pm 0.02$} & {$2.63\% \pm 0.03$} & {$2.89\% \pm 0.02$}
\\
& & & {$10^{-3}$} & {$2.89\% \pm 0.02$} & {$2.94\% \pm 0.02$} & {$2.91\% \pm 0.03$}
\\
& & & {$10^{-2}$} & {$2.70\% \pm 0.02$} & {$3.10\% \pm 0.02$} & {$\underline{\mathbf{2.51\% \pm 0.02}}$}
\\
& & & {$10^{-1}$} & {$\underline{2.63\% \pm 0.02}$} & {$2.65\% \pm 0.02$} & {$5.71\% \pm 0.04$}
\\
\cmidrule{4-7}
&&& {CPU time (s)} & {8.361} & {8.352} & {8.605}
\\
\bottomrule
\end{tabular}
}
\caption{Average out-of-sample testing errors and standard deviations over 96 runs of the robust multiclass model. Holdout: 75\% training set-25\% testing set.} \label{tab_detailed_rob_7525_multiclass}
\end{table}

\begin{table}[h!]
\centering
\resizebox{0.88\textwidth}{!}{
\begin{tabular}{llll *{3}l}\toprule
{Dataset} & {Data transformation} & {Kernel} & {$\rho$} & \multicolumn{3}{c}{{Robust}}\\
& & & & {$p=1$} & {$p=2$} & {$p=\infty$}\\ \toprule
\multirow{7}{*}{{Iris}} & \multirow{7}{*}{{Min-max normalization}} & \multirow{7}{*}{{Gaussian RBF}} & {$10^{-7}$} & {$4.76\% \pm 0.02$} & {$4.97\% \pm 0.02$} & {$4.99\% \pm 0.02$}
\\
& & & {$10^{-6}$} & {$4.92\% \pm 0.02$} & {$\underline{4.54\% \pm 0.02}$} & {$5.04\% \pm 0.02$}
\\
& & & {$10^{-5}$} & {$4.78\% \pm 0.02$} & {$4.63\% \pm 0.02$} & {$4.92\% \pm 0.02$}
\\
& & & {$10^{-4}$} & {$4.67\% \pm 0.02$} & {$5.21\% \pm 0.03$} & {$4.42\% \pm 0.02$}
\\
& & & {$10^{-3}$} & {$\underline{4.58\% \pm 0.02}$} & {$4.67\% \pm 0.02$} & {$\underline{\mathbf{4.35\% \pm 0.02}}$}
\\
& & & {$10^{-2}$} & {$4.69\% \pm 0.02$} & {$4.63\% \pm 0.02$} & {$4.90 \% \pm 0.02$}
\\
& & & {$10^{-1}$} & {$6.00\% \pm 0.03$} & {$6.10\% \pm 0.03$} & {$15.49\% \pm 0.10$}
\\
\cmidrule{4-7}
&&& {CPU time (s)} & {1.863} & {1.870} & {1.875}
\\
\toprule
\multirow{7}{*}{{Wine}} & \multirow{7}{*}{{Standardization}} & \multirow{7}{*}{{Hom. linear}} & {$10^{-7}$} & {$3.43\% \pm 0.02$} & {$3.68\% \pm 0.02$} & {$3.48\% \pm 0.02$}
\\
& & & {$10^{-6}$} & {$3.51\% \pm 0.02$} & {$3.73\% \pm 0.02$} & {$3.41\% \pm 0.02$}
\\
& & & {$10^{-5}$} & {$3.52\% \pm 0.02$} & {$3.79\% \pm 0.02$} & {$3.46\% \pm 0.02$}
\\
& & & {$10^{-4}$} & {$3.48\% \pm 0.02$} & {$3.57\% \pm 0.02$} & {$\underline{3.32\% \pm 0.02}$}
\\
& & & {$10^{-3}$} & {$3.69\% \pm 0.02$} & {$3.79\% \pm 0.02$} & {$3.45\% \pm 0.02$}
\\
& & & {$10^{-2}$} & {$3.28\% \pm 0.02$} & {$3.45\% \pm 0.02$} & {$3.53\% \pm 0.02$}
\\
& & & {$10^{-1}$} & {$\underline{\mathbf{3.25\% \pm 0.02}}$} & {$\underline{3.28\% \pm 0.02}$} & {$5.14\% \pm 0.03$}
\\
\cmidrule{4-7}
&&& {CPU time (s)} & {3.311} & {3.089} & {3.148}
\\
\bottomrule
\end{tabular}
}
\caption{Average out-of-sample testing errors and standard deviations over 96 runs of the robust multiclass model. Holdout: 50\% training set-50\% testing set.} \label{tab_detailed_rob_5050_multiclass}
\end{table}

\begin{table}[h!]
\centering
\resizebox{0.88\textwidth}{!}{
\begin{tabular}{llll *{3}l}\toprule
{Dataset} & {Data transformation} & {Kernel} & $\rho$ & \multicolumn{3}{c}{{Robust}}\\
& & & & {$p=1$} & {$p=2$} & {$p=\infty$}\\ \toprule
\multirow{7}{*}{{Iris}} & \multirow{7}{*}{{$-$}} & \multirow{7}{*}{{Inhom. quadratic}} & {$10^{-7}$} & {$\underline{4.94\% \pm 0.02}$} & {$5.11\% \pm 0.03$} & {$5.06\% \pm 0.03$}
\\
& & & {$10^{-6}$} & {$5.23\% \pm 0.03$} & {$5.70\% \pm 0.03$} & {$5.09\% \pm 0.02$}
\\
& & & {$10^{-5}$} & {$5.22\% \pm 0.03$} & {$\underline{4.97\% \pm 0.03}$} & {$5.06\% \pm 0.02$}
\\
& & & {$10^{-4}$} & {$5.48\% \pm 0.03$} & {$5.05\% \pm 0.02$} & {$\underline{\mathbf{4.86\% \pm 0.02}}$}
\\
& & & {$10^{-3}$} & {$5.12\% \pm 0.02$} & {$5.59\% \pm 0.02$} & {$5.51\% \pm 0.03$}
\\
& & & {$10^{-2}$} & {$10.37\% \pm 0.07$} & {$10.70\% \pm 0.07$} & {$13.22\% \pm 0.08$}
\\
& & & {$10^{-1}$} & {$26.24\% \pm 0.07$} & {$25.19\% \pm 0.07$} & {$31.27\% \pm 0.14$}
\\
\cmidrule{4-7}
&&& CPU time (s) & {0.596} & {0.588} & {0.593}
\\
\toprule
\multirow{7}{*}{{Wine}} & \multirow{7}{*}{{Standardization}} & \multirow{7}{*}{{Gaussian RBF}} & {$10^{-7}$} & {$5.21\% \pm 0.07$} & {$4.06\% \pm 0.03$} & {$4.56\% \pm 0.06$}
\\
& & & {$10^{-6}$} & {$4.91\% \pm 0.07$} & {$3.96\% \pm 0.03$} & {$4.07\% \pm 0.02$}
\\
& & & {$10^{-5}$} & {$4.14\% \pm 0.04$} & {$3.79\% \pm 0.03$} & {$4.57\% \pm 0.04$}
\\
& & & {$10^{-4}$} & {$3.45\% \pm 0.02$} & {$3.71\% \pm 0.02$} & {$4.18\% \pm 0.02$}
\\
& & & {$10^{-3}$} & {$6.23\% \pm 0.06$} & {$5.49\% \pm 0.06$} & {$6.43\% \pm 0.03$}
\\
& & & {$10^{-2}$} & {$11.80\% \pm 0.10$} & {$12.65\% \pm 0.11$} & {$20.37\% \pm 0.12$}
\\
& & & {$10^{-1}$} & {$\underline{1.34\% \pm 0.05}$} & {$\underline{1.24\% \pm 0.05} $} & {$\underline{\mathbf{0.00\% \pm 0.00}}$}
\\
\cmidrule{4-7}
&&& {CPU time (s)} & {0.718} & {0.717} & {0.713}
\\
\bottomrule
\end{tabular}
}
\caption{Average out-of-sample testing errors and standard deviations over 96 runs of the robust multiclass model. Holdout: 25\% training set-75\% testing set.} \label{tab_detailed_rob_2575_multiclass}
\end{table}

\clearpage
\newpage

\begin{table}[h!]
\centering
\resizebox{\textwidth}{!}{
\begin{tabular}{l| *{3}l | *{3}l |*{3}l}\toprule
{Dataset} & \multicolumn{3}{c|}{{75\% training set-25\% testing set}} & \multicolumn{3}{c|}{{50\% training set-50\% testing set}} & \multicolumn{3}{c}{{25\% training set-75\% testing set}}\\
& {Deterministic} & {Best robust} & {Improvement ratio} & {Deterministic} & {Best robust} & {Improvement ratio} & {Deterministic} & {Best robust} & {Improvement ratio}
\\
\toprule
{Arrhythmia}  & {$20.47\%$}	& {$19.12\%$}	 & {$6.60\%$} & {$23.59\%$} & {$23.44\%$} & {$0.64\%$} & {$27.31\%$} & {$26.70\%$} & {$2.23\%$}
\\
{Parkinson} & {$13.19\%$} &	{$12.37\%$} & {$6.22\%$} & {$15.10\%$} & {$13.84\%$} & {$8.34\%$} & {$18.49\%$} & {$16.96\%$} & {$8.27\%$}
\\
{Heart Disease} & {$17.48\%$} & {$16.36\%$} & {$6.41\%$} & {$18.92\%$} & {$17.29\%$} & {$8.62\%$} & $21.85\%$ & {$19.47\%$} & {$10.89\%$}
\\
{Dermatology} & {$1.64\%$} & {$0.55\%$} & {$66.46\%$} & {$1.96\%$} & {$0.51\%$} & {$73.98\%$} & {$3.01\%$} & {$2.04\%$} & {$32.23\%$}
\\
{Climate Model Crashes} & {$5.01\%$} & {$4.34\%$} & {$13.37\%$} & {$5.35\%$} & {$5.21\%$} & {$2.62\%$} & {$7.15\%$} & {$6.98\%$} & {$2.38\%$}
\\
{Breast Cancer Diagnostic} & {$3.02\%$} & {$2.39\%$} & {$20.86\%$} & {$3.69\%$} & {$2.86\%$} & {$22.49\%$} & {$5.02\%$} & {$3.68\%$} & {$26.69\%$}
\\
\bottomrule
\end{tabular}
}
\caption{Improvement ratios of the robust models over the deterministic counterparts for the three holdouts.} \label{tab_impr_ratio}
\end{table}

\begin{table}[h!]
\centering
\resizebox{\textwidth}{!}{
\begin{tabular}{ll l l l}\toprule
Dataset & Data transformation & & Mean value of features & CV of features\\
\toprule
\multirow{2}{*}{Arrhythmia} & \multirow{2}{*}{$-$} & Min & $2.23\times 10^2$ & $0$
\\
& & Max & $6.20\times 10^2$ & $5.29\times 10^{-1}$
\\
\cmidrule{1-5}
\multirow{2}{*}{Parkinson} & \multirow{2}{*}{Min-max normalization} & Min & $4.40\times 10^{-5}$ & $7.71 \times 10^{-2}$
\\
& & Max & $1.97\times 10^{2}$ & $1.63\times 10^{0}$
\\
\cmidrule{1-5}
\multirow{2}{*}{Heart Disease} & \multirow{2}{*}{Standardization} & Min & $1.45\times 10^{-1}$ & $1.35 \times 10^{-1}$
\\
& & Max & $2.47\times 10^{2}$ & $2.43\times 10^{0}$
\\
\cmidrule{1-5}
\multirow{2}{*}{Dermatology} & \multirow{2}{*}{$-$} & Min & $1.06\times 10^{-1}$ & $3.20 \times 10^{-1}$
\\
\bottomrule
\end{tabular}
}
\caption{Minimum and maximum values for the mean and the Coefficient of Variation (CV) computed feature-wise. The data transformation refers to the best choice when classifying the holdout 75\%-25\% with the deterministic model.} \label{tab_data_transformation}
\end{table}

\end{document}